\documentclass[12pt]{article}
\usepackage{amssymb,amsmath, amsthm,latexsym, verbatim, amscd}
\usepackage{here}
  \newtheorem{theorem}{Theorem}[section] %
  \newtheorem{proposition}[theorem]{Proposition} %
  \newtheorem{lemma}[theorem]{Lemma} %
  
  \newtheorem{corollary}[theorem]{Corollary} %

  \newtheorem{definition-theorem}[theorem]{Definition-Theorem}
\theoremstyle{definition} %
  \newtheorem{definition}[theorem]{Definition} %
  \newtheorem{example}[theorem]{Example} %
  \newtheorem{prob}[theorem]{Problem}

  \newtheorem{remark}[theorem]{Remark} %
\def \Piinfty{\Pi}
\def \piinfty{\pi}

\def\rarrowsim{\smash{\mathop{\,\rightarrow\,}\limits
  ^{\lower1.5pt\hbox{$\scriptstyle\sim$}}}}

\newcommand{\invHom}[3]{\operatorname{Hom}_{#1}(#2,#3)}

\numberwithin{equation}{section}
\numberwithin{equation}{section}
\numberwithin{table}{section}

\begin{document}
\title{Bounded multiplicity theorems for induction and restriction} 
\author{Toshiyuki KOBAYASHI}
\date{} %
\maketitle %

\begin{abstract}
We prove a geometric criterion
 for the bounded multiplicity property
 of \lq\lq{small}\rq\rq\
 infinite-dimensional representations
 of real reductive Lie groups
 in both induction and restrictions.

Applying the criterion to symmetric pairs, 
 we give a full description of the triples $H \subset G \supset G'$
 such that any irreducible admissible representations of $G$
 with $H$-distinguished vectors
 have the bounded multiplicity property 
 when restricted to the subgroup $G'$.  
This article also completes the proof of the general results
 announced in the previous paper
 [Adv.~Math. 2021, Section 7].  
\end{abstract}


\section{Introduction}
In \cite{Ksuron, xktoshima}
we initiated a new line of investigation
 on the finiteness or the boundedness of {\it{multiplicities}}
 in induction and restriction, 
 and proposed a new avenue of research
 by clarifying a \lq\lq{nice framework}\rq\rq\
 for both global analysis 
 and branching problems
 with \lq\lq{firm grip}\rq\rq\ of group representations.  
This article gives its refinement
 by focusing
 on a family of \lq\lq{small}\rq\rq\ infinite-dimensional representations
 such as irreducible representations of $G$
 having $H$-distinguished vectors
 for reductive symmetric pairs $(G,H)$.

Let $G$ be a real reductive algebraic Lie group 
 with Lie algebra ${\mathfrak{g}}$.  
We assume $G$ is contained in a connected complex Lie group
 $G_{\mathbb{C}}$
 with Lie algebra ${\mathfrak{g}}_{\mathbb{C}}={\mathfrak{g}} \otimes_{\mathbb{R}}{\mathbb{C}}$, 
 though this assumption is easily relaxed.  
Let ${\mathcal{M}}(G)$ be the category
 of finitely generated, smooth admissible representations of $G$
 of moderate growth, 
 sometimes referred to as the {\it{Casselman--Wallach globalization}} 
 \cite[Chap.~11]{WaI}.  
Denote by $\operatorname{Irr}(G)$
 the set of irreducible objects
 in ${\mathcal{M}}(G)$, 
 and by $\operatorname{Irr}(G)_f$
 that of irreducible finite-dimensional ones.

We shall use the uppercase letter $\Pi$
 for representations of the group $G$, 
 and the lowercase letter $\pi$
 for those of a subgroup.

Suppose $G'$ is a reductive subgroup of $G$.  
For $\Piinfty \in \operatorname{Irr}(G)$
 and $\piinfty \in \operatorname{Irr}(G')$, 
 we define the {\it{multiplicity}} 
 of the restriction $\Pi|_{G'}$ 
 in the category ${\mathcal{M}}$ by 
\begin{equation}
\label{eqn:JLTp2}
[{\Pi}|_{G'}: \pi]
:= \dim_{\mathbb{C}} \invHom {G'}{\Piinfty|_{G'}}{\piinfty}
\in {\mathbb{N}} \cup \{\infty\}, 
\end{equation}
where 
$\operatorname{Hom}_{G'}(\,,\,)$ denotes the space 
 of continuous $G'$-homomorphisms
 between the Fr{\'e}chet representations.

In \cite[Thms.~C and D]{xktoshima}
 we established the following geometric criteria:

{\bf{Bounded multiplicity}} for a pair $(G,G')$: \enspace
\begin{equation}
\label{eqn:BB}
  \underset{\Pi \in \operatorname{Irr}(G)}\sup\,\,
  \underset{\pi \in \operatorname{Irr}(G')}\sup\,\,
  [\Pi|_{G'}:\pi]<\infty
\end{equation}
if and only if $(G_{\mathbb{C}} \times G_{\mathbb{C}}')/\operatorname{diag} G_{\mathbb{C}}'$ is spherical.

{\bf{Finite multiplicity}} for a pair $(G,G')$:\enspace
\begin{equation}
\label{eqn:PP}
  [\Pi|_{G'}:\pi]<\infty
\qquad
\text{for all $\Pi \in \operatorname{Irr}(G)$
 and $\pi \in \operatorname{Irr}(G')$}  
\end{equation}
if and only if $(G \times G')/\operatorname{diag} G'$ is real spherical.

Here we recall 
 that a complex $G_{\mathbb{C}}$-manifold $X$
 is called {\it{spherical}}
 if a Borel subgroup of $G_{\mathbb{C}}$
 has an open orbit in $X$, 
 and that a $G$-manifold $Y$ is called
 {\it{real spherical}}
 if a minimal parabolic subgroup of $G$
 has an open orbit in $Y$.

A remarkable feature of the above criterion is
 that the bounded multiplicity property \eqref{eqn:BB} is determined
 only by the pair of the complexified Lie algebras $({\mathfrak{g}}_{\mathbb{C}}, {\mathfrak{g}}_{\mathbb{C}}')$, 
 hence the classification
 of the pairs $(G,G')$ satisfying \eqref{eqn:BB}
 is reduced to a classical one \cite{xkramer}:
the pair $({\mathfrak{g}}, {\mathfrak{g}}')$ is 
 any real form 
 in the direct sum of the following pairs 
up to abelian ideals:

\begin{equation}
\label{eqn:BBlist}
({\mathfrak{g}}_{\mathbb{C}}, {\mathfrak{g}}_{\mathbb{C}}')
=
({\mathfrak{sl}}_n, {\mathfrak{gl}}_{n-1}), 
({\mathfrak{so}}_{n}, {\mathfrak{so}}_{n-1}),
\text{ or } 
({\mathfrak{so}}_8, {\mathfrak{spin}}_7). 
\end{equation}

When \eqref{eqn:BBlist} holds, 
 the supremum in \eqref{eqn:BB} equals one
 for many of the real forms
 such as $(SO(p,q), SO(p-1,q))$ or 
 $(SU(p,q), U(p-1,q))$
 \cite{xsunzhu}.

On the other hand, 
 the finite multiplicity property \eqref{eqn:PP}
 depends on real forms.  
It is fulfilled 
 for a Riemannian symmetric pair
 by Harish-Chandra's admissibility theorem, 
 whereas  it is not the case
 for some reductive symmetric pairs
 such as $(G,G')=(S L(p+q,{\mathbb{R}}), S O(p,q))$.  
A complete classification of the irreducible symmetric pairs
 $(G,G')$ satisfying 
 the finite multiplicity property \eqref{eqn:PP} was accomplished
 in \cite{xKMt}
 based on the above geometric criterion.

To go beyond these cases, 
 we observe
 that even when the pair $(G,G')$ does not satisfy
 the bounded multiplicity property \eqref{eqn:BB}
 or more broadly, 
 the finiteness property \eqref{eqn:PP}, 
 there may still exist a specific $\Pi \in \operatorname{Irr}(G)$
 for which a detailed study 
 of the restriction $\Pi|_{G'}$ 
 will be reasonable.  
Such $\Pi$ should be a \lq\lq{small}\rq\rq\ representation
 relative to the subgroup $G'$
 in some sense.  
This observation suggests 
 to look at the triple $(\Pi, G, G')$
 rather than a pair $(G,G')$ of groups.  
This formulation has been successful 
 in the study of $G'$-{\it{admissible restriction}}
 of $\Pi$, 
 namely, 
 the restriction $\Pi|_{G'}$ of $\Pi \in \operatorname{Irr}(G)$
 being discretely decomposable with finite multiplicity, 
 see \cite{K94Invent, xkAnn98, K98b, K19b}
 for the general theory, 
 and \cite{KO12, KO15} for some classification theory
 of the triples $(\Pi,G,G')$.

In this article, 
 we allow the case 
 where the restriction $\Pi|_{G'}$ is not 
 \lq\lq{discretely decomposable}\rq\rq, 
 and highlight the {\it{bounded multiplicity property}}.  
For this purpose, 
 we consider for $\Pi \in \operatorname{Irr}(G)$
 the following quantity:
\begin{equation}
\label{eqn:msup}
m(\Pi|_{G'}):=\underset{\pi \in \operatorname{Irr}(G')}\sup [\Pi|_{G'}:\pi]
 \in {\mathbb{N}} \cup \{\infty\}.  
\end{equation}

We address the following:
\begin{prob}
\label{q:Bdd}
Given a pair $(G,G')$, 
 find a subset $\Omega \equiv \Omega(G')$
 of $\operatorname{Irr}(G)$
 (or of ${\mathcal{M}}(G)$)
 such that 
$
  \underset{\Pi \in \Omega} \sup \,\, m(\Pi|_{G'})<\infty.  
$
\end{prob}

We note that Problem \ref{q:Bdd} is nontrivial 
 even when $G$ is a compact Lie group
 where $m(\Pi|_{G'})$ is individually finite, 
 see Example \ref{ex:SU3}
 for the case $(G, G')=(SU(3), SO(3))$.  
We begin with an observation of two opposite extremal choices of $\Omega$:
 a singleton {\it{vs}} the whole set $\operatorname{Irr}(G)$.    
When $\Omega$ is a singleton, 
 Problem \ref{q:Bdd} concerns the triple $(\Pi, G, G')$
 for which $\Pi \in \operatorname{Irr}(G)$ satisfies
 the bounded multiplicity property
 $m(\Pi|_{G'})<\infty$, 
 see \cite[Probl.~6.2 (2)]{xKVogan2015}.  
When $\Omega = \operatorname{Irr}(G)$, 
Problem \ref{q:Bdd} is nothing but the bounded multiplicity property
 \eqref{eqn:BB}
 for the pair $(G,G')$, 
 and the aforementioned geometric criterion was proved
 in \cite[Thm.~D]{xktoshima}.  
We are particularly interested in the intermediate case
 $\Omega=\operatorname{Irr}(G)_H$, 
  the infinite set of $H$-distinguished 
 irreducible representations of $G$.  
We also discuss Problem \ref{q:Bdd}
 when $\Omega$ is a subset
 of degenerate principal series representations, 
 see Theorem \ref{thm:introQsph} below.

Let us fix some notation.  
For $\Pi \in \operatorname{Irr}(G)$, 
 we denote by $\Pi^{-\infty}$ 
 the representation on the space
 of distribution vectors, 
 that is, 
 the topological dual of $\Pi$.  
For a closed subgroup $H$ of $G$, 
 we set
\begin{equation}
\label{eqn:IrrGH}
\operatorname{Irr}(G)_H
:=\{\Pi \in \operatorname{Irr}(G)
    :(\Pi^{-\infty})^H \ne \{0\}\}.  
\end{equation}

Let $\Pi^{\vee}$ be the contragredient representation
 in the category ${\mathcal{M}}(G)$.  
Then, 
 one has
 $\Pi \in \operatorname{Irr}(G)_H$
 if and only if $\invHom G {\Pi^{\vee}}{C^{\infty}(G/H)} \ne \{0\}$
 by the Frobenius reciprocity.  
Elements $\Pi$ in $\operatorname{Irr}(G)_H$
 (or $\Pi^{\vee}$)
 are sometimes referred to as {\it{$H$-distinguished}}, 
 or having {\it{nonzero $H$-periods}}.  
As a concrete setting of Problem \ref{q:Bdd}, 
 we study the following problem 
 when $(G,H)$ is a reductive symmetric pair.  
In this case, 
 all elements in $\operatorname{Irr}(G)_H$
 are quite \lq\lq{small}\rq\rq\
 representations in general, 
 see  {\it{e.g.}}, Proposition \ref{prop:GKdim}
 for an estimate of the Gelfand--Kirillov dimension.  

\begin{prob}
\label{q:bdd}
Find a criterion for the triple $H \subset G \supset G'$
 with bounded multiplicity property
 for the restriction:
\begin{equation}
\label{eqn:BBH}
\underset{\Pi \in \operatorname{Irr}(G)_H} \sup\,\,
m(\Pi|_{G'}) < \infty.  
\end{equation}
\end{prob}

In \cite[Thm.~7.6]{K21}
 we have given a geometric answer to Problem \ref{q:bdd}, 
 see Theorem \ref{thm:bdd} below, 
 together with some motivations, 
examples, and perspectives, 
but have postponed the detailed proof until this article.  
We also prove a full classification of the triples $(G,H,G')$
 satisfying the bounded multiplicity property \eqref{eqn:BBH}
 in the setting 
 where $(G,G')$ is a symmetric pair.

We recall that $(G,H)$ is a {\it{symmetric pair}}
 defined by an involution $\sigma$ of $G$, 
 if $H$ is an open subgroup
 of $G^{\sigma}=\{g \in G: \sigma g = g\}$.  
The same letter $\sigma$ will be used 
 to denote its differential.  
We take a maximal semisimple abelian subspace ${\mathfrak{j}}$
 in ${\mathfrak{g}}^{-\sigma}=\{X \in {\mathfrak{g}}:\sigma X=-X\}$.  
The dimension of ${\mathfrak{j}}$ is independent
 of the choice of such a subspace, 
 and is called the {\it{rank}} of the symmetric space $G/H$.  
We introduce the following terminology:

\begin{definition}
[Borel subalgebra for $G/H$]
\label{def:Borel}
A {\it{Borel subalgebra}}
 for the symmetric space $G/H$ is the complex parabolic subalgebra
 ${\mathfrak{q}}$
 of ${\mathfrak{g}}_{\mathbb{C}}$
 associated to a positive system $\Sigma^+({\mathfrak{g}}_{\mathbb{C}},{\mathfrak{j}}_{\mathbb{C}})$.  
We say the corresponding complex parabolic subgroup
 $Q$ $(\subset G_{\mathbb{C}})$
 is a {\it{Borel subgroup}} for the symmetric space $G/H$.  
\end{definition}

Borel subalgebras for the symmetric space $G/H$ are unique 
 up to inner automorphisms of ${\mathfrak{g}}_{\mathbb{C}}$.  
We sometimes write ${\mathfrak{b}}_{G/H}$ for ${\mathfrak{q}}$, 
 and $B_{G/H}$ for $Q$.  
We note that a Borel subalgebra ${\mathfrak{b}}_{G/H}$
 for $G/H$
 is determined 
 only by the complexification $({\mathfrak{g}}_{\mathbb{C}}, {\mathfrak{h}}_{\mathbb{C}})$, 
 and that it is not necessarily solvable.  
If ${\mathfrak{b}}_{G/H}$ is solvable 
 then the regular representation on $L^2(G/H)$ is tempered
 \cite[Thm.~1.1]{BK21}.

We prove the following.  
\begin{theorem}
\label{thm:bdd}
Let $B_{G/H}$ be a Borel subgroup for a reductive symmetric space $G/H$.  
Suppose $G'$ is an algebraic reductive subgroup of $G$, 
 and $G_U'$ is a compact real form of $G_{\mathbb{C}}'$.  
Then the following three conditions on the triple $(G,H,G')$
 are equivalent:
\begin{enumerate}
\item[{\rm{(i)}}]
$\underset{\Pi \in \operatorname{Irr}(G)_H}\sup m(\Pi|_{G'}) <\infty$.  

\item[{\rm{(ii)}}]
$G_{\mathbb{C}}/{B_{G/H}}$ is $G_{\mathbb{C}}'$-spherical.  
\item[{\rm{(iii)}}]
$G_{\mathbb{C}}/{B_{G/H}}$ is $G_U'$-strongly visible
 (\cite[Def.~3.3.1]{xrims40}).  
\end{enumerate}
\end{theorem}

A special case of Theorem \ref{thm:bdd} includes
 the tensor product case.  
For $\Pi_1, \Pi_2 \in \operatorname{Irr}(G)$, 
 we set 
\begin{equation}
\label{eqn:JLTp5}
m(\Pi_1 \otimes \Pi_2):=
\underset{\Pi \in \operatorname{Irr}(G)}\sup
\dim_{\mathbb{C}} \invHom G{\Pi_1 \otimes \Pi_2}{\Pi}
\in {\mathbb{N}} \cup \{\infty\}.  
\end{equation}

\begin{theorem}
[Tensor product]
\label{thm:tensor}
Suppose that $(G,H_j)$ are reductive symmetric pairs, 
 and that $B_{G/H_{j}}$ are Borel subgroups
 for $G/H_j$ for $j=1,2$.  
Then the following three conditions on the triple 
 $(G,H_1, H_2)$ are equivalent:
\begin{enumerate}
\item[{\rm{(i)}}]
\begin{equation}
\label{eqn:bddt}
  \underset{\Pi_1 \in \operatorname{Irr}(G)_{H_1}}\sup\,\,
  \underset{\Pi_2 \in \operatorname{Irr}(G)_{H_2}}\sup\,\,
  m(\Pi_1 \otimes \Pi_2) <\infty.  
\end{equation}
\item[{\rm{(ii)}}]
$(G_{\mathbb{C}} \times G_{\mathbb{C}})/(B_{G/H_{1}} \times B_{G/H_{2}})$
 is $G_{\mathbb{C}}$-spherical via the diagonal action.  
\item[{\rm{(iii)}}]
$(G_{\mathbb{C}} \times G_{\mathbb{C}})/(B_{G/H_{1}} \times B_{G/H_{2}})$
 is $G_U$-strongly visible via the diagonal action.  
\end{enumerate}
\end{theorem}

The classification theory for spherical varieties
 {\it{e.g.}}, \cite{xhnoo}, 
 or alternatively that for strongly visible actions
 {\it{e.g.}}, \cite{xtanaka12}
leads us to the classification
 of the triples $(G,H,G')$
 for Theorem \ref{thm:bdd}
 and the triples $(G,H_1, H_2)$ for Theorem \ref{thm:tensor}.  
See Theorems \ref{thm:listreal}, \ref{thm:cpxlist}, \ref{thm:gplist}, 
 and \ref{thm:tensorlist}
 for a full description.

Although \lq\lq{smallness}\rq\rq\
 of the representation $\Pi \in \operatorname{Irr}(G)$
 should be necessary in some sense
 for the boundedness property
 $m(\Pi|_{G'})<\infty$
 of the restriction $\Pi|_{G'}$, 
 invariants such as the associated variety
 are not informative enough
 for Problem \ref{q:bdd}, 
 as one may notice that a delicate example
 already shows up in the compact setting, 
 see Example \ref{ex:SU3}.  
To overcome this difficulty, 
 a key idea of our proof is 
 to use \lq\lq{$QP$ estimates}\rq\rq\
 which implement a pair
 of parabolic subalgebras ${\mathfrak{q}} \subset {\mathfrak{p}}_{\mathbb{C}}$ 
 dealing with the induction from $P$ to $G$, 
 where ${\mathfrak{q}}$ is not necessarily defined over ${\mathbb{R}}$.  
For a finite-dimensional irreducible $P$-module $\eta$, 
 we define $d_{\mathfrak{q}}(\eta)$
 to be the minimum of the dimensions 
 of non-zero ${\mathfrak{q}}$-submodules in $\eta$, 
 and denote by $\operatorname{Irr}(P;{\mathfrak{q}})_f$
 the subset of $\operatorname{Irr}(P)_f$
 with $d_{\mathfrak{q}}(\eta)=1$.

We deduce Theorem \ref{thm:bdd} from the following two results:
Theorem \ref{thm:introQsph} below gives
 \lq\lq{$QP$ estimates for restriction}\rq\rq\
 and Theorem \ref{thm:introquotient} is 
 a generalization of Harish-Chandra's subquotient theorem
 and Casselman's subrepresentation theorem
 for $H$-distinguished representations of $G$.

Let $\Omega_P:=\{\operatorname{Ind}_P^G (\xi):\text{$\xi$ is a character of $P$}\}$.  
We set 
\begin{equation}
\label{eqn:OmegaPQ}
  (\Omega_P \subset)\,\, \Omega_{P,{\mathfrak{q}}}
:=\{\operatorname{Ind}_P^G(\xi)
  :\xi \in \operatorname{Irr}(P;{\mathfrak{q}})_f\}
\quad
(\subset {\mathcal{M}}(G)).
\end{equation}
\begin{theorem}
[see Thereom \ref{thm:Qsph}]
\label{thm:introQsph}
Let $G \supset G'$ be a pair of real reductive algebraic Lie groups, 
 $P$ a parabolic subgroup of $G$, 
 and $Q$ a complex subgroup of $G_{\mathbb{C}}$
 such that ${\mathfrak{q}} \subset {\mathfrak{p}}_{\mathbb{C}}$.  
Then one has the equivalence:
\begin{equation}
\label{eqn:JLTthm16}
\text{
$G_{\mathbb{C}}/Q$ is $G_{\mathbb{C}}'$-spherical}
\iff
\text{
$
  \underset{\Pi \in \Omega_{P,{\mathfrak{q}}}}\sup\,\,
  m(\Pi|_{G'}) <\infty.  
$}
\end{equation}
\end{theorem}

A special case of Theorem \ref{thm:introQsph} with $Q=P_{\mathbb{C}}$ shows:
\begin{example}
[see Example \ref{ex:Qsph2}]
\label{ex:introQsph2}
One has the equivalence from \eqref{eqn:JLTthm16}:
\begin{equation}
\label{eqn:JLTex17}
  \text{$G_{\mathbb{C}}/P_{\mathbb{C}}$ is $G_{\mathbb{C}}'$-spherical}
\iff
\underset{\Pi \in \Omega_P}\sup m(\Pi|_{G'})< \infty.  
\end{equation}
\end{example}

Let $P_{G/H}$ be a \lq\lq{minimal parabolic subgroup}\rq\rq\
 for the symmetric space $G/H$
 (Definition \ref{def:Psigma}), 
 and ${\mathfrak{b}}_{G/H}$ a Borel subalgebra for $G/H$
 with ${\mathfrak{b}}_{G/H} \subset ({\mathfrak{p}}_{G/H})_{\mathbb{C}}$.  
\begin{theorem}
[see Theorem \ref{thm:quotient}]
\label{thm:introquotient}
Any $\Pi \in \operatorname{Irr}(G)_H$
 is obtained as the quotient of the induced representation 
 $\operatorname{Ind}_{P_{G/H}}^G(\xi)$ 
 for some $\xi \in \operatorname{Irr}(P_{G/H};{\mathfrak{b}}_{G/H})$.  
\end{theorem}

Along the same line as in \cite{xktoshima, xkProg2014}, 
 the \lq\lq{$QP$ estimate}\rq\rq\
 for {\it{restriction}}
 ({\it{e.g.,}}
 the implication (ii) $\Rightarrow$ (i) in Theorem \ref{thm:bdd})
 is derived from the following 
 \lq\lq{$QP$ estimates for {\it{induction}}}\rq\rq\
 applied to $(G \times G')/\operatorname{diag} G'$.  
Theorem \ref{thm:introind} is a generalization
 of some results in \cite{xktoshima}
 relying on the theory 
 of \lq\lq{boundary valued maps}\rq\rq\
 and in Tauchi \cite{xtauchi}
 relying on the theory
 of holonomic ${\mathcal{D}}$-modules
 \cite{Ka83, KaKw81}.

\begin{theorem}
[see Theorem \ref{thm:indbdd} (1)]
\label{thm:introind}
Let $P$ be a parabolic subgroup of a real reductive algebraic Lie group
 $G$, 
 and  $H$ an algebraic subgroup.  
Suppose that $Q$ is a parabolic subgroup of $P_{\mathbb{C}}$
 such that $\#(Q \backslash G_{\mathbb{C}}/H_{\mathbb{C}})<\infty$.  
Then one has 
\[
  \underset{\eta \in \operatorname{Irr}(P)_f}\sup\,\,
  \underset{\tau \in \operatorname{Irr}(H)_f}\sup
  \frac{1}{d_{\mathfrak{q}}(\eta)\dim \tau}
  \dim_{\mathbb{C}} 
  \invHom G {\operatorname{Ind}_P^G(\eta)}{\operatorname{Ind}_H^G(\tau)} 
  <\infty.  
\]
\end{theorem}

Individual finite multiplicity results
 for induction and restriction 
 can be read from the \lq\lq{$QP$ estimates}\rq\rq\
 for induction and restriction, 
 respectively, 
 by putting $Q:=P_{\mathbb{C}}$, 
 see Remark \ref{rem:33} and Example \ref{ex:46}
 for instance.  

\vskip 1pc
\par\noindent
{\bf{Organization of the paper}}\enspace
\newline
In Section \ref{sec:IrrPq}
 we introduce the set $\operatorname{Irr}(P;{\mathfrak{q}})_f$
 and discuss some basic properties
 of finite-dimensional representations.  
Bounded multiplicity theorems 
 for induction and restriction 
 of degenerate principal series representations
 with \lq\lq{$QP$ estimates}\rq\rq\
 are given in Sections \ref{sec:indbdd} and \ref{sec:restbdd}, 
 respectively.  
Section \ref{sec:Quotient} is devoted to a refinement
 of Casselman's subrepresentation theorem
 for $H$-distinguished representations
 (Theorem \ref{thm:introquotient}).  
With these preparations, 
 our main results (Theorems \ref{thm:bdd} and \ref{thm:tensor})
 will be proved
 in Section \ref{sec:5}.  
The classification of the triples $(G,H,G')$
 with the bounded multiplicity property \eqref{eqn:BBH} is given
 in Section \ref{sec:classification}, 
 and is proved in Section \ref{sec:pfclass}
 based on the geometric criteria
 in Theorems \ref{thm:bdd} and \ref{thm:tensor}.

\section{Preliminaries on $\operatorname{Irr}(P;{\mathfrak{q}})_f$}
\label{sec:IrrPq}

We prepare some finer properties
 of finite-dimensional representations
 that we shall need in the \lq\lq{$QP$ estimate}\rq\rq, 
 the uniform estimate
 of multiplicities 
 for a family of representations
 in induction and restriction.

\subsection{Definition of $\operatorname{Irr}(P;{\mathfrak{q}})_f$}
\label{subsec:IrrPq}

In this subsection
 we examine finite-dimensional irreducible representations
 of a Lie group $P$
 with respect to a parabolic subalgebra ${\mathfrak{q}}$.

Let $P$ be a real algebraic group
 or its open subgroup in the classical topology.  
We write $P=L N$ for its Levi decomposition, 
 and ${\mathfrak{p}}={\mathfrak{l}}+{\mathfrak{n}}$
 for the corresponding decomposition of the Lie algebras.  
We denote by $\operatorname{Irr}(P)_f$
 the set of equivalence classes
 of irreducible finite-dimensional representations of $P$, 
 and by $\operatorname{Irr}({\mathfrak{p}})_f$
 that of the Lie algebra ${\mathfrak{p}}$.  
If $P$ is connected, 
 one may regard $\operatorname{Irr}(P)_f$
 as a subset of $\operatorname{Irr}({\mathfrak{p}})_f$.  
Since the unipotent subgroup $N$ acts trivially
 on any irreducible finite-dimensional representation of $P$, 
 one has a natural bijection 
$
\operatorname{Irr}(P)_f
\simeq 
\operatorname{Ind}(L)_f
$
via the quotient map $P \to L \simeq P/N$.  

\begin{definition}
\label{def:FPq}
Let ${\mathfrak{p}}_{\mathbb{C}}={\mathfrak{l}}_{\mathbb{C}}+{\mathfrak{n}}_{\mathbb{C}}
$ be the complexified Lie algebra of 
$
 {\mathfrak{p}}={\mathfrak{l}}+{\mathfrak{n}}
$, 
 and ${\mathfrak{q}}$ a parabolic subalgebra
 in 
$
 {\mathfrak{p}}_{\mathbb{C}}
$, 
namely, 
 ${\mathfrak{q}}$ is the full inverse 
 of a parabolic subalgebra of ${\mathfrak{l}}_{\mathbb{C}}$
 via the quotient map 
 $\varpi \colon {\mathfrak{p}}_{\mathbb{C}} \to 
 {\mathfrak{p}}_{\mathbb{C}}/{\mathfrak{n}}_{\mathbb{C}}
 \simeq {\mathfrak{l}}_{\mathbb{C}}$.  
For $(\xi,V) \in \operatorname{Irr}(P)_f$, 
we define $d_{\mathfrak{q}}(\xi)$
 to be the minimal dimension
 of an irreducible ${\mathfrak{q}}$-submodule in $V$.  
We set 
\begin{align}
\label{eqn:JLTIrrPq}
\operatorname{Irr}(P;{\mathfrak{q}})_f
:=&\{\xi \in \operatorname{Irr}(P)_f:
    d_{\mathfrak{q}}(\xi)=1
\}, 
\\
\label{eqn:JLTIrrpq}
\operatorname{Irr}({\mathfrak{p}};{\mathfrak{q}})_f
:=&\{\xi \in \operatorname{Irr}({\mathfrak{p}})_f:
    d_{\mathfrak{q}}(\xi)=1
\}.  
\end{align}

\end{definition}

We say $v$ is a {\it{relatively ${\mathfrak{q}}$-invariant}} vector, 
 if there is a complex linear form $\lambda$ on ${\mathfrak{q}}$
 such that 
$
  \xi(X) v = \lambda (X) v
$
for all 
$X \in {\mathfrak{q}}$.  
By definition $d_{\mathfrak{q}}(\xi)=1$
 if and only if there is a non-zero
 relatively ${\mathfrak{q}}$-invariant vector.  
Let ${\mathfrak{u}}$ be the nilpotent radical of ${\mathfrak{q}}$.  
If $(\xi, V) \in \operatorname{Irr}({\mathfrak{p}})$,
 then 
\[
V^{\mathfrak{u}}:=\{v \in V
      :
      \xi(X)v=0\quad
     {}^{\forall}X \in {\mathfrak{u}}\} 
\]
 is the unique irreducible ${\mathfrak{q}}$-submodule of $V$, 
 and $d_{\mathfrak{q}}(\xi)=\dim V^{\mathfrak{u}}$.

Unlike the notation 
$
   \mathfrak{p}_{\mathbb{C}}=\mathfrak{p} \otimes_{\mathbb{R}}{\mathbb{C}}, 
$
 we do not use the letter $\mathfrak{q}_{\mathbb{C}}$
 in Definition \ref{def:FPq}
 to denote the (complex) parabolic subalgebra $\mathfrak{q}$, 
 because the subalgebra ${\mathfrak{q}}$ is not necessarily defined
 over ${\mathbb{R}}$.

We shall formulate Theorems \ref{thm:indbdd} and \ref{thm:restQQ}
 (\lq\lq{$QP$ estimate}\rq\rq)
 and Theorem \ref{thm:quotient}
 (quotient representation theorem 
 for $H$-distinguished representations)
 by using the pair $P$ and ${\mathfrak{q}}$, 
 which is a key 
 in proving Theorem \ref{thm:bdd}
 through a unified treatment
 both for the real polarization
 (the usual parabolic induction)
 when $\mathfrak{p} \cap \mathfrak{q}$ 
 is a real form of $\mathfrak{q}$
 and for the complex polarization
 (the Borel--Weil type induction)
 when $(\mathfrak{p} \cap \mathfrak{q})/\mathfrak{n}$ is reductive.

We collect some basic properties
 on $\operatorname{Irr}(P;{\mathfrak{q}})_f$.  
The proof is straightforward from the definition.  
\begin{lemma}
\label{lem:IrrPq}
Let ${\mathfrak{q}}$ and ${\mathfrak{q}}'$ be parabolic subalgebras
 of ${\mathfrak{p}}_{\mathbb{C}}$. 
One has 
\begin{alignat}{2}
\operatorname{Irr}({\mathfrak{p}};{\mathfrak{p}}_{\mathbb{C}})_f
&=\{\text{characters of ${\mathfrak{p}}$}\}, 
\notag
\\
\operatorname{Irr}(P;{\mathfrak{q}})_f
&\subset
\operatorname{Irr}({\mathfrak{p}};{\mathfrak{q}})_f
\quad
&&\text{if $P$ is connected};
\notag
\\
\operatorname{Irr}(P;{\mathfrak{q}})_f
&\simeq
\operatorname{Irr}(L;{\mathfrak{q}} \cap {\mathfrak{l}}_{\mathbb{C}})_f, 
\label{eqn:IrrPL}
\\
\label{eqn:AdIPq}
\operatorname{Irr}(P;{\mathfrak{q}})_f
&=
\operatorname{Irr}(P;\operatorname{Ad}(g){\mathfrak{q}})_f
\qquad
&&\text{for any $g \in P_{\mathbb{C}}$};
\\
   \operatorname{Irr}(P;{\mathfrak{q}})_f
   &\supset
   \operatorname{Irr}(P;{\mathfrak{q}}')_f
\qquad
&&\text{if ${\mathfrak{q}}\subset {\mathfrak{q}}'$};
\notag
\\
   \operatorname{Irr}(P;{\mathfrak{b}})_f
   &=
   \operatorname{Irr}(P)_f 
\quad
&&\text{if ${\mathfrak{b}}$ is a Borel subalgebra of ${\mathfrak{p}}_{\mathbb{C}}$}.  
\notag
\end{alignat}
\end{lemma}

It is convenient to prepare notation
 for (finite-dimensional) holomorphic representations
 of a complex Lie group.

\begin{definition}
\label{def:IrrPhol}
For a connected complex Lie group $P_{\mathbb{C}}$, 
 we denote by $\operatorname{Irr}(P_{\mathbb{C}})_{\operatorname{hol}}$
 the set of equivalence classes
 of finite-dimensional irreducible holomorphic representations
 of $P_{\mathbb{C}}$.  
If $P$ is a real form of $P_{\mathbb{C}}$, 
 we have a natural inclusion 
 $\operatorname{Irr}(P_{\mathbb{C}})_{\operatorname{hol}} \hookrightarrow \operatorname{Irr}(P)_f$
 by restriction.  
Accordingly, 
 we set 
$\operatorname{Irr}(P_{\mathbb{C}};{\mathfrak{q}})_{\operatorname{hol}}
:=\operatorname{Irr}(P;{\mathfrak{q}})_f
 \cap \operatorname{Irr}(P_{\mathbb{C}})_{\operatorname{hol}}$.  
\end{definition}

The concept of opposite parabolic subalgebras
 in reductive Lie algebras
 is naturally extended to 
 the non-reductive case:

\begin{definition}
\label{def:Qopp}
Let $\varpi \colon {\mathfrak{p}}_{\mathbb{C}} \to {\mathfrak{p}}_{\mathbb{C}}/{\mathfrak{n}}_{\mathbb{C}} \simeq {\mathfrak{l}}_{\mathbb{C}}$
 be the projection as before, 
 and ${\mathfrak{q}}$ a parabolic subalgebra of ${\mathfrak{p}}_{\mathbb{C}}$.  We say ${\mathfrak{q}}_{\operatorname{opp}}$ is the opposite parabolic subalgebra of ${\mathfrak{q}}$
 in ${\mathfrak{p}}_{\mathbb{C}}$
 if ${\mathfrak{q}}_{\operatorname{opp}}$ is the full inverse 
of the opposite parabolic subalgebra
 of $\varpi({\mathfrak{q}})$
 in ${\mathfrak{l}}_{\mathbb{C}}$
 (with respect to a fixed Cartan subalgebra).  
\end{definition}

We denote by $\xi^{\vee}$
 the contragredient representation of $\xi$.  
Then one has the following.  

\begin{lemma}
\label{lem:xidual}
\begin{enumerate}
\item[{\rm{(1)}}]
$\xi \in \operatorname{Irr}(P;{\mathfrak{q}})_f$
 if and only if 
$\xi^{\vee} \in \operatorname{Irr}(P;{\mathfrak{q}}_{\operatorname{opp}})_f$.  
\item[{\rm{(2)}}]
$\xi \in \operatorname{Irr}({\mathfrak{p}};{\mathfrak{q}})_f$
 if and only if 
$\xi^{\vee} \in \operatorname{Irr}({\mathfrak{p}};{\mathfrak{q}}_{\operatorname{opp}})_f$.  
\end{enumerate}
\end{lemma}

By \eqref{eqn:IrrPL}, 
 it suffices to prove Lemma \ref{lem:xidual}
 in the reductive case, 
 which is shown in Lemma \ref{lem:IrrGQ} (2) below.

%
\subsection{Description of $\operatorname{Irr}(P;{\mathfrak{q}})_f$}
\label{subsec:ptheta}
%
As we saw in \eqref{eqn:IrrPL}, 
 the description of $\operatorname{Irr}(P;{\mathfrak{q}})_f$
 reduces to the case
 where $P$ is a reductive group, 
 for which we use the letter $G$ in this subsection
 for later purpose.  
Let $\widetilde{\mathfrak{j}}$ be a Cartan subalgebra
 of ${\mathfrak{g}}$, 
 $W$ the Weyl group
 of the root system
 $\Delta({\mathfrak{g}}_{\mathbb{C}}, \widetilde{\mathfrak{j}}_{\mathbb{C}})$, 
 and $w_0$ the longest element in $W$.  
We fix a $W$-invariant non-degenerate symmetric bilinear form $\langle\,,\,\rangle$ 
 on the dual space $\widetilde{\mathfrak{j}_{\mathbb{C}}}^{\ast}$.  
We take a positive system 
 $\Delta^+({\mathfrak{g}}_{\mathbb{C}},
  \widetilde{\mathfrak{j}_{\mathbb{C}}})$, 
 denote by $\Psi$ the set of simple roots, 
 and write $\Lambda_+ \equiv \Lambda_+({\mathfrak{g}}_{\mathbb{C}})$
 for the set of 
 dominant integral weights of $\widetilde{{\mathfrak{j}}_{\mathbb{C}}}$.  
Then the Cartan--Weyl highest weight theory establishes the bijection:
\begin{equation}
\label{eqn:CW}
\operatorname{Irr}({\mathfrak{g}})_f \simeq \Lambda_+, 
\qquad
 \Pi_{\lambda} \leftrightarrow \lambda.  
\end{equation}

If $\Pi_{\lambda}$ lifts to a holomorphic representation
 of  a connected complex reductive Lie group 
 $G_{\mathbb{C}}$
 with Lie algebra ${\mathfrak{g}}_{\mathbb{C}}$, 
 we use the same letter $\Pi_{\lambda}$
 to denote the lift.

Given a subset $\Theta$ in $\Psi$, 
 we write 
$
   {\mathfrak{g}}_{\mathbb{C}}
   =
   {\mathfrak{n}}_-^{\Theta}
   \oplus
   {\mathfrak{l}}_{\mathbb{C}}^{\Theta}
   \oplus
   {\mathfrak{n}}_+^{\Theta}
$
 for the Gelfand--Naimark decomposition, 
 where ${\mathfrak{l}}_{\mathbb{C}}^{\Theta}$ is a reductive subalgebra
 containing $\widetilde{\mathfrak{j}_{\mathbb{C}}}$
 with $\Delta({\mathfrak{l}}_{\mathbb{C}}^{\Theta}, \widetilde{\mathfrak{j}}_{\mathbb{C}})
=\Delta({\mathfrak{g}}_{\mathbb{C}}, \widetilde{\mathfrak{j}}_{\mathbb{C}})
 \cap {\mathbb{Z}}\text{-span} \Theta$, 
$
   {\mathfrak{p}}^{\Theta}\equiv{\mathfrak{p}}_+^{\Theta} :={\mathfrak{l}}_{\mathbb{C}}^{\Theta} \oplus {\mathfrak{n}}_+^{\Theta}$
 is a parabolic subalgebra
 with $\Delta({\mathfrak{n}}_{+}^{\Theta}, \widetilde{\mathfrak{j}}_{\mathbb{C}}) \subset \Delta^{+}({\mathfrak{g}}, \widetilde{\mathfrak{j}}_{\mathbb{C}}) $ 
 and 
$
   {\mathfrak{p}}_-^{\Theta}
   :=
   {\mathfrak{l}}_{\mathbb{C}}^{\Theta}
   \oplus
   {\mathfrak{n}}_-^{\Theta}
$
 is the opposite parabolic subalgebra.

\begin{lemma}
\label{lem:IrrGQ}
Suppose $\Theta$  is a subset of $\Psi$.  
\begin{enumerate}
\item[{\rm{(1)}}]
The map \eqref{eqn:CW} of taking highest weights
 induces the following bijection: 
\begin{equation}
\label{eqn:IrrGQ}
   \operatorname{Irr}({\mathfrak{g}};{\mathfrak{p}}^{\Theta})_f
   \simeq
    \{\lambda \in \Lambda_+
      :
      \langle \lambda, \alpha \rangle =0
      \quad
      {}^{\forall} \alpha \in \Theta
    \}.  
\end{equation}
\item[{\rm{(2)}}]
$\xi \in 
\operatorname{Irr}({\mathfrak{g}};{\mathfrak{p}}^{\Theta})_f
$
 if and only if 
$\xi^{\vee} \in 
\operatorname{Irr}({\mathfrak{g}};{\mathfrak{p}}_-^{\Theta})_f
$.  
\end{enumerate}
\end{lemma}

\begin{proof}
(1)\enspace
We  set
$
\widetilde {\mathfrak{j}}_{\mathbb{C}}^{\Theta}
:=
[{\mathfrak{l}}_{\mathbb{C}}^{\Theta}, {\mathfrak{l}}_{\mathbb{C}}^{\Theta}] \cap \widetilde {\mathfrak{j}}_{\mathbb{C}}, 
$
which is a Cartan subalgebra
 of the semisimple part 
$
   [{\mathfrak{l}}_{\mathbb{C}}^{\Theta},{\mathfrak{l}}_{\mathbb{C}}^{\Theta}]
$
 of ${\mathfrak{l}}_{\mathbb{C}}^{\Theta}$.  
Then the right-hand side of \eqref{eqn:IrrGQ} equals
 $\{\lambda \in \Lambda_+:\text{$\lambda$ vanishes on $\widetilde {\mathfrak{j}}_{\mathbb{C}}^{\Theta}$}\}$.

Suppose $v$ is a relatively ${\mathfrak{p}}^{\Theta}$-invariant vector
 of  
$
   \xi \in \operatorname{Irr}({\mathfrak{g}};{\mathfrak{p}}^{\Theta})_f
$, 
 namely, 
 $v$ satisfies $\xi(X)v=\lambda(X)v$
 (${}^{\forall} X \in {\mathfrak{p}}^{\Theta}$) 
 for some  linear form $\lambda$ on ${\mathfrak{p}}^{\Theta}$.  
Then  $\lambda$ vanishes on $[{\mathfrak{p}}^{\Theta},{\mathfrak{p}}^{\Theta}]$, 
hence, 
 $\lambda|_{\widetilde{\mathfrak{j}}_{\mathbb{C}}}$ is the highest weight of $\xi$
 and $\lambda$ vanishes on $\widetilde{\mathfrak{j}}_{\mathbb{C}}^{\Theta}$.  
Conversely,
 if $\lambda \in \Lambda_+$ vanishes on $\widetilde {\mathfrak{j}}_{\mathbb{C}}^{\Theta}$, 
 then $\lambda$ extends to a character of ${\mathfrak{p}}^{\Theta}$
 via the quotient map 
$
   {\mathfrak{p}}^{\Theta} \to {\mathfrak{p}}^{\Theta}/[{\mathfrak{p}}^{\Theta}, {\mathfrak{p}}^{\Theta}]
  \simeq
 \widetilde {\mathfrak{j}}_{\mathbb{C}}/\widetilde {\mathfrak{j}}_{\mathbb{C}}^{\Theta}$, 
 and the highest vector $v$ of $\Pi_{\lambda}$ satisfies $\xi(X) v = \lambda(X) v$
 ${}^{\forall} X \in {\mathfrak{p}}^{\Theta}$.  
Hence $\lambda
 \in \operatorname{Irr}({\mathfrak{g}};{\mathfrak{p}}^{\Theta})_f$.  
\par\noindent
(2)\enspace
Since $-w_0 \lambda$ is the highest weight
 of the contragredient representation $\xi^{\vee}$, 
one has
 $\xi^{\vee} \in 
\operatorname{Irr}({\mathfrak{g}};{\mathfrak{p}}^{-w_0 \Theta})_f$
 if $\xi \in 
\operatorname{Irr}({\mathfrak{g}};{\mathfrak{p}}^{\Theta})_f
$, 
 and vice versa.  
Since ${\mathfrak{p}}^{-w_0 \Theta}$ is conjugate to ${\mathfrak{p}}_-^{\Theta}$
 by an inner automorphism of ${\mathfrak{g}}_{\mathbb{C}}$, 
 one has $\operatorname{Irr}({\mathfrak{g}};{\mathfrak{p}}^{-w_0 \Theta})_f
=\operatorname{Irr}({\mathfrak{g}};{\mathfrak{p}}_-^{\Theta})_f
$, 
 whence the assertion follows.  
\end{proof}

\subsection{Geometric realization
 for $\operatorname{Irr}(P;{\mathfrak{q}})_f$}
\label{subsec:BWIrrPq}

Suppose we are in the setting of Definition \ref{def:FPq}.  
In this subsection,
 we provide a geometric interpretation
 of $\operatorname{Irr}(P; {\mathfrak{q}})_f$.

We let the Lie algebra ${\mathfrak{p}}$ of $P$ act on $C^{\infty}(P)$
 as left invariant vector fields
 by $(d R(X) f)(g):=\frac{d}{d t}|_{t=0}f(g \exp t X)$, 
 and the same letter $d R$ is used
 to denote its complex linear extension to ${\mathfrak{p}}_{\mathbb{C}}$.  
For a ${\mathfrak{q}}$-module $(\tau, W)$, 
 we let ${\mathfrak{q}}$ act on $C^{\infty}(P) \otimes W$ 
 by $d R \otimes \operatorname{id} + \operatorname{id} \otimes \tau$, 
 which may be written simply as $d R \otimes \tau$.

\begin{lemma}
\label{lem:geomIrrPq}
Suppose that $(\tau, W)$ is a quotient
 of $(\xi,V) \in \operatorname{Irr}(P)_f$
 as a ${\mathfrak{q}}$-module.  
Then the left translation of $P$ leaves
 $(C^{\infty}(P) \otimes W)^{\mathfrak{q}}$ invariant, 
 and there is a natural injective $P$-homomorphism
 $\widetilde T \colon V \to (C^{\infty}(P)\otimes W)^{\mathfrak{q}}$.  
\end{lemma}

\begin{proof}
Denote by $(\xi^{\vee}, V^{\vee})$ the contragredient representation
 of $\xi$, 
 and consider a bilinear map
\[
   T \colon V \times V^{\vee} \to C^{\infty}(P), 
\quad
   (v,u) \mapsto (g \mapsto \langle \xi(g^{-1})v, u\rangle).  
\]
Taking the dual of the quotient map $V \to W$, 
 one has an injective ${\mathfrak{q}}$-homomorphism
 $W^{\vee} \hookrightarrow V^{\vee}$, 
 hence the restriction 
 $T|_{V \times W^{\vee}}$ induces a $P$-homomorphism 
 $\widetilde T \colon V \to \invHom{\mathbb{C}}{W^{\vee}}{C^{\infty}(P)}
 \simeq C^{\infty}(P) \otimes W$
 by $\widetilde T(v)(g):=T(v, \cdot\,)(g)$ for $v\in V$ and $g \in P$.  
Then 
$
   \widetilde T(v) \in \operatorname{Hom}_{\mathfrak{q}}(W^{\vee}, C^{\infty}(P))
   \simeq
    (C^{\infty}(P) \otimes W)^{{\mathfrak{q}}}
$
 because 
$
   d R(X) T(v,u)=T(v,d \xi^{\vee}(X)u)
   =
   T(v, \tau^{\vee} (X)u)
$
 for any $u \in W^{\vee}$
 and any $X \in {\mathfrak{p}}_{\mathbb{C}}$.  
The resulting $P$-homomorphism 
$
   \widetilde T \colon V \to (C^{\infty}(P)\otimes W)^{\mathfrak{q}}
$
 is injective
 because $(\xi,V)$ is irreducible.  
\end{proof}

\begin{remark}
In general, 
 the $P$-homomorphism $\widetilde T \colon V \to (C^{\infty}(P) \otimes W)^{\mathfrak{q}}$
 is not surjective.  
We note
 that $\widetilde T$ is bijective
 if $P$ is a connected compact Lie group
 by the Borel--Weil theorem.  
\end{remark}

Suppose $P_{\mathbb{C}}$ is a connected complex Lie group, 
 and $Q$ a parabolic subgroup of $P_{\mathbb{C}}$
 with Lie algebra ${\mathfrak{q}}$.  
For a holomorphic character ${\mathbb{C}}_{\lambda}$
 of $Q$, 
 we denote by ${\mathcal{L}}_{\lambda}$
 the $P_{\mathbb{C}}$-equivariant holomorphic line bundle
 $P_{\mathbb{C}} \times_Q {\mathbb{C}}_{\lambda}$
 over the flag variety $P_{\mathbb{C}}/Q$, 
 and by ${\mathcal{O}}(P_{\mathbb{C}}/Q, {\mathcal{L}}_{\lambda})$
 the space of holomorphic sections for ${\mathcal{L}}_{\lambda}$.

\begin{lemma}
[Geometric realization for $\operatorname{Irr}(P_{\mathbb{C}};{\mathfrak{q}})_{\operatorname{hol}}$]
\label{lem:BW}
{\rm{(1)}}\enspace
The regular representation of $P_{\mathbb{C}}$
 on ${\mathcal{O}}(P_{\mathbb{C}}/Q,{\mathcal{L}}_{\lambda})$
 belongs to 
$
   \operatorname{Irr}
   (P_{\mathbb{C}};
   {\mathfrak{q}}_{\operatorname{opp}})
_{\operatorname{hol}}
$
 if it is non-zero, 
 and its contragredient representation
 belongs to $\operatorname{Irr}(P_{\mathbb{C}};{\mathfrak{q}})_{\operatorname{hol}}$.  
\par\noindent
{\rm{(2)}}\enspace
Assume that $(\xi,V)$ is an irreducible holomorphic representation of $P_{\mathbb{C}}$
 such that its contragredient representation $\xi^{\vee}$ belongs 
 to $\operatorname{Irr}(P_{\mathbb{C}};{\mathfrak{q}})_{\operatorname{hol}}$.   
Then there exists a holomorphic character $\lambda$ of $Q$
 such that 
 $\xi$ is isomorphic to the regular representation
 of $P_{\mathbb{C}}$
 on ${\mathcal{O}}(P_{\mathbb{C}}/Q, {\mathcal{L}}_{\lambda})$.  
\end{lemma}

We shall apply the above lemma
 also to real forms 
 $P$ of $P_{\mathbb{C}}$, 
 where we do not assume $P$ to be connected.

\begin{proof}
Let $L_{\mathbb{C}}$ be a Levi subgroup of $P_{\mathbb{C}}$.  
We note that $L_{\mathbb{C}}$ is connected.  
\par\noindent
(1)\enspace
By \eqref{eqn:IrrPL}, 
 it suffices to prove the assertion
 when $P_{\mathbb{C}}=L_{\mathbb{C}}$.  
We take a Cartan subalgebra $\widetilde{\mathfrak{j}}_{\mathbb{C}}$
 of ${\mathfrak{l}}_{\mathbb{C}}$, 
 and fix a positive system 
 $\Delta^+({\mathfrak{l}}_{\mathbb{C}},\widetilde {\mathfrak{j}}_{\mathbb{C}})$
 such that 
 $\Delta^+({\mathfrak{l}}_{\mathbb{C}},\widetilde {\mathfrak{j}}_{\mathbb{C}})
 \subset
 \Delta ({\mathfrak{q}},\widetilde{\mathfrak{j}}_{\mathbb{C}})$.  
We use the same letter $\lambda$ 
 to denote the differential, 
 and also its restriction 
 to the Cartan subalgebra $\widetilde {\mathfrak{j}}_{\mathbb{C}}$.  
Then the Borel--Weil theorem 
 for the  connected complex reductive Lie group
 $L_{\mathbb{C}}$
 tells that 
 ${\mathcal{O}}(P_{\mathbb{C}}/Q, {\mathcal{L}}_{\lambda})$ is non-zero
 if and only if $-\lambda$ is dominant
 with respect to $\Delta^+({\mathfrak{l}}_{\mathbb{C}},\widetilde {\mathfrak{j}}_{\mathbb{C}})$.  
In this case, 
 its contragredient representation contains
 ${\mathbb{C}}_{\lambda}$ 
as a ${\mathfrak{q}}$-submodule, 
 namely, 
 has the highest weight $\lambda$, 
 as seen in Lemma \ref{lem:IrrGQ}.  
\par\noindent
(2)\enspace
Retain the notation as in the proof
 of Lemma \ref{lem:geomIrrPq}.  
Then the matrix coefficient $T(v,u)$ is a holomorphic function on $P_{\mathbb{C}}$, 
 because $\xi$ is a holomorphic representation of $P_{\mathbb{C}}$.  
Since $(\xi^{\vee}, V^{\vee}) \in \operatorname{Irr}(P_{\mathbb{C}};{\mathfrak{q}})_{\operatorname{hol}}$, 
 there exists a one-dimensional ${\mathfrak{q}}$-submodule ${\mathbb{C}}u$
 in $V^{\vee}$, 
 on which the connected group $Q$ acts as a holomorphic character, 
 to be denoted by $\lambda$.  
Then $T(\cdot, u)$ induces a $P_{\mathbb{C}}$-homomorphism from $V$
 to $({\mathcal{O}}(P_{\mathbb{C}}) \otimes {\mathbb{C}}_{\lambda})^{\mathfrak{q}} \simeq {\mathcal{O}}(P_{\mathbb{C}}/Q, {\mathcal{L}}_{\lambda})$, 
 which is bijective by the irreducibility.  
Thus the lemma is shown.  
\end{proof}

\subsection{Multiplicities in finite-dimensional representations}
\label{subsec:fdrep}

For finite-dimensional representations, 
 the boundedness of multiplicity is equivalent
 to multiplicity-freeness
 in many settings.  
In this subsection 
 we prepare two lemmas
 in a way 
 that we need later.  
For the sake of completeness, 
 we give a proof of the first one.

Let $G_{\mathbb{C}}$ be a connected complex reductive Lie group.   
We take a Cartan subalgebra $\widetilde {\mathfrak{j}}_{\mathbb{C}}$
 of ${\mathfrak{g}}_{\mathbb{C}}$
 and fix a positive system $\Delta^+({\mathfrak{g}}_{\mathbb{C}}, 
\widetilde {\mathfrak{j}}_{\mathbb{C}})$.  
We denote by 
${\mathfrak{b}}$ the corresponding Borel subalgebra
 of ${\mathfrak{g}}_{\mathbb{C}}$, 
 by $\Lambda_+$ 
 the set of dominant integral weights, 
 and by $\Pi_{\lambda}$
 the irreducible holomorphic representation of $G_{\mathbb{C}}$
 if $\lambda \in \Lambda_+$ lifts
 to a character of the Cartan subgroup
 as before.

Suppose that ${\mathfrak{q}}$ is a parabolic subalgebra
 of ${\mathfrak{g}}_{\mathbb{C}}$.  
Without loss of generality, 
 we may assume
 that ${\mathfrak{q}}$ contains ${\mathfrak{b}}$.  
As in Lemma \ref{lem:IrrGQ} (1),
 we regard 
 $\operatorname{Irr}(G_{\mathbb{C}};{\mathfrak{q}})_{\operatorname{hol}}$
 as a subset of $\Lambda_+$
 via the bijection \eqref{eqn:CW}.

First, 
 suppose that $G_{\mathbb{C}}'$ is a connected complex reductive subgroup.  
For $\Pi \in \operatorname{Irr}(G_{\mathbb{C}})_{\operatorname{hol}}$, 
 we set
\begin{equation}
\label{eqn:mPi}
  m(\Pi|_{G_{\mathbb{C}}'})
  :=
  \underset{\pi \in \operatorname{Irr}(G_{\mathbb{C}}')_{\operatorname{hol}}}{\max} 
  [\Pi|_{G_{\mathbb{C}}'}:\pi], 
\end{equation}
 as an analog of $m(\Pi|_{G'})$ in \eqref{eqn:msup}.

\begin{lemma}
\label{lem:VKN}
If $G_{\mathbb{C}}/Q$ is not $G_{\mathbb{C}}'$-spherical, 
 then there exists $\lambda \in \Lambda_+$
 satisfying 
 $\Pi_{N\lambda} \in \operatorname{Irr}(G_{\mathbb{C}};{\mathfrak{q}})_{\operatorname{hol}}$
 and $m(\Pi_{N\lambda}|_{G_{\mathbb{C}}'}) \ge N+1$
 for all $N \in {\mathbb{N}}$.  
\end{lemma}

\begin{proof}
Suppose $G_{\mathbb{C}}/Q$ is not $G_{\mathbb{C}}'$-spherical.  
By a result of Vinberg--Kimelfeld \cite[Cor.~1]{VK78}, 
 there exists a $G_{\mathbb{C}}$-homogeneous holomorphic line bundle ${\mathcal{L}}$
 over $G_{\mathbb{C}}/Q$
 such that the irreducible $G_{\mathbb{C}}$-module
 ${\mathcal{O}}(G_{\mathbb{C}}/Q, {\mathcal{L}})$
 contains an irreducible representation of $G_{\mathbb{C}}'$
 with multiplicity.  
This means that there exist linearly independent sections
 $f_1$, $f_2 \in {\mathcal{O}}(G_{\mathbb{C}}/Q, {\mathcal{L}})$
 and a dominant character $\mu$ of a Borel subgroup $B'$
 of $G_{\mathbb{C}}'$
 satisfying
$
  f_j(b^{-1} g)=\mu(b) f_j(g)
$
 ($j=1,2$) for any $b\in B'$, $g \in G_{\mathbb{C}}$.

We claim that the holomorphic sections 
 $f_1^i f_2^{N-i} \in {\mathcal{O}}(G_{\mathbb{C}}/Q, {\mathcal{L}}^{\otimes N})$ 
 ($0 \le i \le N$)
 are linearly independent.  
Indeed, 
 if $a_0 f_1^N + a_1 f_1^{N-1}f_2 + \cdots +a_N f_2^N=0$
 were a linear dependence, 
 then one would have $f_1 -t f_2=0$
 where $t$ is a zero of the equation
 $a_0 t^N + a_1 t^{N-1} + \cdots + a_N=0$
 because the ring ${\mathcal{O}}(G_{\mathbb{C}})$
 has no divisor.  
This means 
$
   \dim_{\mathbb{C}} 
   \invHom {G_{\mathbb{C}}'}
   {\pi_{N\mu}}
{{\mathcal{O}}
(G_{\mathbb{C}}/Q, {\mathcal{L}}^{\otimes N})|_{G_{\mathbb{C}}'}} 
   \ge N+1$
 because $B'$ acts on $f_1^i f_2^{N-i}$
 as the character ${\mathbb{C}}_{N\mu}$.  
Let $\lambda$ be the character of $Q$ acting 
 on the fiber of ${\mathcal{L}}^{-1}$ at the origin 
 $o = eQ \in G_{\mathbb{C}}/Q$.  
Then $\Pi_{N\lambda}$ is  the contragredient representation
 on ${\mathcal{O}}(G_{\mathbb{C}}/Q, {\mathcal{L}}^{\otimes N})$
 and  belongs to $\operatorname{Irr}(G_{\mathbb{C}};{\mathfrak{q}})_{\operatorname{hol}}$
 by Lemma \ref{lem:BW}
 (the Borel--Weil theorem).  
Hence
$
   \dim_{\mathbb{C}} \invHom {G_{\mathbb{C}}'}{\Pi_{N\lambda}|_{G_{\mathbb{C}}'}} {\pi_{N\mu}^{\vee}}
  \ge N+1
$, 
 showing the lemma.  
\end{proof}

Second, 
 we drop the reductive assumption of a subgroup.  
By a similar argument as in Lemma \ref{lem:VKN}, 
 one obtains the following:
\begin{lemma}
\label{lem:211197}
Let $H_{\mathbb{C}}$ be a complex algebraic subgroup of $G_{\mathbb{C}}$
(not necessarily reductive).  
If $H_{\mathbb{C}}$ does not have an open orbit in $G_{\mathbb{C}}/Q$, 
 then there exist a holomorphic character $\chi$ of $H_{\mathbb{C}}$
 and $\lambda \in \Lambda_+$ 
 satisfying $\Pi_{N \lambda} \in \operatorname{Irr}(G_{\mathbb{C}};{\mathfrak{q}})_{\operatorname{hol}}$
 and 
$
  \dim_{\mathbb{C}} \invHom{H_{\mathbb{C}}}{\Pi_{N\lambda}|_{H_{\mathbb{C}}}}{\chi^N} \ge N+1
$
 for all $N \in {\mathbb{N}}$.  
\end{lemma}

\subsection{Complex symmetric pair and the Satake diagram}
\label{subsec:Satake}

This subsection and the next one will not be used
 until Section \ref{sec:5}.

Let $G$ be a connected semisimple Lie group, 
 $\sigma$ an involutive automorphism of $G$, 
 and $H$ an open subgroup of the fixed point group $G^{\sigma}$.  
We use the same letter $\sigma$
 to denote the complex linear extension 
 of its differential on ${\mathfrak{g}}_{\mathbb{C}}$.  
Then the Lie algebra ${\mathfrak{g}}$ of $G$ has
 a decomposition 
$
   {\mathfrak{g}}={\mathfrak{g}}^{\sigma} \oplus {\mathfrak{g}}^{-\sigma}
$
 into the eigenspaces
 of $\sigma$
 with eigenvalues $1$ and $-1$.  
We note
 that the Lie algebra ${\mathfrak{h}}$ of $H$ equals ${\mathfrak{g}}^{\sigma}$. We take a maximal semisimple abelian subspace ${\mathfrak{j}}$
 in ${\mathfrak{g}}^{-\sigma}$, 
 and fix $\Sigma^+({\mathfrak{g}}_{\mathbb{C}},{\mathfrak{j}}_{\mathbb{C}})$
 as in Definition \ref{def:Borel}.  
We extend ${\mathfrak{j}}$
 to a Cartan subalgebra $\widetilde{\mathfrak{j}}$
 of ${\mathfrak{g}}$.  
We refer to $\widetilde {\mathfrak{j}}$
 as a $\sigma$-{\it{split Cartan subalgebra}}.  
Via the direct sum decomposition 
 $\widetilde{\mathfrak{j}}={\mathfrak{t}} \oplus \mathfrak{j}$
 with ${\mathfrak{t}}:=\widetilde{\mathfrak{j}} \cap {\mathfrak{h}}$, 
 we may regard ${\mathfrak{j}}_{\mathbb{C}}^{\ast}$
 as a subspace of $\widetilde {\mathfrak{j}}_{\mathbb{C}}^{\ast}$.  
We choose a compatible positive system
$\Delta^+({\mathfrak{g}}_{\mathbb{C}},\widetilde{\mathfrak{j}_{\mathbb{C}}})$
 such that the restriction map 
 $\alpha \mapsto \alpha|_{\mathfrak{j}_{\mathbb{C}}}$
 sends $\Delta^+({\mathfrak{g}}_{\mathbb{C}},\widetilde{\mathfrak{j}_{\mathbb{C}}})$
 to $\Sigma^+({\mathfrak{g}}_{\mathbb{C}},\mathfrak{j}_{\mathbb{C}})
 \cup \{0\}$, 
 and denote by $\Psi$
the set of simple roots
 in $\Delta^+({\mathfrak{g}}_{\mathbb{C}},\widetilde {\mathfrak{j}}_{\mathbb{C}})$
 as in Section \ref{subsec:ptheta}.  
We set 
\begin{equation}
\label{eqn:Theta}
\Theta
:=\{\alpha \in \Psi: \alpha|_{{\mathfrak{j}}_{\mathbb{C}}}\equiv 0\}.  
\end{equation}
By Definition \ref{def:Borel}, 
 one has the following:
\begin{lemma}
\label{lem:BSatake}
The parabolic subalgebra ${\mathfrak{p}}^{\Theta}$
 of ${\mathfrak{g}}_{\mathbb{C}}$ is a Borel subalgebra ${\mathfrak{q}}$
 for the symmetric space $G/H$.  
\end{lemma}

One can read $\Theta$ from the Satake diagram
 ({\it{e.g.}}, \cite[p.~531]{He78})
 of another real form ${\mathfrak{g}}_{\mathbb{R}}$
 of the complex Lie algebra ${\mathfrak{g}}_{\mathbb{C}}$, 
 which we explain below.  
We take a Cartan involution $\theta$
 of ${\mathfrak{g}}_{\mathbb{C}}$
 commuting with $\sigma$.  
Since $\sigma$ is complex linear
 and $\theta$ is antilinear on ${\mathfrak{g}}_{\mathbb{C}}$, 
 $\sigma \theta$ is a complex conjugation
 of ${\mathfrak{g}}_{\mathbb{C}}$, 
 and ${\mathfrak{g}}_{\mathbb{R}}:={\mathfrak{g}}_{\mathbb{C}}^{\sigma\theta}$
 is a real form of ${\mathfrak{g}}_{\mathbb{C}}$.  
This yields 
 a one-to-one correspondence:
\begin{equation}
\label{eqn:gRghC}
\text{
a real form ${\mathfrak{g}}_{\mathbb{R}}$
 of ${\mathfrak{g}}_{\mathbb{C}}$
}
\longleftrightarrow
\text{a complex symmetric pair
 $({\mathfrak{g}}_{\mathbb{C}}, {\mathfrak{h}}_{\mathbb{C}})$.  
}
\end{equation}

We set ${\mathfrak{k}}_{\mathbb{R}}:={\mathfrak{h}}_{\mathbb{C}} \cap {\mathfrak{g}}_{\mathbb{R}}$.  
Note that $\sigma$ leaves ${\mathfrak{g}}_{\mathbb{R}}$ invariant, 
 and the restriction $\sigma|_{{\mathfrak{g}}_{\mathbb{R}}}$
 is a Cartan involution of ${\mathfrak{g}}_{\mathbb{R}}$.  
In particular,
 $\widetilde {\mathfrak{j}}_{\mathbb{C}} \cap {\mathfrak{g}}_{\mathbb{R}}$
 is a maximally split Cartan subalgebra of ${\mathfrak{g}}_{\mathbb{R}}$.  
Since a Borel subgroup 
 for the symmetric space $G/H$
 is determined 
 only by the complexified Lie algebras 
 ${\mathfrak{g}}_{\mathbb{C}}$ and ${\mathfrak{h}}_{\mathbb{C}}$, 
 and since 
$({\mathfrak{g}}_{\mathbb{C}},{\mathfrak{h}}_{\mathbb{C}}) \simeq 
(({\mathfrak{g}}_{\mathbb{R}})_{\mathbb{C}}, ({\mathfrak{k}}_{\mathbb{R}})_{\mathbb{C}})
$, 
 one obtains from Lemma \ref{lem:BSatake}
 the following:

\begin{lemma}
\label{lem:Satake}
The complexification of a minimal parabolic subalgebra
 of ${\mathfrak{g}}_{\mathbb{R}}$ is a Borel subalgebra
 for the symmetric space $G/H$.  
In particular, 
 if we take $\Theta$ to be the set of black circles 
in the Satake diagram, 
 then ${\mathfrak{p}}^{\Theta}$ is a Borel subalgebra
 of $G/H$.  
\end{lemma}

We shall use Lemma \ref{lem:Satake} in Section \ref{sec:pfclass}
 in the proof of the classification results
 in Section \ref{sec:classification}.  

\subsection{The Cartan--Helgason theorem vs $\operatorname{Irr}({\mathfrak{g}};{\mathfrak{p}}^{\Theta})_f$}
\label{subsec:CH}
In this subsection, 
 we examine $\operatorname{Irr}(G)_{H,f}$
 for a symmetric space $G/H$
 and compare it with 
$
   \operatorname{Irr}({\mathfrak{g}};{\mathfrak{p}}^{\Theta})_f
$, 
 see \eqref{eqn:JLTIrrpq}, 
 where ${\mathfrak{p}}^{\Theta}$ is the Borel subalgebra
 of the symmetric space $G/H$.

We retain the notation 
 as in Section \ref{subsec:Satake}, 
 and set 
\begin{equation}
\label{eqn:JLTCH}
  \Lambda_+({\mathfrak{g}}_{\mathbb{C}};{\mathfrak{h}}_{\mathbb{C}})
  :=
  \{\lambda \in {\mathfrak{j}}_{\mathbb{C}}^{\ast}
   :
   \frac{\langle \lambda, \beta \rangle}{\langle \beta, \beta \rangle} \in {\mathbb{N}}
   \quad
   \text{for all } \beta \in \Sigma^+({\mathfrak{g}}_{\mathbb{C}},{\mathfrak{j}}_{\mathbb{C}})
\}.  
\end{equation}
We regard 
$\Lambda_+({\mathfrak{g}}_{\mathbb{C}};{\mathfrak{h}}_{\mathbb{C}})$
 $(\subset {\mathfrak{j}}_{\mathbb{C}}^{\ast})$
 as a subset of $\Lambda_+$ ($\subset \widetilde{{\mathfrak{j}}_{\mathbb{C}}}^{\ast}$)
 via the decomposition 
 $\widetilde{\mathfrak{j}}={\mathfrak{t}} \oplus {\mathfrak{j}}$.  
Since $\langle \lambda, \alpha \rangle =0$
 for any $\lambda \in {\mathfrak{j}}_{\mathbb{C}}^{\ast}$
 and any $\alpha \in \Theta$, 
 Lemma \ref{lem:IrrGQ} (1) implies the following:

\begin{lemma}
\label{lem:CHabove}
Via the Cartan--Weyl bijection \eqref{eqn:CW}, 
 one has
\begin{equation}
\label{eqn:CHabove}
   \Lambda_+({\mathfrak{g}}_{\mathbb{C}};{\mathfrak{h}}_{\mathbb{C}})
   \subset
   \operatorname{Irr}({\mathfrak{g}};{\mathfrak{p}}^{\Theta})_f.  
\end{equation}
\end{lemma}

\begin{remark}
\label{rem:rankCH}
Both $\Lambda_+({\mathfrak{g}}_{\mathbb{C}};{\mathfrak{h}}_{\mathbb{C}})$
 and $\operatorname{Irr}({\mathfrak{g}};{\mathfrak{p}}^{\Theta})_f$
 are free semigroups, 
 but the former
 may be much smaller
 than the latter.  
For example, 
 if $(G,H)=({}^\backprime G \times {}^\backprime G, \operatorname{diag}{}^\backprime G)$, 
 then the rank of the semigroup 
$
   \operatorname{Irr}({\mathfrak{g}};{\mathfrak{p}}^{\Theta})_f
$
 is twice the rank of the semigroup  
$
     \Lambda_+({\mathfrak{g}}_{\mathbb{C}};{\mathfrak{h}}_{\mathbb{C}})
$.  
\end{remark}

For the simplicity of the proof, 
 we adopt the definition 
 of $\operatorname{Irr}(G)_{H,f}$ as the set 
$
  \{\Pi \in \operatorname{Irr}(G)_{f}:
   \Pi^H \ne \{0\}
  \}
$
 rather than 
$
   \{\Pi \in \operatorname{Irr}(G)_{f}:
   (\Pi^{\vee})^H \ne \{0\}
  \}
$
in the next lemma
by an abuse of notation, 
 however, 
this definition coincides 
 with the previous one
 as we shall prove in Lemma \ref{lem:GHdual} below.  

\begin{lemma}
[Cartan--Helgason theorem]
\label{lem:CH}
Let $(G,H)$ be a symmetric pair 
 defined by an involution $\sigma$ of a connected semisimple Lie group
 $G$.  
We regard both $\operatorname{Irr}(G)_f$
 and $\Lambda_+({\mathfrak{g}}_{\mathbb{C}}, {\mathfrak{h}}_{\mathbb{C}})$
 as a subset of $\Lambda_+$.  
\begin{enumerate}
\item[{\rm{(1)}}]
$\Lambda_+ ({\mathfrak{g}}_{\mathbb{C}}; {\mathfrak{h}}_{\mathbb{C}})
 = \operatorname{Irr}(G)_{H,f}$
 if $\sigma$ is a Cartan involution of $G$.  

\item[{\rm{(2)}}]
If $G$ is a real form of the simply connected complex Lie group 
 $G_{\mathbb{C}}$
 with Lie algebra ${\mathfrak{g}}_{\mathbb{C}}$, 
 then $\Lambda_+({\mathfrak{g}}_{\mathbb{C}}; {\mathfrak{h}}_{\mathbb{C}})=
\operatorname{Irr}(G)_{H,f}$.  

\item[{\rm{(3)}}]
For a general semisimple symmetric pair $(G,H)$, 
there exists a positive integer $k$
 such that  
\begin{equation}
\label{eqn:GHut}
k   \Lambda_+({\mathfrak{g}}_{\mathbb{C}};{\mathfrak{h}}_{\mathbb{C}})
   \subset 
   \operatorname{Irr}(G)_{H,f}
   \subset
   \Lambda_+({\mathfrak{g}}_{\mathbb{C}}; {\mathfrak{h}}_{\mathbb{C}}).  
\end{equation}
\end{enumerate}
\end{lemma}

\begin{proof}[Proof of Lemma \ref{lem:CH}]
(1)\enspace
This is the (usual) Cartan--Helgason theorem.  
See \cite[Thm.~3.3.1.1]{War72} or \cite[p.~139]{He94}
 for example.  
\newline
(2)\enspace
The involution $\sigma$ of $G$ lifts to a holomorphic involution
 of the simply connected complex group $G_{\mathbb{C}}$, 
 for which we shall use the same letter $\sigma$.  
We take a Cartan involution $\theta$ of $G_{\mathbb{C}}$
 commuting with $\sigma$.  
Then $\sigma \theta$ is an anti-holomorphic involution
 of $G_{\mathbb{C}}$.  
We set $H_{\mathbb{C}}:=G_{\mathbb{C}}^{\sigma}$, 
 $G_{\mathbb{R}}:=G_{\mathbb{C}}^{\theta\sigma}$.  
Since $G_{\mathbb{C}}$ is simply connected, 
 both $H_{\mathbb{C}}$ and $G_{\mathbb{R}}$ are connected
 by a result of Borel
 \cite{Bo61}, 
 and $K_{\mathbb{R}}:=H_{\mathbb{C}} \cap G_{\mathbb{R}}$ is a maximal compact subgroup of $G_{\mathbb{R}}$.  
Therefore, 
 one has a natural bijection $\operatorname{Irr}(G_{\mathbb{R}})_f \simeq \operatorname{Irr}(G)_f \simeq \Lambda_+$
 via the holomorphic continuation 
 because $G_{\mathbb{C}}$ is simply connected.  
Since both $H$ and $K_{\mathbb{R}}$ are real forms
 of the connected complex Lie group $H_{\mathbb{C}}$, 
 there is a one-to-one correspondence
 between $\operatorname{Irr}(G_{\mathbb{R}})_{K_{\mathbb{R}},f}$
 and $\operatorname{Irr}(G)_{H,f}$, 
 and the former identifies with $\Lambda_+({\mathfrak{g}}_{\mathbb{C}};{\mathfrak{h}}_{\mathbb{C}})$ by (1).   
\newline
(3)\enspace
We now consider the general case
 where $G$ is not necessarily a subgroup
 of the simply connected group $G_{\mathbb{C}}$, 
 and use the letter $\widetilde G$
 to denote the analytic subgroup of $G_{\mathbb{C}}$
 with Lie algebra ${\mathfrak{g}}$.  
(Note that $G$ in the proof of (2) played the role of $\widetilde G$ here.)
The holomorphic involution $\sigma$ 
 of $G_{\mathbb{C}}$ leaves $\widetilde G$ invariant.  
We set $\widetilde H:=\widetilde G^{\sigma}$.  
Then $\operatorname{Irr}(\widetilde G)_{\widetilde H, f} \simeq \Lambda_+({\mathfrak{g}}_{\mathbb{C}};{\mathfrak{h}}_{\mathbb{C}})$ by (2).

For any $(\Pi, V) \in \operatorname{Irr}(G)_f$, 
 the simply connected group $G_{\mathbb{C}}$ acts holomorphically on $V$, 
 and thus one has a natural quotient map 
 $\widetilde G \to G/\operatorname{Ker} \Pi$.  
In turn, 
 one has an injection 
$
   \operatorname{Irr}(G)_f \hookrightarrow \operatorname{Irr}(\widetilde G)_f
 \simeq \Lambda_+
$, 
 which induces
\[
\operatorname{Irr}(G)_{H,f} \subset \operatorname{Irr}(G)_{H_0,f}
             \hookrightarrow \operatorname{Irr}(\widetilde G)_{\widetilde H,f}
=
\Lambda_+({\mathfrak{g}}_{\mathbb{C}};{\mathfrak{h}}_{\mathbb{C}}), 
\]
 where $H_0$ denotes the identity component of $H$.  
For the middle inclusion, 
 we have used 
 that $\widetilde H=\widetilde G^{\sigma}$ is contained
 in the connected subgroup $G_{\mathbb{C}}^{\sigma}$.  
Hence we have shown the right inclusion in \eqref{eqn:GHut}.

To see the left inclusion in \eqref{eqn:GHut}, 
 we set $b:=[H:H_0]$, 
 the number of connected components in $H$.  
We claim that $\Pi_{b \lambda} \in \operatorname{Irr}(G)_{H,f}$
 for any $(\Pi_{\lambda}, V) \in \operatorname{Irr}(G)_{H_0,f}$.  
In fact, 
 we take a generator $v$
 in the space $V^{H_0}$ of $H_0$-fixed vectors in $V$, 
 which is one-dimensional.  
Then the quotient group $H/H_0$ leaves $V^{H_0}={\mathbb{C}}v$ invariant.  
On the other hand, 
 the $b$-th tensor product representation
 $V \otimes \cdots \otimes V$
 contains uniquely an irreducible subrepresentation $\Pi_{b \lambda}$.  
Let $S \colon V \otimes \cdots \otimes V \to \Pi_{b \lambda}$
 be the projection.  
Then the $H_0$-fixed indecomposable vector
 $v \otimes \cdots \otimes v \in V \otimes \cdots \otimes V$
 has a non-zero image, 
 say $v_b$, in $\Pi_{b \lambda}$.  
Moreover, 
 since $H/H_0$ acts on ${\mathbb{C}}v$
 as a scalar, 
 its diagonal action on ${\mathbb{C}} v\otimes \cdots \otimes v$
 is trivial
 because $b$ is the order of the finite group $H/H_0$.  
Thus $\Pi_{b \lambda} \in \operatorname{Irr}(G)_{H,f}$.

We take a positive integer $a$ such that
$a \Lambda_+ =a \operatorname{Irr}(\widetilde G)_f \subset \operatorname{Irr}(G)_f$.  
Then one has 
$
   a \Lambda_+({\mathfrak{g}}_{\mathbb{C}};{\mathfrak{h}}_{\mathbb{C}})
  = a \operatorname{Irr}(\widetilde G)_{\widetilde H, f}
  \subset \operatorname{Irr}(G)_{H_0,f}.  
$
Hence we have shown
\[
  a b \Lambda_+({\mathfrak{g}}_{\mathbb{C}};{\mathfrak{h}}_{\mathbb{C}})
  \subset
  b \operatorname{Irr}(G)_{H_0, f} \subset \operatorname{Irr}(G)_{H,f}.  
\]
This proves the left inclusion of \eqref{eqn:GHut} with $k=ab$.  
\end{proof}

\begin{lemma}
\label{lem:GHdual}
We consider two involutions of $\operatorname{Irr}(G)_f$ given by
\begin{align*}
\Pi & \mapsto \Pi^{\sigma}:=\Pi \circ \sigma, 
\\
\Pi & \mapsto \Pi^{\vee}
\quad
\text{(contragredient representation)}.  
\end{align*}
Then $\Pi^{\sigma} \simeq \Pi^{\vee}$
 for all $\Pi \in \operatorname{Irr}(G)_{H,f}$.  
In particular, 
 $\Pi^{\vee} \in \operatorname{Irr}(G)_{H,f}$
 if and only if $\Pi \in \operatorname{Irr}(G)_{H,f}$.  
\end{lemma}

\begin{proof}
Suppose $\lambda$ is the highest weight of $\Pi$.  
Then $\Pi^{\vee}$ has an extremal weight $-\lambda$, 
 whereas $\Pi^{\sigma}$ has an extremal weight $\sigma \lambda$
 which equals $-\lambda$ by Lemma \ref{lem:CH}.  
Hence $\Pi^{\sigma} \simeq \Pi^{\vee}$.

Suppose $\Pi \in \operatorname{Irr}(G)_{H,f}$.  
Then obviously $\Pi^{\sigma} \in \operatorname{Irr}(G)_{H,f}$, 
hence $\Pi^{\vee} \in \operatorname{Irr}(G)_{H,f}$.
\end{proof}

\section{Bounded multiplicity results for induction}
\label{sec:indbdd}

In the classical harmonic analysis
 on the Riemannian symmetric space $G/K$, 
 building blocks
 of representations in $C^{\infty}(G/K)$
 are constructed
 by the twisted Poisson transform, 
 an integral $G$-intertwining operator from 
the spherical principal series representation
 to $C^{\infty}(G/K)$, 
 see \cite{He94} for instance.  
More generally, 
 for a closed subgroup $H$ in $G$, 
 we consider the space
 $\invHom{G}
{\operatorname{Ind}_P^G(\xi)}
{\operatorname{Ind}_H^G(\tau)}
$
 of generalized Poisson transforms, 
 where $P$ is a parabolic subgroup of $G$, 
 $\xi \in \operatorname{Irr}(P)_f$, 
 and $\tau \in \operatorname{Irr}(H)_f$.  
In this section, 
 we give an estimate of
 $\dim_{\mathbb{C}} \invHom{G}
{\operatorname{Ind}_P^G(\xi)}
{\operatorname{Ind}_H^G(\tau)}
$
 as a refinement of the bounded multiplicity theorems
 proved in \cite[Thm.~B]{xktoshima}
 in terms of a pair of parabolic subgroups $Q \subset P_{\mathbb{C}}$.  
The main result of this section is Theorem \ref{thm:indbdd}, 
 of which the first statement provides a uniform bound
 of the multiplicities
 (\lq\lq{$QP$ estimate}\rq\rq)
 under a geometric condition 
 $\#(Q \backslash G_{\mathbb{C}}/H_{\mathbb{C}})<\infty$, 
 strengthening a formulation of Tauchi \cite{xtauchi}.  
In turn, 
 this leads us
 to \lq\lq{$QP$ estimates}\rq\rq\
 for restriction in Section \ref{sec:restbdd}.

\subsection{Geometric condition for bounded multiplicity}

Let $H$ be a closed subgroup of a Lie group $G$.  
For a finite-dimensional representation $(\eta, V)$ of $H$, 
 we write $\operatorname{Ind}_H^G(\eta)$
 for the (unnormalized) induced representation of $G$
 on the Fr{\'e}chet space 
 $C^{\infty}(G/H, {\mathcal{V}})$
 of smooth sections 
 for the homogeneous $G$-vector bundle
 ${\mathcal{V}}:=G\times_H V$
 over $G/H$.  
If $H$ is a parabolic subgroup $P$ of $G$, 
 then $\operatorname{Ind}_P^G(\eta)$
 is of moderate growth.

\begin{theorem}
[\lq\lq{$QP$ estimate}\rq\rq\ for induction]
\label{thm:indbdd}
Let $G$ be a real reductive algebraic Lie group, 
 $H$ an algebraic subgroup, 
 $P$ a parabolic subgroup of $G$, 
 and $G_{\mathbb{C}} \supset H_{\mathbb{C}}$, $P_{\mathbb{C}}$ their complexifications.  
Suppose that $Q$ is a (complex) parabolic subgroup of $G_{\mathbb{C}}$
 with Lie algebra ${\mathfrak{q}}$
 such that $Q \subset P_{\mathbb{C}}$.  
\begin{enumerate}
\item[{\rm{(1)}}]
 If $\#(Q \backslash G_{\mathbb{C}}/H_{\mathbb{C}})<\infty$, 
 then there exists $C>0$
such that
\begin{equation}
\label{eqn:JLTthm31a}
  \dim_{\mathbb{C}} \invHom G {\operatorname{Ind}_P^G(\eta)}{\operatorname{Ind}_H^G(\tau)} \le C d_{\mathfrak{q}}(\eta)\dim \tau
\end{equation}
for any $\eta \in \operatorname{Irr}(P)_f$
 and any $\tau \in \operatorname{Irr}(H)_f$.  
In particular, 
 one has 
\begin{equation}
\label{eqn:JLTthm31b}
  \dim_{\mathbb{C}} \invHom G {\operatorname{Ind}_P^G(\eta)}{\operatorname{Ind}_H^G(\tau)} \le C \dim \tau
\end{equation}
for any $\eta \in \operatorname{Irr}(P;{\mathfrak{q}})_f$
 and any $\tau \in \operatorname{Irr}(H)_f$.  
\item[{\rm{(2)}}]
Conversely, 
 if there exists $C>0$ such that 
\eqref{eqn:JLTthm31b} holds
for any $\eta \in \operatorname{Irr}(P;{\mathfrak{q}})_f$
 and any $\tau \in \operatorname{Irr}(H)_f$, 
 then $Q$ has an open orbit in $G_{\mathbb{C}}/H_{\mathbb{C}}$.  
\end{enumerate}
\end{theorem}

See Definition \ref{def:FPq}
 for the definition of $d_{\mathfrak{q}}(\eta)$
 $(\le \dim \eta)$, 
 and \eqref{eqn:JLTIrrPq}
 for the definition of $\operatorname{Irr}(P;{\mathfrak{q}})_f$.  
The point of Theorem \ref{thm:indbdd} is 
 that $Q$ is not necessarily defined over ${\mathbb{R}}$, 
 which we shall see useful in the proof of Theorem \ref{thm:bdd}
 in Section \ref{sec:5}.  
We also present a number of bounded multiplicity results for restriction in Section \ref{sec:restbdd}.  

\begin{remark}
As the proof shows, 
 one can relax the assumption
 of the second statement
 by the following condition:
there exists $C>0$ such that 
\[
  \dim_{\mathbb{C}} \invHom G {\operatorname{Ind}_P^G(\eta)}{\operatorname{Ind}_H^G(\tau)} \le C 
\]
for any $\eta \in \operatorname{Irr}(P;{\mathfrak{q}})_f$
 and any character $\tau$ of $H$.  
\end{remark}

\begin{remark}
\label{rem:33}
{\rm{(1)}}\enspace
If $Q$ is a Borel subgroup of $G_{\mathbb{C}}$, 
 then $d_{\mathfrak{q}}(\eta)=1$.  
In this case, 
 the first statement of Theorem \ref{thm:indbdd} 
 was proved in \cite[Thm.~B]{xktoshima}.  
\newline
{\rm{(2)}}\enspace
If $Q=P_{\mathbb{C}}$, 
 then $d_{\mathfrak{q}}(\eta)=\dim \eta$.  
In this case, 
 the first statement of Theorem \ref{thm:indbdd} was proved in Tauchi \cite[Thm.~1.13]{xtauchi}.  
\newline
{\rm{(3)}}\enspace
If $\#(P_{\mathbb{C}} \backslash G_{\mathbb{C}}/H_{\mathbb{C}})<\infty$, 
 then Theorem \ref{thm:indbdd} (1) implies
 a {\it{finite multiplicity theorem}}
 of the induction
\begin{equation}
\label{eqn:fm}
\dim_{\mathbb{C}} \invHom {G}{\operatorname{Ind}_P^G(\eta)}{\operatorname{Ind}_H^G(\tau)}<\infty
\quad
{}^{\forall} \eta  \in \operatorname{Irr}(P)_f, 
{}^{\forall} \tau \in \operatorname{Irr}(H)_f.  
\end{equation}
However, 
 the converse statement is not true.  
For example, 
 if $P$ is a minimal parabolic subgroup, 
 as one sees from \cite[Thm.~A]{xktoshima}
 that 
\[
\text{\eqref{eqn:fm}}
\iff
\text{$G/H$ is real spherical}
\quad
{\begin{matrix}
{\times \hskip -1.3pc \Longrightarrow}
\\[-.3em]
\!\!
\Longleftarrow
\end{matrix}}
\quad
\#(P_{\mathbb{C}}\backslash G_{\mathbb{C}}/H_{\mathbb{C}})<\infty, 
\]
hence $\#(P_{\mathbb{C}} \backslash G_{\mathbb{C}}/H_{\mathbb{C}})<\infty$
 is {\it{not}} a necessary condition
 for the finite multiplicity property \eqref{eqn:fm}.  
For a general parabolic subgroup $P$, 
 the following geometric necessary condition was 
 proved in \cite[Cor.~6.8]{xkProg2014}:
\[\text{
 \eqref{eqn:fm} $\Longrightarrow$ $P$ has an open orbit in $G/H$.  
}\]
The second statement in Theorem \ref{thm:indbdd}
 is a refinement 
 of this statement for \lq\lq{$QP$ estimate}\rq\rq.  
\end{remark}

\subsection{Proof of Theorem \ref{thm:indbdd} (1)}

In this subsection 
 we give a proof of the first statement 
 of Theorem \ref{thm:indbdd}.  
For a real analytic manifold $M$, 
 we denote by ${\mathcal{B}}$
 the sheaf of hyperfunctions
 {\`a} la Sato \cite{S59}.  
We shall regard distributions
 as generalized functions {\`a} la Gelfand \cite{GS64}
 rather than continuous linear forms on $C_c^{\infty}(M)$
 so that one has a natural inclusion
 $C^{\infty}(M) \subset {\mathcal{D}}'(M) \subset {\mathcal{B}}(M)$.

For $M=G$ (group manifold), 
 we consider two actions $L$ and $R$ of $G$
 on $C^{\infty}(G)$, ${\mathcal{D}}'(G)$, 
 and ${\mathcal{B}}(G)$
 by $L(g_1)R(g_2) f := f(g_1^{-1} \cdot g_2)$
 for $g_1, g_2 \in G$ on the same spaces, 
 which induce the actions
 of the complexified Lie algebra
 ${\mathfrak{g}}_{\mathbb{C}}$, 
 to be denoted by $d L$ and $d R$, respectively.

For a parabolic subgroup $P$ of $G$, 
 we denote by ${\mathbb{C}}_{2\rho}$
 the one-dimensional representation of $P$
 defined by 
\begin{equation}
\label{eqn:2rho}
  P \to {\mathbb{R}}_{>0}, \quad p \mapsto \det|\operatorname{Ad}(p) \colon {\mathfrak{g}}/{\mathfrak{p}} \to {\mathfrak{g}}/{\mathfrak{p}}|^{-1}.  
\end{equation}
Then ${\mathcal{L}}_{2\rho}:=G \times_P {\mathbb{C}}_{2\rho}$
 is the volume bundle over $G/P$, 
 and the integration yields a $G$-invariant linear form $C^{\infty}(G/P, {\mathcal{L}}_{\lambda}) \to {\mathbb{C}}$.

Let $(\eta, V)$ be a finite-dimensional representation of $P$
 and ${\mathcal{V}}:=G \times_P V$ the $G$-homogeneous vector bundle
 over $G/P$
 associated to $(\eta,V)$.  
We set $\eta^{\ast}:= \eta^{\vee} \otimes {\mathbb{C}}_{2 \rho}$, 
where $\eta^{\vee}$ is the contragredient representation of $\eta$.  
The dualizing bundle ${\mathcal{V}}^{\ast}$
 of ${\mathcal{V}}=G \times_P V$ is given 
 as the $G$-homogeneous vector bundle over $G/P$
 associated to $\eta^{\ast}$.  
Then one has a canonical $G$-invariant perfect pairing
 between $\operatorname{Ind}_P^G(\eta)$
 and $\operatorname{Ind}_P^G(\eta^{\ast})$
 by the composition of the two maps:
\[
  C^{\infty}(G/P, {\mathcal{V}}) \times C^{\infty}(G/P, {\mathcal{V}}^{\ast}) 
\to C^{\infty}(G/P, {\mathcal{L}}_{2\rho})
\to {\mathbb {C}}.  
\]

Suppose that $(\tau,W)$ is a finite-dimensional representation
 of a closed subgroup $H$.  
By the Schwartz kernel theorem, 
 any continuous linear operator $T$ from $\operatorname{Ind}_P^G(\eta)$ to  $\operatorname{Ind}_H^G(\tau)$
 can be obtained by a bundle-valued distribution kernel
 on $G/P \times G/H$.  
This distribution is $G$-invariant under the diagonal 
action on $G/P \times G/H$
 if $T$ intertwines $G$-actions.  
Moreover, 
 since $G/P$ is compact, 
 it follows from \cite[Prop.~3.2]{KS15}
 that one has a natural bijection:
\begin{equation}
\label{eqn:KSisom}
   \invHom G {\operatorname{Ind}_P^G(\eta)}{\operatorname{Ind}_H^G(\tau)}
   \simeq
   ({\mathcal{D}}'(G) \otimes \eta^{\ast} \otimes \tau)^{P \times H}, 
\end{equation}
where we let $P$ act on ${\mathcal{D}}'(G) \otimes V^{\vee}\otimes W$
 by $R \otimes \eta^{\ast} \otimes \operatorname{id}$, 
 and $H$ by $L \otimes \operatorname{id} \otimes \tau$.

Let $L$ be a Levi part of $P$, 
 and $L_{\mathbb{C}}$ its complexification.  
Since $P_{\mathbb{C}}/Q \simeq L_{\mathbb{C}}/(L_{\mathbb{C}} \cap Q)$ is a (generalized) flag variety
 of the complex reductive group $L_{\mathbb{C}}$, 
 one has 
$
    \# (P \backslash P_{\mathbb{C}}/Q)<\infty
$
 because 
$ 
   \# (L \backslash L_{\mathbb{C}}/(L_{\mathbb{C}} \cap Q))< \infty
$ 
 by a result of Wolf \cite{W69}.  
In particular, 
 one finds $x \in P_{\mathbb{C}}$ such that $P x Q/Q$ is closed
 in $P_{\mathbb{C}}/Q$.  
Obviously, 
 the assumption of Theorem \ref{thm:indbdd} is unchanged
 if we replace $Q$ by $\operatorname{Ad}(x) Q$, 
 and so is the conclusion of Theorem \ref{thm:indbdd}
 by \eqref{eqn:AdIPq}.  
Thus we may and do assume
 that $P/(P \cap Q)$ is closed in $P_{\mathbb{C}}/Q$ from now on.

Suppose $\eta \in \operatorname{Irr}(P)_f$.  
One observes $d_{\mathfrak{q}}(\eta)=d_{\mathfrak{q}}(\eta \otimes {\mathbb{C}}_{-2\rho})$.  
Take an irreducible ${\mathfrak{q}}$-submodule
 $\lambda^{\vee}$
 of the ${\mathfrak{p}}_{\mathbb{C}}$-module
 $(\eta^{\ast})^{\vee} \simeq \eta \otimes {\mathbb{C}}_{-2\rho}$
 such that
$
   \dim \lambda^{\vee}= d_{\mathfrak{q}}(\eta)
$.  
We write $\lambda$ $(\simeq \lambda^{\vee \, \vee})$
 for the contragredient representation of $\lambda^{\vee}$.  
Clearly, 
 $\dim \lambda=d_{\mathfrak{q}}(\eta)$.  
By Lemma \ref{lem:geomIrrPq}, 
 there is an injective $P$-homomorphism
 $\eta^{\ast}
 \hookrightarrow
 (C^{\infty}(P)\otimes {\lambda})^{\mathfrak{q}}$. 
Hence, 
 one has 
\[
 ({\mathcal{D}}'(G) \otimes \eta^{\ast})^P
\hookrightarrow 
 ({\mathcal{D}}'(G) \otimes C^{\infty}(P)
  \otimes {\lambda})^{P \times {\mathfrak{q}}}
\simeq
  ({\mathcal{D}}'(G) \otimes {\lambda})^{\mathfrak{q}}.  
\]
By \eqref{eqn:KSisom}, 
 the first statement of Theorem \ref{thm:indbdd} is reduced 
 to the following:

\begin{proposition}
\label{prop:qhbdd}
Assume $\#(Q \backslash G_{\mathbb{C}}/H_{\mathbb{C}})<\infty$.    
Then there exists $C>0$
 such that for any ${\mathfrak{q}}$-module $\lambda$
 and for any $\tau \in \operatorname{Irr}(H)_f$
\begin{equation}
\label{eqn:JLTprop34}
\dim_{{\mathbb{C}}}({\mathcal{D}}'(G) \otimes {\lambda}
\otimes \tau)^{{\mathfrak{q}} \oplus {\mathfrak{h}}_{\mathbb{C}}}
\le 
C \dim \lambda \dim \tau, 
\end{equation}
where ${\mathfrak{q}}$ acts on the first and second factors
 by $d R \otimes \lambda$, 
 and ${\mathfrak{h}}_{\mathbb{C}}$ on the first
 and third by $d L \otimes \tau$.  
\end{proposition}

For the proof, 
 we use the following result
 by Tauchi, 
 which is based on the theory
 of (regular) holonomic ${\mathcal{D}}$-modules
 (see \cite{Ka83, KaKw81} for instance).  

\begin{proposition}
[{\cite[Thm.~1.14]{xtauchi}}]
\label{prop:Dmodule}
Suppose a complex Lie group $B$ acts holomorphically 
 on  a complex manifold $X_{\mathbb{C}}$
 with finite number of orbits.  
Let $X_{\mathbb{R}}$ be a real form of $X_{\mathbb{C}}$, 
 and $U$ a relatively compact semi-analytic open subset of $X_{\mathbb{R}}$.  Then there exists a constant $C>0$
 such that 
$\dim_{\mathbb{C}}
({\mathcal{B}}(U) \otimes \mu)^{\mathfrak{b}} 
 \le C \dim \mu$
 for any finite-dimensional representation $\mu$
 of the Lie algebra ${\mathfrak{b}}$ of $B$.  
Here $({\mathcal{B}}(U) \otimes \mu)^{\mathfrak{b}}$
 denotes the space of ${\mathfrak{b}}$-invariant vector-valued hyperfunctions
 on $U$
 via the diagonal action.  
\end{proposition}

\begin{proof}
[Proof of Proposition \ref{prop:qhbdd}]
Let $K$ be a maximal compact subgroup of $G$.  
We recall that we have chosen $P$ 
 such that $P/(P\cap Q)$ is closed in the flag variety $P_{\mathbb{C}}/Q$.  
In particular, 
 the algebraic subgroup $P \cap Q$ is cocompact in $P$, 
 hence one has $G=K(P \cap Q)$.

As in \cite[Thm.~3.16]{KS15}, 
 we capture all invariant distributions
 (or hyperfunctions) on $G$
 by those on a sufficiently large open subset $U$.  
For this, 
we fix a relatively compact semi-analytic open neighbourhood
 of $o=K e \in K \backslash G$, 
 and define $U$ to be its full inverse 
 via the quotient map $G \to K \backslash G$.  
Then $U$ is a left $K$-invariant,
 relatively compact,
 semi-analytic open subset in $G$.  
Moreover the restriction map
$
 ({\mathcal{B}}(G) \otimes \mu)^{\mathfrak{q}}
 \to
 ({\mathcal{B}}(U) \otimes \mu)^{\mathfrak{q}}
$
 is injective
 because $(G_{\mathbb{C}} \supset)$ $U Q \supset K(P \cap Q)=G$.  
Hence one has natural inclusions:
\[
  ({\mathcal{D}}'(G) \otimes {\lambda} \otimes \tau)^{{\mathfrak{q}} \oplus {\mathfrak{h}}_{\mathbb{C}}}
\subset 
  ({\mathcal{B}}(G) \otimes {\lambda} \otimes \tau)^{{\mathfrak{q}} \oplus {\mathfrak{h}}_{\mathbb{C}}}
\subset   
  ({\mathcal{B}}(U) \otimes {\lambda} \otimes \tau)^{{\mathfrak{q}} \oplus {\mathfrak{h}}_{\mathbb{C}}}.  
\]
We take $X_{\mathbb{R}} :=(G \times G)/\operatorname{diag}G$ $(\simeq G)$
 and apply Proposition \ref{prop:Dmodule}
 to the setting
 $(B, X_{\mathbb{C}}, \mu):=(Q \times H_{\mathbb{C}}, G_{\mathbb{C}}, 
{\lambda} \otimes \tau)
$.  
Then $\dim_{\mathbb{C}} ({\mathcal{B}}_G(U) \otimes {\lambda} \otimes \tau)^{{\mathfrak{q}} \oplus {\mathfrak{h}}_{\mathbb{C}}} \le C \dim \lambda \dim \tau$, 
 whence the inequality \eqref{eqn:JLTprop34} follows.  
\end{proof}

\subsection{Proof of Theorem \ref{thm:indbdd} (2)}
\label{subsec:pfindbdd2}

We recall our setting
 where $G$ is a real form of $G_{\mathbb{C}}$.  
We take a positive system 
 $\Delta^+({\mathfrak{g}}_{\mathbb{C}}, \widetilde {\mathfrak{j}}_{\mathbb{C}})$
 such that the corresponding Borel subgroup $B$
 is contained in the complex parabolic subgroup $Q$ $(\subset P_{\mathbb{C}})$.

\begin{lemma}
\label{lem:210943}
Let $\Pi_{\lambda}$ be an irreducible holomorphic 
representation
 of $G_{\mathbb{C}}$
 with highest weight $\lambda \in \Lambda_+$.  
Let $\eta$ be the regular representation
 of $P_{\mathbb{C}}$
 on ${\mathcal{O}}(P_{\mathbb{C}}/B, {\mathcal{L}}_{-\lambda})$, 
 and define a representation of $P$
 by $\xi:=\eta^{\vee}|_P \otimes {\mathbb{C}}_{2\rho}$.  

\begin{enumerate}
\item[{\rm{(1)}}]
There is a surjective $G$-homomorphism 
 $\operatorname{Ind}_P^G(\xi) \to \Pi_{\lambda}$.  
\item[{\rm{(2)}}]
If $\Pi_{\lambda} \in \operatorname{Irr}(G;{\mathfrak{q}})_f$
 then $\xi \in \operatorname{Irr}(P;{\mathfrak{q}})_f$.  
\end{enumerate}
\end{lemma}

\begin{proof}
By the Borel--Weil theorem, 
 the contragredient representation $\Pi_{\lambda}^{\vee}$
 is realized as the regular representation
 on ${\mathcal{O}}(G_{\mathbb{C}}/B, {\mathcal{L}}_{-\lambda})$.  
We denote
 by ${\mathcal{V}}_{\eta}$ the $G_{\mathbb{C}}$-equivariant vector bundle
 over $G_{\mathbb{C}}/P_{\mathbb{C}}$ 
associated to the $P_{\mathbb{C}}$-module $\eta$.  
Induction by stages for $B \subset P_{\mathbb{C}} \subset G_{\mathbb{C}}$ shows a natural isomorphism as $G_{\mathbb{C}}$-modules:
\[
  (\Pi_{\lambda}^{\vee}
  \simeq)\,\,
  {\mathcal{O}}(G_{\mathbb{C}}/B, {\mathcal{L}}_{-\lambda})
  \simeq
  {\mathcal{O}}(G_{\mathbb{C}}/P_{\mathbb{C}}, {\mathcal{V}}_{\eta}),  
\]
which yields an injective $G$-homomorphism
$
  \Pi_{\lambda}^{\vee} \hookrightarrow \operatorname{Ind}_P^G(\eta|_P).  
$
Taking the dual, 
 we see that $\Pi_{\lambda}$ occurs as the quotient $G$-module
 of the induced representation 
$
 \operatorname{Ind}_P^G(\xi). 
$

\par\noindent
(2)\enspace
Using induction by stages for $B \subset Q\subset P_{\mathbb{C}}$ this time, 
 one has
$
  \eta \simeq {\mathcal{O}}(P_{\mathbb{C}}/Q, {\mathcal{V}}_{-\lambda}), 
$
 where ${\mathcal{V}}_{-\lambda}$ stands for the $P_{\mathbb{C}}$-equivariant 
holomorphic vector bundle over $P_{\mathbb{C}}/Q$
 associated to the $Q$-module 
 ${\mathcal{O}}(Q/B, {\mathcal{L}}_{-\lambda})$.  
By Lemma \ref{lem:IrrGQ} (1), 
 $\lambda$ is orthogonal to all the roots
 in the Levi subalgebra of ${\mathfrak{q}}$, 
 hence 
 ${\mathcal{O}}(Q/B, {\mathcal{L}}_{-\lambda})$ is one-dimensional.   
In turn, 
 one has $\eta^{\vee} \in \operatorname{Irr}(P_{\mathbb{C}};{\mathfrak{q}})_{\operatorname{hol}}$
 by Lemma \ref{lem:BW} (2), 
 and thus $\xi \in \operatorname{Irr}(P;{\mathfrak{q}})_f$
 because $\eta^{\vee}|_P  =\xi \otimes {\mathbb{C}}_{-2\rho}$.  
\end{proof}

\begin{proof}
[Proof of Theorem \ref{thm:indbdd} (2)]
Suppose that $Q$ does not have an open orbit
 in $G_{\mathbb{C}}/H_{\mathbb{C}}$.  
By Lemma \ref{lem:211197}, 
 there exist $\lambda \in \Lambda_+$ and a character $\chi$ of $H$
such that $\dim_{\mathbb{C}} \invHom H{\Pi_{N\lambda}|_H}{\chi^N} \ge N+1$
 for all $N \in {\mathbb{N}}$.  
In turn, 
it follows from Lemma \ref{lem:210943}
 that there exists 
 $\xi_N \in \operatorname{Irr}(P;{\mathfrak{q}})_f$
 for every $N \in {\mathbb{N}}$
 such that 
 the irreducible $G$-module $\Pi_{N\lambda}$
 is a quotient of the degenerate principal series representation 
 $\operatorname{Ind}_P^G(\xi_N)$, 
 hence one has
$\dim_{\mathbb{C}}\invHom H {\operatorname{Ind}_P^G(\xi_N)|_{H}}{\chi^N} \ge N+1$.  
The Frobenius reciprocity shows
$
   \dim_{\mathbb{C}}\invHom G 
  {\operatorname{Ind}_P^G(\xi_N)}{\operatorname{Ind}_H^G(\chi^N)} \ge N+1
$, 
whence the second statement of Theorem \ref{thm:indbdd}.  
\end{proof}

\section{Bounded multiplicity results for restriction}
\label{sec:restbdd}

In this section, 
 we derive bounded multiplicity results
 (\lq\lq{$QP$ estimates}\rq\rq)
 for {\it{restriction}} from those for {\it{induction}}
 in Theorem \ref{thm:indbdd}
 along the same line of the argument
 as in \cite{xkProg2014, xktoshima}. 
The results here will be used in Section \ref{sec:5}
 for the proof of Theorem \ref{thm:bdd}
 by a specific choice 
 of the parabolic subgroups $Q \subset P_{\mathbb{C}}$.  
Throughout this section, 
 we suppose
 that $G_{\mathbb{C}} \supset G_{\mathbb{C}}'$ are connected reductive Lie groups
 and that $G \supset G'$ are their real forms.

\subsection{Bounded multiplicity theorems for restriction}
\label{subsec:bddrest}
We begin with a \lq\lq{$QP$ estimates}\rq\rq\
 of the space of \lq\lq{symmetry breaking operators}\rq\rq\
 between (degenerate) principal series representations
 of a group $G$ and those of a subgroup $G'$.  
The results  might be
 of their own interest  
 because they indicate a nice broader framework 
 for detailed study
 of such operators.  
See 
 {\it{e.g.}}, \cite{KS15, xksbonvec}
 for the construction and the classification
 of symmetry breaking operators
 for principal series of two conformal groups,  
 see also Examples \ref{ex:kop}, \ref{ex:NO}, and \ref{ex:tri}
 by \cite{CKOP11, KOP11, NO18} for some of the most degenerate cases.  

\begin{theorem}
[\lq\lq{$QP$ estimate}\rq\rq\ for restriction]
\label{thm:restQQ}
Let $G \supset G'$ be a pair of real reductive algebraic Lie groups, 
 and $P$ and $P'$ are parabolic subgroups
 of $G$ and $G'$, 
respectively.  
Suppose that $Q$ and $Q'$ are (complex) parabolic subgroups
 of $G_{\mathbb{C}}$ and $G_{\mathbb{C}}'$, 
 respectively, 
 such that 
 ${\mathfrak{q}}\subset {\mathfrak{p}}_{\mathbb{C}}$, 
 ${\mathfrak{q}}' \subset {\mathfrak{p}}_{\mathbb{C}}'$, 
 and ${\#}(Q_{\operatorname{opp}}' \backslash G_{\mathbb{C}}/Q) < \infty$.  
Then there exists $C>0$ such that 
\begin{equation}
\label{eqn:dpsbdd}
  \dim_{\mathbb{C}} \invHom{G'}{\operatorname{Ind}_P^G(\xi)|_{G'}}{\operatorname{Ind}_{P'}^{G'}(\eta)} \le C d_{\mathfrak{q}}(\xi) d_{\mathfrak{q}'}(\eta) 
\end{equation}
 for any $\xi \in \operatorname{Irr}(P)_f$
 and any $\eta \in \operatorname{Irr}(P')_f$.  
In particular, 
 one has 
\begin{equation}
\label{eqn:JLT41}
  \underset{\xi \in \operatorname{Irr}(P;{\mathfrak{q}})_f}\sup\,\,
  \underset{\eta \in \operatorname{Irr}(P';{\mathfrak{q}}')_f}\sup\,\,
  \dim_{\mathbb{C}} \invHom{G'}{\operatorname{Ind}_P^G(\xi)|_{G'}}{\operatorname{Ind}_{P'}^{G'}(\eta)} <\infty.  
\end{equation}
\end{theorem}

Here we recall from Definition \ref{def:FPq} for the quantity 
 $d_{\mathfrak{q}}(\xi)$, 
 and from Definition \ref{def:Qopp}
 that 
 $Q_{\operatorname{opp}}'$ is the opposite parabolic subgroup
 of $Q'$
 in $P_{\mathbb{C}}'$.

\begin{proof}
[Proof of Theorem \ref{thm:restQQ}]
For $\eta \in \operatorname{Irr}(P')_f$, 
 we set $\eta^{\ast}:=\eta^{\vee} \otimes {\mathbb{C}}_{2\rho'}$ 
 where ${\mathbb{C}}_{2\rho'}$ is a character
 of $P'$
 defined as in \eqref{eqn:2rho}.  
Then the induced representation $\operatorname{Ind}_{P'}^{G'}(\eta^{\ast})$
 is the contragredient representation
 of $\operatorname{Ind}_{P'}^{G'}(\eta)$
 in the category ${\mathcal{M}}(G')$, 
 and one has the following natural isomorphisms:

\begin{align*}
\invHom{G'}
{\operatorname{Ind}_{P}^G(\xi)|_{G'}}
{\operatorname{Ind}_{P'}^{G'}(\eta)}
\overset \sim \to &
\invHom{G'}
{\operatorname{Ind}_{P}^G(\xi)|_{G'} \otimes
 \operatorname{Ind}_{P'}^{G'}(\eta^{\ast})}
{{\mathbb{C}}}
\\
\simeq&
\invHom{G'}
{\operatorname{Ind}_{P \times P'}^{G \times G'}(\xi \boxtimes \eta^{\ast})|_{\operatorname{diag}G'}}
{\mathbb{C}}
\\
\overset \sim \leftarrow&
\invHom{G \times G'}
{\operatorname{Ind}_{P \times P'}^{G \times G'}(\xi \boxtimes \eta^{\ast})}
{\operatorname{Ind}_{\operatorname{diag}G'}^{G \times G'}({\bf{1}})}. 
\end{align*}

Here the injectivity of the first isomorphism
 is easy, 
 and for the proof of the surjectivity, 
 see \cite[Thm.~5.4]{xksbonvec}
 for instance.  
The last isomorphism is the Frobenius reciprocity.

Then the first assertion of Theorem \ref{thm:restQQ} follows from Theorem \ref{thm:indbdd} 
 applied to the direct product group $G \times G'$ 
 because 
\begin{equation}
\label{eqn:JLTQQ}
   (Q \times Q_{\operatorname{opp}}') \backslash (G_{\mathbb{C}} \times G'_{\mathbb{C}})
   /\operatorname{diag}G_{\mathbb{C}}'
  \simeq
   Q_{\operatorname{opp}}' \backslash G_{\mathbb{C}} /Q
\end{equation}
 is a finite set.  
Moreover, 
 if $\eta \in \operatorname{Irr}(P'; {\mathfrak{q}}')_f$, 
 then the contragredient representation $\eta^{\vee}$
 belongs to $\operatorname{Irr}(P; {\mathfrak{q}}_{\operatorname{opp}}')_f$
 by Lemma \ref{lem:xidual}, 
hence so does $\eta^{\ast}$.  
Thus the second assertion also holds
 because the outer tensor product representation
$
   \xi \boxtimes \eta^{\ast}
$
 belongs to 
$
\operatorname{Irr}(P \times P'; {\mathfrak{q}} \oplus {\mathfrak{q}}_{\operatorname{opp}}')_f
$
 if $\xi \in \operatorname{Irr}(P; {\mathfrak{q}})_f$
 and $\eta \in \operatorname{Irr}(P'; {\mathfrak{q}}')_f$.  
\end{proof}

When $Q'$ is a Borel subgroup of $G_{\mathbb{C}}'$
 in Theorem \ref{thm:restQQ}, 
 one obtains the converse statement as follows.  
We recall from \eqref{eqn:OmegaPQ}
 that 
$
\Omega_{P,{\mathfrak{q}}}
=
\{\operatorname{Ind}_P^G(\xi):\xi \in \operatorname{Irr}(P;{\mathfrak{q}})_f\}
$ 
$(\subset {\mathcal{M}}(G)$), 
 and from \eqref{eqn:msup}
 that $m(\Pi|_{G'})=\underset{\pi \in \operatorname{Irr}(G')}\sup
  \dim_{\mathbb{C}} \invHom{G'}{\Pi|_{G'}}{\pi}$.

\begin{theorem}
\label{thm:Qsph}
Let $G \supset G'$ be a pair of real reductive algebraic Lie groups, 
 and $P$ a (real) parabolic subgroup of $G$.  
Suppose that $Q$ is a parabolic subgroup of $G_{\mathbb{C}}$
 such that ${\mathfrak{q}} \subset {\mathfrak{p}}_{\mathbb{C}}$.  
Then the following four conditions are equivalent:
\begin{enumerate}
\item[{\rm{(i)}}]
There exists $C>0$
such that 
\begin{equation}
\label{eqn:JLT42a}
  \dim_{\mathbb{C}} \invHom{G'}{\operatorname{Ind}_P^G(\xi)|_{G'}}{\pi} \le C d_{\mathfrak{q}}(\xi)
\end{equation}
 for any $\xi \in \operatorname{Irr}(P)_f$
 and any $\pi \in \operatorname{Irr}(G')$.  
\item[{\rm{(ii)}}]
$
  \underset{\Pi \in \Omega_{P,{\mathfrak{q}}}}\sup\,\,
  m(\Pi|_{G'})
   <\infty.  
$
\item[{\rm{(iii)}}]
$G_{\mathbb{C}}/Q$ is $G_{\mathbb{C}}'$-spherical.  
\item[{\rm{(iv)}}]
$G_{\mathbb{C}}/Q$ is $G_U'$-strongly visible
 (\cite[Def.~3.3.1]{xrims40}).  
\end{enumerate}
\end{theorem}

\begin{remark}
\begin{enumerate}
\item[(1)]
A distinguished feature of Theorem \ref{thm:Qsph} is 
 that the necessary and sufficient condition 
of the bounded multiplicity property
 is given only by the complexification
 $(G_{\mathbb{C}}, G_{\mathbb{C}}')$, 
 which traces back
 to \cite{Ksuron, xktoshima}.  
\item[(2)]
When $(G_{\mathbb{C}},G_{\mathbb{C}}')$ is a symmetric pair, 
 the parabolic subgroups $Q$ 
 satisfying the sphericity condition (iii) were classified in \cite{xhnoo}.  
See also \cite{xrims40, xtanaka12} 
 for some classification of strongly visible actions.  
\end{enumerate}
\end{remark}

We present two extreme choices
 of the parabolic subgroup $Q$, 
 namely, 
 the smallest one $Q=B$ (Borel subgroup)
 in Example \ref{ex:Qsph1}
 and the largest one $Q=P_{\mathbb{C}}$ in Example \ref{ex:Qsph2} below.  
An intermediate choice of $Q$ in Theorem \ref{thm:Qsph}
 will be crucial in Section \ref{sec:5}
 for the proof of Theorem \ref{thm:bdd}.  

\begin{example}
\label{ex:Qsph1}
When $Q$ is a Borel subgroup $B$ of $G_{\mathbb{C}}$, 
 one has $\operatorname{Irr}(P;{\mathfrak{b}})_f=\operatorname{Irr}(P)_f$
 by Lemma \ref{lem:IrrPq}, 
 hence the equivalence (ii) $\iff$ (iii) in Theorem \ref{thm:Qsph}
 gives an alternative proof
 of \cite[Thm.~D]{xktoshima}.  
\end{example}

\begin{example}
\label{ex:Qsph2}
When $Q=P_{\mathbb{C}}$, 
 Theorem \ref{thm:Qsph} implies the equivalence
 of the following three conditions on the triple $(G,P,G')$:
\begin{enumerate}
\item[{\rm{(i)}}]
$[\operatorname{Ind}_P^G(\xi)|_{G'}:\pi]\le C \dim \xi$
\quad
 for all $\xi \in \operatorname{Irr}(P)_f$
 and all $\pi \in \operatorname{Irr}(G')$;
\item[{\rm{(ii)}}]
$\underset{\chi \in \operatorname{Hom}(P,{\mathbb{C}}^{\times})}\sup
\,\,
 m(\operatorname{Ind}_P^G(\chi)|_{G'})
 <\infty
$;
\item[{\rm{(iii)}}]
$G_{\mathbb{C}}/P_{\mathbb{C}}$ is $G_{\mathbb{C}}'$-spherical.  
\end{enumerate}
See Example \ref{ex:kop} for some concrete cases
 where the branching laws of the unitary representation $L^2\operatorname{-Ind}_P^G(\chi)|_{G'}$
 are explicitly obtained
 in this framework.  
\end{example}

\begin{example}
\label{ex:46}
When $Q=P_{\mathbb{C}}$ 
 and $Q'$ is the complexification 
of a minimal parabolic subgroup $P_{\operatorname{min}}'$ of $G'$
 in Theorem \ref{thm:restQQ}, 
 one has the following
 as in the proof of Theorem \ref{thm:Qsph} below: 
{\sl{
 if $P$ is a parabolic subgroup of $G$
 satisfying 
 $\#(P_{\operatorname{min}}' \backslash G_{\mathbb{C}}/P_{\mathbb{C}})<\infty$
 then one has
\begin{equation}
\label{eqn:JLT46}
 \dim_{\mathbb{C}} \invHom {G'} {\operatorname{Ind}_P^G(\xi)|_{G'}}{\pi}<\infty
\text{ for any $\xi \in \operatorname{Irr}(P)_f$
 and $\pi \in \operatorname{Irr}(G')$.  
}
\end{equation}
}}
\end{example}

\begin{proof}
[Proof of Theorem \ref{thm:Qsph}]
(iii) $\Rightarrow$ (i).\enspace
Let $P'$ be a minimal parabolic subgroup of $G'$, 
 and $B'$ a Borel subgroup of $G_{\mathbb{C}}'$
 such that ${\mathfrak{b}}' \subset {\mathfrak{p}}_{\mathbb{C}}'$.  
By Casselman's subrepresentation theorem 
(Example \ref{ex:Casselman}), 
 for any $\pi \in \operatorname{Irr}(G')$
 there exists $\eta \in \operatorname{Irr}(P')_f$
 such that 
$\invHom{G'}{\pi}{\operatorname{Ind}_{P'}^{G'}(\eta)} \ne \{0\}$.  
Then one has an injection
$
\invHom{G'}{\operatorname{Ind}_{P}^{G}(\xi)|_{G'}}{\pi}
\hookrightarrow
\invHom{G'}{\operatorname{Ind}_{P}^{G}(\xi)|_{G'}}
           {\operatorname{Ind}_{P'}^{G'}(\eta)}.  
$
Thus the implication (iii) $\Rightarrow$ (i) follows
 as a special case of Theorem \ref{thm:restQQ}
 because 
 $\operatorname{Irr}(P')_f=\operatorname{Irr}(P';{\mathfrak{b}}')_f$
 and the number of $B'$-orbits in the $G_{\mathbb{C}}'$-spherical variety $G_{\mathbb{C}}/Q$
 is finite by a result of Brion \cite{B86} and Vinberg \cite{V86}.  
\newline
(i) $\Rightarrow$ (ii).\enspace Obvious
 because $d_{\mathfrak{q}}(\xi)=1$
 if $\xi \in \operatorname{Irr}(P;{\mathfrak{q}})_f$
 by definition.  
\newline
(ii) $\Rightarrow$ (iii).\enspace
This follows from Lemmas \ref{lem:VKN}
 and \ref{lem:210943}
 as in the proof of Theorem \ref{thm:indbdd} (2).  
\newline
(iii) $\iff$ (iv).\enspace
The equivalence (iii) $\iff$ (iv) is proved
 in \cite{xtanaka}.  
\end{proof}

\subsection{Bounded multiplicity for tensor product}
\label{subsec:tensorbdd}

The tensor product of two representations
 is regarded
 as the restriction 
of the outer tensor product representation 
 of the direct product group 
 $G \times G$ with respect to its subgroup $\operatorname{diag}G$.  
Thus the following theorem follows readily
 as a special case 
 of Theorem \ref{thm:restQQ}.  

\begin{theorem}
[\lq\lq{$QP$ estimate}\rq\rq\ for tensor product]
\label{thm:QQtensor}
Let $G$ be a real reductive algebraic Lie group, 
 and $P_j$ $(j=1,2,3)$
 (real) parabolic subgroups of $G$.  
Suppose that $Q_j$ $(j=1,2,3)$ are parabolic subgroups of $G_{\mathbb{C}}$
 such that 
 $Q_j \subset (P_j)_{\mathbb{C}}$
 $(1 \le j \le 3)$
 and 
$
   {\#}
   ((Q_1 \times Q_2)\backslash (G_{\mathbb{C}} \times G_{\mathbb{C}})
    /\operatorname{diag} (Q_3)_{\operatorname{opp}})<\infty
$.  
Then there exists $C>0$
 such that 
\[
   \dim_{\mathbb{C}} \invHom G 
   {\operatorname{Ind}_{P_1}^G(\xi_1) \otimes \operatorname{Ind}_{P_2}^G(\xi_2)}
   {\operatorname{Ind}_{P_3}^G(\xi_3)}
   \le C d_{\mathfrak{q}_1}(\xi_1) d_{\mathfrak{q}_2}(\xi_2) d_{\mathfrak{q}_3}(\xi_3)
\]
for any $\xi_j \in \operatorname{Irr}(P_j)_f$
 $(j=1,2,3)$.  
In particular, 
 one has
\begin{equation}
\label{eqn:JLT47}
   \dim_{\mathbb{C}} \invHom G 
   {\operatorname{Ind}_{P_1}^G(\xi_1) \otimes \operatorname{Ind}_{P_2}^G(\xi_2)}
   {\operatorname{Ind}_{P_3}^G(\xi_3)}
   \le C
\end{equation}
for any $\xi_j \in \operatorname{Irr}(P_j;{\mathfrak{q}}_j)_f$ 
$(1 \le j \le 3)$.  
\end{theorem}

As in \cite{xkProg2014} for instance, 
 one may reformulate Theorem \ref{thm:QQtensor}
 as a bounded multiplicity theorem for invariant trilinear forms.

\begin{theorem}
[Invariant trilinear forms]
\label{thm:trilinear}
Let $G$ be a real reductive algebraic Lie group, 
 and $P_j$ $(j=1,2,3)$
 (real) parabolic subgroups of $G$.  
Suppose that $Q_j$ $(j=1,2,3)$ are parabolic subgroups of $G_{\mathbb{C}}$
 such that 
 $Q_j \subset (P_j)_{\mathbb{C}}$
 $(1 \le j \le 3)$
 and 
$
   {\#}
   ((Q_1 \times Q_2)\backslash (G_{\mathbb{C}} \times G_{\mathbb{C}})
    /\operatorname{diag} Q_3)<\infty
$.  
Then there exists $C>0$
 such that 
\[
   \dim_{\mathbb{C}} \invHom G 
   {\operatorname{Ind}_{P_1}^G(\xi_1) \otimes \operatorname{Ind}_{P_2}^G(\xi_2)
   \otimes \operatorname{Ind}_{P_3}^G(\xi_3)}
   {\mathbb{C}}
   \le C d_{\mathfrak{q}_1}(\xi_1) d_{\mathfrak{q}_2}(\xi_2) d_{\mathfrak{q}_3}(\xi_3)
\]
for any $\xi_j \in \operatorname{Irr}(P_j)_f$
 $(j=1,2,3)$.  
In particular, 
one has
\begin{equation}
\label{eqn:tribdd}
   \dim_{\mathbb{C}} \invHom G 
   {\operatorname{Ind}_{P_1}^G(\xi_1) \otimes \operatorname{Ind}_{P_2}^G(\xi_2)
   \otimes
   \operatorname{Ind}_{P_3}^G(\xi_3)}
   {\mathbb{C}}
   \le C,   
\end{equation}
for any $\xi_j \in \operatorname{Irr}(P_j;{\mathfrak{q}}_j)_f$
 $(j=1,2,3)$. 
\end{theorem}

See {\it{e.g.,}} \cite{MWZ99, Mt15} for classification results 
 of the triples $(Q_1, Q_2, Q_3)$
 satisfying 
$
  \#((Q_1 \times Q_2) \backslash (G_{\mathbb{C}} \times G_{\mathbb{C}})/\operatorname{diag} Q_3)<\infty
$
 for some classical complex simple Lie groups $G_{\mathbb{C}}$.  
See also Example \ref{ex:tri} for some recent works 
 on integral trilinear forms
 by Clerc et al.
 \cite{CKOP11, C15}
 which fits well into the framework of Theorem \ref{thm:trilinear}.

It also deserves to discuss Theorem \ref{thm:QQtensor}
 in the special setting
 where one of $Q_1$, $Q_2$ or $Q_3$ is a Borel subgroup of $G_{\mathbb{C}}$:

\begin{theorem}
\label{thm:QQB}
Let $G$ be a real reductive algebraic Lie group, 
 $Q_1$, $Q_2$ be parabolic subgroups of $G_{\mathbb{C}}$, 
 and $G_U$ a maximal compact subgroup of $G_{\mathbb{C}}$.  
Suppose that $P_1$ and $P_2$ are real parabolic subgroups of $G$
 such that 
 $Q_j \subset (P_j)_{\mathbb{C}}$
 $(j=1,2)$.  
Then the following four conditions are equivalent:
\begin{enumerate}
\item[{\rm{(i)}}]
There exists 
 $C>0$
 such that 
\begin{equation}
\label{eqn:JLT49}
   \dim_{\mathbb{C}} \invHom G 
   {\operatorname{Ind}_{P_1}^G(\xi_1) \otimes \operatorname{Ind}_{P_2}^G(\xi_2)}
   {\pi}
   \le C d_{\mathfrak{q}_1}(\xi_1) d_{\mathfrak{q}_2}(\xi_2)
\end{equation}
for any $\xi_j \in \operatorname{Irr}(P_j)_f$ $(j=1,2)$
 and any $\pi \in \operatorname{Irr}(G)$.  

\item[{\rm{(ii)}}]
There exists 
 $C>0$
 such that 
\begin{equation}
\label{eqn:JLT49b}
   \dim_{\mathbb{C}} \invHom G 
   {\operatorname{Ind}_{P_1}^G(\xi_1) \otimes \operatorname{Ind}_{P_2}^G(\xi_2)}
   {\pi}
   \le C 
\end{equation}
for any $\xi_j \in \operatorname{Irr}(P_j;{\mathfrak{q}}_j)_f$
 $(j=1,2)$
 and any $\pi \in \operatorname{Irr}(G)$.  

\item[{\rm{(iii)}}]
$G_{\mathbb{C}}/Q_1 \times G_{\mathbb{C}}/Q_2$ is 
$G_{\mathbb{C}}$-spherical
 via the diagonal action.  
\item[{\rm{(iv)}}]
$G_{\mathbb{C}}/Q_1 \times G_{\mathbb{C}}/Q_2$ is 
$G_U$-strongly visible
 via the diagonal action.  
\end{enumerate}
\end{theorem}

A special case with $Q_j=(P_j)_{\mathbb{C}}$
 $(j=1,2)$
 implies the following:
\begin{corollary}
\label{cor:fmtensor}
Let $G$ be a real reductive algebraic Lie group, 
 and $P_j$ $(j=1,2)$ parabolic subgroups.  
Then the following five conditions
 on the triple $(G,P_1, P_2)$
 are equivalent:
\par\noindent
{\rm{(i)}}\enspace
There exists 
 $C>0$
 such that 
\begin{equation*}
   m({\operatorname{Ind}_{P_1}^G(\xi_1) \otimes \operatorname{Ind}_{P_2}^G(\xi_2)})
   \le C \dim \xi_1 \dim \xi_2, 
\quad {}^{\forall}\xi_j \in \operatorname{Irr}(P_j)_f \, (j=1,2).  
\end{equation*}
{\rm{(ii)}}\enspace
There exists 
 $C>0$
 such that 
\begin{equation}
\label{eqn:JLT4102}
   m({\operatorname{Ind}_{P_1}^G(\xi_1) \otimes \operatorname{Ind}_{P_2}^G(\xi_2)})
   \le C
\end{equation}
for all characters $\xi_j$ of $P_j$ $(j=1,2)$.  
\par\noindent
{\rm{(iii)}}\enspace
${\mathcal{O}}(G_{\mathbb{C}}/{P_1}_{\mathbb{C}}, {\mathcal{L}}_1)
 \otimes {\mathcal{O}}(G_{\mathbb{C}}/{P_2}_{\mathbb{C}}, {\mathcal{L}}_2)$
 is a multiplicity free $G_{\mathbb{C}}$-module
 for any $G_{\mathbb{C}}$-equivariant holomorphic line bundles
 ${\mathcal{L}}_j$ on $G_{\mathbb{C}}/{P_j}_{\mathbb{C}}$
 $(j=1,2)$.  
\par\noindent
{\rm{(iv)}}\enspace
$G_{\mathbb{C}}/{P_1}_{\mathbb{C}} \times G_{\mathbb{C}}/{P_2}_{\mathbb{C}}$ is $\operatorname{diag}(G_U)$-strongly visible.  

\par\noindent
{\rm{(v)}}\enspace
$G_{\mathbb{C}}/{P_1}_{\mathbb{C}} \times G_{\mathbb{C}}/{P_2}_{\mathbb{C}}$ 
 is $\operatorname{diag}(G_{\mathbb{C}})$-spherical.  
\end{corollary}

\begin{proof}
The equivalence (i) $\iff$ (ii) $\iff$ (iv) $\iff$ (v)
 is a special case
 of Theorem \ref{thm:QQB}.  
The equivalence (iii) $\iff$ (v)
 for the finite-dimensional representation theory
 follows from \cite{VK78}
 and Lemma \ref{lem:VKN}.  
\end{proof}

Littelmann \cite{Li94} classified the pairs of parabolic subgroups $({P_1}_{\mathbb{C}}, {P_2}_{\mathbb{C}})$
 such that
 $G_{\mathbb{C}}/{P_1}_{\mathbb{C}} \times G_{\mathbb{C}}/{P_2}_{\mathbb{C}}$
 are $G_{\mathbb{C}}$-spherical
 under the assumption 
 that ${P_1}_{\mathbb{C}}$ and ${P_2}_{\mathbb{C}}$
 are maximal, 
 whereas Tanaka \cite{xtanaka12}
 classified all the pairs $({P_1}_{\mathbb{C}}, {P_2}_{\mathbb{C}})$
 such that 
 $G_{\mathbb{C}}/{P_1}_{\mathbb{C}} \times G_{\mathbb{C}}/{P_2}_{\mathbb{C}}$
 is $G_U$-strongly visible.

For later applications 
 it would deserve to mention a further special case:
\begin{corollary}
\label{cor:pantensor}
Let $G$ be a real reductive algebraic Lie group, 
 and $P_1$, $P_2$ parabolic subgroups
 with abelian unipotent radical.  
Then the uniform bounded estimate \eqref{eqn:JLT4102} holds
 for any characters $\xi_j$ of $P_j$
$(j=1,2)$.  
\end{corollary}

\begin{proof}
If the unipotent radicals of $P_1$ and $P_2$ are abelian, 
 then $G_{\mathbb{C}}/{P_1}_{\mathbb{C}} \times G_{\mathbb{C}}/{P_2}_{\mathbb{C}}$
 is strongly visible 
 via the diagonal $G_U$-action
 \cite[Thm.~1.7]{K07}.  
\end{proof}

\begin{example}
Let $G=G L_n({\mathbb{R}})$
 and $P_1$, $P_2$ any two maximal parabolic subgroups of $G$.  
Then Corollary \ref{cor:pantensor} applies.  
\end{example}

\begin{example}
Let $G=S O(n,1)$ and $P$ be a minimal parabolic subgroup of $G$.  
Then Corollary \ref{cor:pantensor} applies to $P_1=P_2=P$.  
We note
 that the uniform bounded estimate \eqref{eqn:JLT4102} fails
 if we allow $\xi_j \in \operatorname{Irr}(P_j)_f$
 when $n \ge 4$ by \cite{xkProg2014}.  
\end{example}

\section{Oshima's embedding theorem --- revisited}
\label{sec:Quotient}

In this section, 
 we analyze irreducible $H$-distinguished representations of $G$
 for reductive symmetric spaces $G/H$.

The classical Casselman's subrepresentation theorem
 asserts that any irreducible representation 
 $\Pi \in \operatorname{Irr}(G)$
 can be realized
 as a subrepresentation
 of a principal series representation, 
 or equivalently, 
 as a quotient of another principal series representation.  
If $\Pi$ is $H$-distinguished, 
 there should be some constraints 
 on the parameter of the principal series representations
 depending on $H$.  
The main result of this section
 is Theorem \ref{thm:quotient}, 
 which asserts 
 that $\Pi$ is a quotient
 of some degenerate principal series $\operatorname{Ind}_{P_{G/H}}^{G}(\xi)$
 where $P_{G/H}$ is a \lq\lq{minimal parabolic subgroup}\rq\rq\
 for $G/H$
 (Definition \ref{def:Psigma}).  
We discover a useful description 
 of the constraints of $\xi$
 by using the notion
 \lq\lq{$\operatorname{Irr}(P;{\mathfrak{q}})_f$}\rq\rq\
 introduced in Definition \ref{def:FPq}
 with ${\mathfrak{q}}$ being a Borel subalgebra for $G/H$ 
 (Definition \ref{def:Borel}).  
The results of this section 
 will be used to deduce Theorem \ref{thm:bdd} from Theorem \ref{thm:indbdd}.

Throughout this section, 
 we retain our setting
 that $G$ is  contained in a connected complexification $G_{\mathbb{C}}$.  
We also assume $G$ is connected.  

\subsection{Quotient theorem for $\operatorname{Irr}(G)_H$}
\label{subsec:Quotient}

Suppose that $G/H$ is a symmetric space
 given by an involutive automorphism $\sigma$ of a connected real reductive Lie group $G$.  
We take a Cartan involution $\theta$ of $G$
 which commutes with $\sigma$, 
 and write $K$ for the corresponding maximal compact subgroup
 of $G$.

\begin{definition}
[Minimal parabolic subgroup for $G/H$]
\label{def:Psigma}
Let ${\mathfrak{a}}$ be a maximal abelian subspace
 in ${\mathfrak{g}}^{-\sigma, -\theta}
:= \{X \in {\mathfrak{g}}:\sigma(X)=\theta(X)=-X\}$, 
 and $L_{\sigma}$ the centralizer
 of ${\mathfrak{a}}$ in $G$.  
We fix a positive system $\Sigma^+({\mathfrak{g}}, {\mathfrak{a}})$, 
 and write $P_{G/H}$ or $P_{\sigma}$ 
 for the corresponding (real) parabolic subgroup of $G$.  
We say $P_{G/H}$ is
 a {\it{minimal parabolic subgroup 
 for $G/H$}}.  
We write $P_{G/H}=L_{\sigma} N_{\sigma}$ 
 and ${\mathfrak{p}}_{G/H}={\mathfrak{l}}_{\sigma} +{\mathfrak{n}}_{\sigma}$
 for the Levi decomposition 
 with ${\mathfrak{l}}_{\sigma} \supset {\mathfrak{a}}$.  
\end{definition}

\begin{remark}
(1)\enspace
We note that ${\mathfrak{a}}$ may be strictly smaller
 than the split abelian subspace of ${\mathfrak{l}}_{\sigma}$.  
\par\noindent
(2)\enspace
It seems
that the terminologies \lq\lq{Borel subalgebras
 for reductive symmetric spaces $G/H$}\rq\rq\
 (Definition \ref{def:Borel})
 and \lq\lq{minimal parabolic subgroups for $G/H$}\rq\rq\
 (Definition \ref{def:Psigma})
 are not commonly used, 
 cf.~\cite[Ex.~4.4]{K19b}.  
As Theorem \ref{thm:quotient} below suggests, 
 we believe
 that these terminologies 
 are natural generalizations
 of the classical ones 
 for the group manifold.  
\end{remark}

As we see in Section \ref{subsec:Oshima} below, 
 the following statement holds:
\begin{lemma}
\label{lem:BP}
One can take a Borel subalgebra ${\mathfrak{b}}_{G/H}$ for $G/H$
 (Definition \ref{def:Borel})
 such that ${\mathfrak{b}}_{G/H} \subset ({\mathfrak{p}}_{G/H})_{\mathbb{C}}$.  
\end{lemma}

The goal of this section is 
 to prove the following reformulation 
 of Oshima's embedding theorem
\cite[Thm.~4.15]{xoshima88b}.

\begin{theorem}
[Quotient representation theorem
 for $H$-distinguished representations]
\label{thm:quotient}
Let $G/H$ be a reductive symmetric space,
 and $P_{G/H}$ a minimal parabolic subgroup for $G/H$.  
For any $\Pi \in \operatorname{Irr}(G)_H$, 
 there exists $\xi \in \operatorname{Irr}(P_{G/H}; {\mathfrak{b}}_{G/H})_f$
 such that $\Pi$ is a quotient
 of the (degenerate) principal series representation
$
{\operatorname{Ind}_{P_{G/H}}^G(\xi)}
$
 of $G$, 
where ${\mathfrak{b}}_{G/H}$ is a Borel subalgebra
 for $G/H$
 with ${\mathfrak{b}}_{G/H} \subset ({\mathfrak{p}}_{G/H})_{\mathbb{C}}$.  
\end{theorem}

It should be noted
 that the argument of this section gives
 a \lq\lq{coarse estimate of the size}\rq\rq\
 of $\Pi \in \operatorname{Irr}(G)_H$ as follows:
\begin{proposition}
\label{prop:GKdim}
For any irreducible subquotient $\Pi$ in $C^{\infty}(G/H)$, 
 the Gelfand--Kirillov dimension, 
 to be denoted by $\operatorname{DIM}(\Pi)$, 
 satisfies
$
\operatorname{DIM}(\Pi) \le \dim G/P_{G/H}.  
$
In particular, 
 $\operatorname{DIM}(\Pi) \le \dim G/P_{G/H}$
 for any $\Pi \in \operatorname{Irr}(G)_H$.  
\end{proposition}
The proof of Theorem \ref{thm:quotient} and Proposition \ref{prop:GKdim}
 will be given in Section \ref{subsec:pfquotient}.

\subsection{Oshima's embedding theorem}
\label{subsec:Oshima}

In order to state Oshima's embedding theorem, 
 we need to prepare some notations
 and structure theorems
 of reductive symmetric spaces $G/H$, 
 which we recall now from \cite{OS84}.

Retain the setting as in Section \ref{subsec:Quotient}.  
We extend ${\mathfrak{a}}$ $(\subset {\mathfrak{g}}^{-\sigma} \cap {\mathfrak{g}}^{-\theta})$ to two maximal abelian subspaces
 ${\mathfrak{j}}$ in ${\mathfrak{g}}^{-\sigma}$
 and ${\mathfrak{a}}_{\mathfrak{p}}$ in ${\mathfrak{g}}^{-\theta}$.  
Then one has $[{\mathfrak{j}}, {\mathfrak{a}}_{\mathfrak{p}}]=\{0\}$, 
 hence there exists a $\sigma$-split and $\theta$-split Cartan subalgebra
 $\widetilde {\mathfrak{j}}$ of ${\mathfrak{g}}$
 containing both ${\mathfrak{j}}$ and ${\mathfrak{a}}_{\mathfrak{p}}$.  
As in Section \ref{subsec:CH}, 
 one has a direct sum decomposition 
 $\widetilde {\mathfrak{j}}={\mathfrak{t}} \oplus {\mathfrak{j}}$
 with ${\mathfrak{t}}:=\widetilde{\mathfrak{j}} \cap {\mathfrak{h}}$.  
Moreover one can take positive systems
 $\Delta^+({\mathfrak{g}}_{\mathbb{C}},
 \widetilde {\mathfrak{j}}_{\mathbb{C}})$, 
 $\Sigma^+({\mathfrak{g}}, {\mathfrak{a}}_{\mathfrak{p}})$
 and $\Sigma^+({\mathfrak{g}}_{\mathbb{C}}, {\mathfrak{j}}_{\mathbb{C}})$
 such that they are compatible
 with a fixed positive system $\Sigma^+({\mathfrak{g}}, {\mathfrak{a}})$
 in the sense that the restriction maps
 induce the following commutative diagram.  
\begin{alignat*}{5}
&
&&
&&\Sigma^+({\mathfrak{g}}_{\mathbb{C}}, {\mathfrak{j}}_{\mathbb{C}})\cup\{0\}
&&
&&
\\
&
&&\nearrow
&&
&&\searrow
&&
\\
&\Delta^+({\mathfrak{g}}_{\mathbb{C}},
 \widetilde {\mathfrak{j}}_{\mathbb{C}})
&&
&&
&&
&&\Sigma^+({\mathfrak{g}}, {\mathfrak{a}})\cup \{0\}
\\
&
&&\searrow
&&
&&\nearrow
&&
\\
&
&&
&&
\Sigma^+({\mathfrak{g}}, {\mathfrak{a}}_{\mathfrak{p}})\cup\{0\}
&&
&&
\end{alignat*}

We define a Borel subalgebra ${\mathfrak{b}}_{G/H}$
 for the symmetric space $G/H$
 by using the positive system $\Sigma^+({\mathfrak{g}}_{\mathbb{C}}, {\mathfrak{j}}_{\mathbb{C}})$.  
Then one has ${\mathfrak{b}}_{G/H} \subset ({\mathfrak{p}}_{G/H})_{\mathbb{C}}$.

We denote by ${\mathfrak{g}}({\mathfrak{a}}_{\mathfrak{p}};\lambda)$
 the weight space 
 for a linear form $\lambda$ on ${\mathfrak{a}}_{\mathfrak{p}}$, 
 and set ${\mathfrak{m}}:={\mathfrak{g}}({\mathfrak{a}}_{\mathfrak{p}};0) \cap 
{\mathfrak{k}}$.  
Since $G_{\mathbb{C}}$ is connected, 
 it follows from a result of Satake \cite{Sa60}
 (also \cite[p.~435]{He78})
 that the centralizer $M$ of ${\mathfrak{a}}_{\mathfrak{p}}$ in $K$
 equals $M_0 K ({\mathfrak{a}}_{\mathfrak{p}})$, 
 where $M_0$ is the identity component of $M$
 and $K({\mathfrak{a}}_{\mathfrak{p}})$ is the finite group 
 defined by $K \cap \exp(\sqrt{-1} {\mathfrak{a}}_{\mathfrak{p}})$.

We fix a $G$-invariant non-degenerate symmetric bilinear form $\langle \cdot, \cdot \rangle$
 on the Lie algebra 
 such that $\langle \cdot, \cdot \rangle$ is negative definite
 on ${\mathfrak{k}} \equiv {\mathfrak{g}}^{\theta}$, 
 positive definite on ${\mathfrak{g}}^{-\theta}$ and ${\mathfrak{g}}^{\theta} \perp {\mathfrak{g}}^{-\theta}$.  
Then $\langle \cdot, \cdot \rangle$ is non-degenerate on any $\theta$-stable subspace
 of ${\mathfrak{g}}$, 
 in particular, 
 on the centralizer ${\mathfrak{l}}_{\sigma}$
 of ${\mathfrak{a}}$ in ${\mathfrak{g}}$.  
Let ${\mathfrak{z}}({\mathfrak{l}}_{\sigma})$
 be the center of the Lie algebra ${\mathfrak{l}}_{\sigma}$.  
Obviously 
 ${\mathfrak{a}} \subset {\mathfrak{z}}({\mathfrak{l}}_{\sigma})$, 
 but we need a more detailed description of ${\mathfrak{l}}_{\sigma}$.  
We set 
\begin{align*}
   &{\mathfrak{a}}_{\sigma}:={\mathfrak{z}}({\mathfrak{l}}_{\sigma})^{-\theta}
\equiv \{Y \in {\mathfrak{z}}({\mathfrak{l}}_{\sigma}): \theta Y =-Y\}, 
\\
 &{\mathfrak{g}}(\sigma):\text{the semisimple ideal of 
 ${\mathfrak{l}}_{\sigma}$ generated by ${\mathfrak{g}}({\mathfrak{a}}_{\mathfrak{p}};\lambda)$ with $\lambda|_{\mathfrak{a}} \not \equiv 0$.  }
\\
&{\mathfrak{m}}({\sigma}):\text{the orthogonal complement of ${\mathfrak{g}}({\sigma})\oplus {\mathfrak{a}}_{\sigma}$ in $\mathfrak{l}_{\sigma}$}. 
\end{align*}
Then one has ${\mathfrak{g}}({\sigma})\subset {\mathfrak{h}}$, 
 ${\mathfrak{m}}({\sigma})\subset {\mathfrak{m}}$, 
 and ${\mathfrak{a}} \subset \mathfrak{a}_{\sigma} \subset {\mathfrak{a}}_{\mathfrak{p}}$.  
The Levi subalgebra 
 ${\mathfrak{l}}_{\sigma}$
 is decomposed into the direct sum of $\sigma$-stable ideals:
\begin{equation}
\label{eqn:lsigma}
   {\mathfrak{l}}_{\sigma}={\mathfrak{m}}({\sigma}) \oplus {\mathfrak{g}}({\sigma}) \oplus {\mathfrak{a}}_{\sigma}.  
\end{equation}

Let $G(\sigma)$, $M(\sigma)_0$, $A_{\sigma}$, and $N_{\sigma}$ be the analytic subgroups
 of $G$
 with Lie algebra ${\mathfrak{g}}({\sigma})$, 
 ${\mathfrak{m}}({\sigma})$, 
 ${\mathfrak{a}}_{\sigma}$, 
 and ${\mathfrak{n}}_{\sigma}$, 
 respectively.  
We set $M({\sigma}):=M({\sigma})_0 K({\mathfrak{a}}_{\mathfrak{p}})$.  
Accordingly to the direct decomposition \eqref{eqn:lsigma}, 
 one has 
$
   P_{G/H}=M({\sigma})G(\sigma)A_{\sigma}N_{\sigma}
$ 
by \cite[Lem.~8.12]{OS84},
 see also \cite[Lem.~1.5]{xoshima88b}.  
We introduce a subset
 of $\operatorname{Irr}(P_{G/H})_f$ as follows.

\begin{definition}
\label{def:Xi}
Let $\Xi$ be the collection of $\zeta \in \operatorname{Irr}(P_{G/H})_f$
 of the form 
\begin{equation}
\label{eqn:zetaeta}
  \zeta(m x e^{Y} n)=\eta(m)e^{\mu(Y)}
\quad
\text{for $m \in M(\sigma)$, $x \in G(\sigma)$, 
$Y \in {\mathfrak{a}}$, $n \in N_{\sigma}$, }
\end{equation}
 for some $\eta \in \operatorname{Irr}(M(\sigma))_f$
 and $\mu \in {\mathfrak{a}}_{\mathbb{C}}^{\ast}$
 such that $\eta$ has a non-zero $({\mathfrak{m}}(\sigma) \cap {\mathfrak{h}})$-fixed vector.  
\end{definition}

We are ready to state Oshima's embedding theorem 
 in a slightly different formulation from the original:
\begin{proposition}
[T.~Oshima]
\label{prop:embed}
If $\Pi \in \operatorname{Irr}(G)$
 occurs as a subquotient
 in $C^{\infty}(G/H)$, 
 then there exists $\zeta \in \Xi$
 such that 
$
  \invHom G \Pi {\operatorname{Ind}_{P_{G/H}}^{G}(\zeta)}
\ne 
\{0\}. 
$
\end{proposition}

\begin{proof}
By taking the \lq\lq{hyperfunction boundary-valued maps}\rq\rq\
 \cite[Thm.~4.15]{xoshima88b}, 
 see also \cite{xktoshima}, 
 one sees
 that there exists $\zeta \in \operatorname{Irr}(P_{G/H})_f$
 such that
 $\Pi$ is a subrepresentation
 of the (degenerate) principal series representation 
 $\operatorname{Ind}_{P_{G/H}}^G(\zeta)$
 with $\zeta$ satisfying 
 certain additional constraints,  
 which imply $\zeta \in \Xi$
 by Proposition 4.1 and (4.5), 
loc.\ cit.  
\end{proof}

Proposition \ref{prop:embed} generalizes Harish-Chandra's subquotient theorem 
 and Casselman's subrepresentation theorem:

\begin{example}
[Casselman's subrepresentation theorem]
\label{ex:Casselman}
The group manifold $(G \times G)/\operatorname{diag}G$
 is an example
 of a symmetric space.  
For any $\Pi \in \operatorname{Irr}(G)$, 
 one has a natural embedding
 of the outer tensor product representation
 $\Pi \boxtimes \Pi^{\vee}$
 into $C^{\infty}(G) \simeq C^{\infty}((G \times G)/\operatorname{diag} G)$
 by taking matrix coefficients.

Let $P_+$ be a minimal parabolic subgroup of the group $G$, 
and $P_-$ the opposite parabolic subgroup.  
Then $P_+ \times P_-$ is a minimal parabolic subgroup
 for the symmetric space $(G \times G)/\operatorname{diag}G$
 in the sense of Definition \ref{def:Psigma}.  
Therefore Proposition \ref{prop:embed} tells 
 that $\Pi \boxtimes \Pi^{\vee}$ is realized as a subrepresentation
 of some principal series representation
 of the direct product group $G \times G$, 
 hence so is $\Pi$ for $G$.  
\end{example}

\subsection{Proof of Theorem \ref{thm:quotient}
 and Proposition \ref{prop:GKdim}}
\label{subsec:pfquotient}

A key lemma to derive Theorem \ref{thm:quotient} from Proposition \ref{prop:embed} is the following.  
We recall
 that the Levi subgroup $L_{\sigma}$ of $P_{G/H}$
 is not necessarily connected.  

\begin{lemma}
\label{lem:zetat}
Let $\zeta = \zeta_1 \oplus \cdots \oplus \zeta_r$
 be the irreducible decomposition
 of $\zeta \in \Xi$
 when restricted to the identity component
 $(L_{\sigma})_0$ of the Levi subgroup $L_{\sigma}$.  
Let $\lambda_i$ 
$(\in \widetilde{\mathfrak{j}}_{\mathbb{C}}^{\ast})$ be the highest weight of $\zeta_i$
 with respect to the positive system
$
   \Delta^+({{\mathfrak{l}}}_{\sigma \mathbb{C}}, \widetilde {\mathfrak{j}}_{\mathbb{C}})
 := 
\Delta^+({{\mathfrak{g}}}_{\mathbb{C}}, \widetilde {\mathfrak{j}}_{\mathbb{C}})
\cap 
\Delta({{\mathfrak{l}}}_{\sigma \mathbb{C}}, \widetilde {\mathfrak{j}}_{\mathbb{C}})
$.  
Then for any $i$ $(1 \le i \le \lambda)$, 
$\lambda_i$ vanishes
 on the subspace ${\mathfrak{t}}=\widetilde {\mathfrak{j}} \cap {\mathfrak{h}}$.  

\end{lemma}

\begin{proof}
[Proof of Lemma \ref{lem:zetat}]
Accordingly to the ideal decomposition \eqref{eqn:lsigma}, 
 ${\mathfrak{t}}$ is given as a direct sum: 
\[
{\mathfrak{t}}
=({\mathfrak{t}} \cap {\mathfrak{m}}(\sigma))
               \oplus
({\mathfrak{t}} \cap {\mathfrak{g}}(\sigma))
\oplus
({\mathfrak{a}}_{\sigma}\cap {\mathfrak{h}}).  
\]
One sees readily from \eqref{eqn:zetaeta}
 that $\lambda_1$, $\cdots$, $\lambda_r$ vanish
 on $({\mathfrak{t}} \cap {\mathfrak{g}}(\sigma))
               \oplus
({\mathfrak{a}}_{\sigma}\cap {\mathfrak{h}})$.  
Let us prove
 that they also vanish 
 on ${\mathfrak{t}} \cap {\mathfrak{m}}(\sigma)$.  
We denote by $\eta_i$ $(1 \le i \le r)$
 the restriction of the differential representation of $\zeta_i$
 to ${\mathfrak{m}}(\sigma)$, 
 which is still irreducible.  
Since $\zeta \in \Xi$, 
 at least one of $\eta_j \in \operatorname{Irr}({\mathfrak{m}}(\sigma))_f$
 has an 
 $({\mathfrak{m}}({\sigma}) \cap {\mathfrak{h}})$-fixed vector.  
Since ${\mathfrak{m}}({\sigma})$ is $\sigma$-stable, 
 $({\mathfrak{m}}({\sigma}), {\mathfrak{m}}({\sigma}) \cap {\mathfrak{h}})$
 forms a symmetric pair, 
 and $
   \widetilde {\mathfrak{j}} \cap {\mathfrak{m}}({\sigma})
   =
   ({\mathfrak{t}} \cap {\mathfrak{m}}({\sigma})) \oplus ({\mathfrak{j}} \cap 
{\mathfrak{m}}({\sigma})) 
$
is a $\sigma$-split Cartan subalgebra
 of ${\mathfrak{m}}(\sigma)$
 such that ${\mathfrak{j}} \cap {\mathfrak{m}}(\sigma)$
 is a maximal abelian subspace in ${\mathfrak{m}}(\sigma)^{-\sigma}$.  
Applying the Cartan--Helgason theorem
 (Lemma \ref{lem:CH}) to the semisimple symmetric pair
 $({\mathfrak{m}}({\sigma}), {\mathfrak{m}}({\sigma}) \cap {\mathfrak{h}})$, 
 one obtains
$
  \lambda_j|_{{\mathfrak{t}} \cap {\mathfrak{m}}(\sigma)}=0.  
$
Hence one has $\lambda_j|_{{\mathfrak{t}}}\equiv 0$.

Since $P_{G/H} = (P_{G/H})_0 K({\mathfrak{a}}_{\mathfrak{p}})$
 and since the finite group 
$
   K({\mathfrak{a}}_{\mathfrak{p}})=K \cap \exp(\sqrt{-1}{\mathfrak{a}}_{\mathfrak{p}})
$ acts trivially
 on $\widetilde {\mathfrak{j}}$, 
 one has $\lambda_1 =\cdots = \lambda_r$.  
Thus Lemma \ref{lem:zetat} is proved.  
\end{proof}

We recall that 
 ${\mathfrak{b}}_{G/H}$ is the parabolic subalgebra
 of ${\mathfrak{g}}_{\mathbb{C}}$
 associated
 to $\sum^+({\mathfrak{g}}_{\mathbb{C}}, {\mathfrak{j}}_{\mathbb{C}})$.

\begin{lemma}
\label{lem:XiPq}
For any $\zeta \in \Xi$, 
both $\zeta$ and $\zeta^{\vee}$ belongs to 
$\operatorname{Irr}(P_{G/H};{\mathfrak{b}}_{G/H})_f$.  
\end{lemma}

\begin{proof}
We observe that the $\sigma$-split Cartan subalgebra
 $\widetilde {\mathfrak{j}}={\mathfrak{t}}+{\mathfrak{j}}$
 of ${\mathfrak{g}}$ is also 
 that of ${\mathfrak{l}}_{\sigma}$, 
 and that ${\mathfrak{b}}_{G/H} \cap {\mathfrak{l}}_{\sigma {\mathbb{C}}}$
 is a minimal parabolic subalgebra
 for the smaller symmetric pair 
 $({\mathfrak{l}}_{\sigma}, {\mathfrak{l}}_{\sigma} \cap {\mathfrak{h}})$, 
 as ${\mathfrak{b}}_{G/H} \cap {\mathfrak{l}}_{\sigma \mathbb{C}}$
 is the parabolic subalgebra 
 of ${\mathfrak{l}}_{\sigma {\mathbb{C}}}$
 associated to $\sum^+({\mathfrak{l}}_{\sigma {\mathbb{C}}}, {\mathfrak{j}}_{\mathbb{C}})$.

Suppose $\zeta \in \Xi$.  
With the notation as in Lemma \ref{lem:zetat}, 
 one has $\lambda_i|_{\mathfrak{t}_{\mathbb{C}}} \equiv 0$
 for all $i$ ($1 \le i \le r$).  
In turn, 
 it follows from Lemma \ref{lem:IrrGQ}
 that $\zeta_i$ belongs to $\operatorname{Irr}({\mathfrak{l}}_{\sigma};{\mathfrak{b}}_{G/H} \cap {\mathfrak{l}}_{\sigma {\mathbb{C}}})_f$, 
 which is naturally identified with 
 $\operatorname{Irr}({\mathfrak{p}}_{G/H};{\mathfrak{b}}_{G/H})_f$
by \eqref{eqn:IrrPL}, 
 hence $\zeta \in \operatorname{Irr}(P_{G/H};{\mathfrak{b}}_{G/H})_f$.

By Lemma \ref{lem:GHdual}, 
 if $\zeta_j$ has
 an $({\mathfrak{m}}(\sigma) \cap {\mathfrak{h}})$-fixed vector
 for some $j$, 
 then so does 
 the contragredient representation
 $\zeta_j^{\vee} \in \operatorname{Irr}({\mathfrak{l}}_{\sigma})_f$.  
Thus the same argument as in Lemma \ref{lem:zetat} 
 works, 
 and one concludes 
 $\zeta^{\vee} \in \operatorname{Irr}(P_{G/H};{\mathfrak{b}}_{G/H})_f$, 
 too.  
\end{proof}

\begin{proof}
[Proof of Theorem \ref{thm:quotient}]
Suppose $\Pi \in \operatorname{Irr}(G)_H$.  
Then the contragredient representation $\Pi^{\vee}$
 in the category ${\mathcal{M}}(G)$ can be realized
 as a subrepresentation of $C^{\infty}(G/H)$
 by the Frobenius reciprocity.  
By Oshima's embedding theorem (Proposition \ref{prop:embed}), 
 one finds $\zeta \in \Xi$
 such that $\invHom G{\Pi^{\vee}}{\operatorname{Ind}_{P_{G/H}}^G(\zeta)}\ne \{0\}$.  
Taking the dual in the category 
 ${\mathcal{M}}(G)$, 
 one has $\invHom G{\operatorname{Ind}_{P_{G/H}}^G(\zeta^{\ast})} \Pi \ne \{0\}$
 where we set $\zeta^{\ast}:=\zeta^{\vee} \otimes {\mathbb{C}}_{2\rho} \in \operatorname{Irr}(P_{G/H})_f$
 as in \eqref{eqn:2rho}.  
Since $\zeta^{\vee} \in \operatorname{Irr}(P_{G/H};{\mathfrak{b}}_{G/H})_f$
 by Lemma \ref{lem:XiPq}, 
 so does $\zeta^{\ast}$.  
Thus Theorem \ref{thm:quotient} is proved.  
\end{proof}

\begin{proof}
[Proof of Proposition \ref{prop:GKdim}]
Since any irreducible subquotient $\Pi \in C^{\infty}(G/H)$
 occurs in $\operatorname{Ind}_{P_{G/H}}^G(\zeta)$
 for some $\zeta \in \operatorname{Irr}(P_{G/H})_f$
 by Proposition \ref{prop:embed}, 
 the conclusion follows from \cite[Lem.~6.7]{xkProg2014}.  
\end{proof}

\section{Proof of Theorems \ref{thm:bdd} and \ref{thm:tensor}}
\label{sec:5}

In this section we complete the proof
 of Theorems \ref{thm:bdd} and \ref{thm:tensor}.  
We retain the assumption 
 that $G_{\mathbb{C}} \supset G_{\mathbb{C}}'$
 are connected complex Lie groups
 and that $G$ and $G'$ are real forms
 of $G_{\mathbb{C}}$ and $G_{\mathbb{C}}'$, 
 respectively.  
Since the equivalence 
 between the sphericity (ii) and the strong visibility (iii)
 is known in this setting
 by \cite{xtanaka}, 
 we shall prove
 the equivalence
 between the bounded multiplicity property (i)
 and the sphericity (ii).

\subsection{Proof of (ii) $\Rightarrow$ (i) in Theorem \ref{thm:bdd}}
Suppose we are in the setting of Theorem \ref{thm:bdd}.  
Let $B_{G/H}$ $(\subset G_{\mathbb{C}})$ be a Borel subgroup 
 for the symmetric space $G/H$
 (Definition \ref{def:Borel}).

Assume that $G_{\mathbb{C}}/{B_{G/H}}$ is $G_{\mathbb{C}}'$-spherical.  
We take a minimal parabolic subgroup $P_{G/H}$
 for the symmetric space $G/H$
 (Definition \ref{def:Psigma})
 such that ${\mathfrak{b}}_{G/H} \subset ({\mathfrak{p}}_{G/H})_{\mathbb{C}}$
 as in Lemma \ref{lem:BP}.

Suppose $\Pi \in \operatorname{Irr}(G)_H$
 and $ \pi \in \operatorname{Irr}(G')$.  
By Theorem \ref{thm:quotient}, 
 one finds $\xi \in \operatorname{Irr}(P_{G/H}; {\mathfrak{b}}_{G/H})_f$ 
 such that 
$
 \invHom G{\operatorname{Ind}_{P_{G/H}}^G (\xi)}{\Pi} \ne \{0\}
$, 
 which induces an inclusion:
\begin{align*}
\invHom{G'}
{\Pi|_{G'}}
{\pi}
\hookrightarrow
&
\invHom{G'}
{\operatorname{Ind}_{P_{G/H}}^G(\xi)|_{G'}}
{\pi}.  
\end{align*}

Thus the implication (ii) $\Rightarrow$ (i)
 in Theorem \ref{thm:bdd}
 is deduced from the implication (iii) $\Rightarrow$ (ii)
 in Theorem \ref{thm:Qsph}.

\subsection{Proof of (i) $\Rightarrow$ (ii) in Theorem \ref{thm:bdd}}
In this section, 
 we give a proof for the implication (i) $\Rightarrow$ (ii)
 in Theorem \ref{thm:bdd}
 by reducing it to the finite-dimensional case 
 as formulated in Theorem \ref{thm:bddfd} below.  

\begin{theorem}
\label{thm:bddfd}
Let $(G,H)$ be a reductive symmetric pair, 
 and $B_{G/H}$ a Borel subgroup for $G/H$ (Definition \ref{def:Borel}).  
Suppose $G'$ is an algebraic reductive subgroup of $G$. 
If $G_{\mathbb{C}}/{B_{G/H}}$ is not $G_{\mathbb{C}}'$-spherical, 
 then one has 
\[
  \underset{\Pi \in \operatorname{Irr}(G)_{H,f}}{\operatorname{sup}}
    m(\Pi|_{G'})
  =\infty.  
\]
\end{theorem}

We should remark
 that under the same assumption 
 one has 
\[
  \underset{\Pi \in \operatorname{Irr}(G;{\mathfrak{b}}_{G/H})_{f}}\sup
  m(\Pi|_{G'})=\infty, 
\]
as was seen in Lemma \ref{lem:VKN}.  
However, 
 the set $\operatorname{Irr}(G)_{H, f}$
 may be much smaller
 than $\operatorname{Irr}(G;{\mathfrak{b}}_{G/H})_{f}$, 
 see Remark \ref{rem:rankCH}
 and Lemma \ref{lem:CH}.  
Thus, 
 the proof of Theorem \ref{thm:bddfd} needs some further argument.

We recall from \eqref{eqn:CW}
 that $\Pi_{\lambda}$ denotes an irreducible finite-dimensional holomorphic representation of $G_{\mathbb{C}}$
 having highest weight $\lambda$ 
 with respect to $\Delta^+({\mathfrak{g}}_{\mathbb{C}}, \widetilde{\mathfrak{j}}_{\mathbb{C}})$.  
The same letter $\Pi_{\lambda}$ is used 
 to denote the restriction to the real form $G$.  
We begin with an elementary 
 but useful observation:
\begin{lemma}
\label{lem:increase}
Let $G_{\mathbb{C}} \supset G_{\mathbb{C}}'$ be
 a pair of connected complex reductive Lie groups.  
Suppose $\Pi_{\lambda} \in \operatorname{Irr}(G_{\mathbb{C}})_{\operatorname{hol}}$.  
Then one has
\begin{equation}
\label{eqn:JLT62}
 m(\Pi_{\lambda+\nu}|_{G_{\mathbb{C}}'})
  \ge 
 m(\Pi_{\lambda}|_{G_{\mathbb{C}}'})
\end{equation}
 for any 
$
   \Pi_{\nu} \in \operatorname{Irr}(G_{\mathbb{C}})_{\operatorname{hol}}.  
$
 Here we recall \eqref{eqn:mPi}
 for the definition of $ m(\Pi_{\lambda}|_{G_{\mathbb{C}}'})$.  
\end{lemma}

\begin{proof}
Without loss of generality, 
 we may assume
 that $G_{\mathbb{C}}$ is simply connected.  
Let $B$ be the Borel subgroup of $G_{\mathbb{C}}$
 corresponding to the positive system $\Delta^+({\mathfrak{g}}_{\mathbb{C}}, \widetilde {\mathfrak{j}}_{\mathbb{C}})$.

We lift $\lambda \in \Lambda_+$
 to a holomorphic character 
 of the Cartan subgroup of $G_{\mathbb{C}}$, 
 extend it to the opposite Borel subgroup $B_-$
 by letting the unipotent radical act trivially, 
 and then form a $G_{\mathbb{C}}$-equivariant holomorphic line bundle
 ${\mathcal{L}}_{\lambda} = G_{\mathbb{C}} \times_{B_-} {\mathbb{C}}_{\lambda}$
 over the flag variety $G_{\mathbb{C}}/B_-$
 so that the Borel--Weil theorem gives a realization of $\Pi_{\lambda}$ 
 on ${\mathcal{O}}(G_{\mathbb{C}}/B_-, {\mathcal{L}}_{\lambda})$.  
Likewise, 
 we realize $\Pi_{\nu}$ on ${\mathcal{O}}(G_{\mathbb{C}}/B_-, {\mathcal{L}}_{\nu})$.

For the subgroup $G_{\mathbb{C}}'$, 
 we take a Cartan subalgebra of ${\mathfrak{g}}_{\mathbb{C}}'$, 
 fix a positive system, 
 write the Cartan--Weyl bijection \eqref{eqn:CW}
 as $\operatorname{Irr}({\mathfrak{g}}')_f \simeq \Lambda_+'$, $\pi_{\mu} \leftrightarrow \mu$, 
 and denote by $B'$ the Borel subgroup of $G_{\mathbb{C}}'$.  
We take $\pi_{\mu} \in \operatorname{Irr}(G')_f$
 such that the multiplicity $[\Pi_{\lambda}|_{G_{\mathbb{C}}'}:\pi_{\mu}]$ attains
 its maximum $k:=m(\Pi_{\lambda}|_{G_{\mathbb{C}}'})$.  
We also take $\pi_{\tau} \in \operatorname{Irr}(G')_f$
such that $[\Pi_{\nu}|_{G_{\mathbb{C}}'}:\pi_{\tau}] \ne 0$.  
Via the Borel--Weil realization, 
 one finds holomorphic functions
 $h_j$ $(1 \le j \le k)$
 and $h \in {\mathcal{O}}(G_{\mathbb{C}})$
 corresponding to highest weight vectors
 for the subgroup $G'$
 satisfying 
\begin{align*}
  h_j (b^{-1} g q)=&\mu(b) \lambda^{-1}(q) h_j(g),
\qquad
 1 \le j \le k, 
\\
  h (b^{-1} g q)=&\tau(b) \nu^{-1}(q) h(g)
\end{align*}
for all $b\in B'$, $g \in G_{\mathbb{C}}$
 and $q \in B_-$, 
 where we use the same letters $\mu$ and $\tau$
 to denote the holomorphic characters of $B'$.  
This shows
 that $h_j h$ $(1 \le j \le k)$ belong
 to ${\mathcal{O}}(G_{\mathbb{C}}/B_-, {\mathcal{L}}_{\lambda+\nu})$
 on which $G_{\mathbb{C}}$ acts as the irreducible representation $\Pi_{\lambda+\nu}$,
 and they are highest weight vectors
 for the subgroup $G_{\mathbb{C}}'$
 with the some highest weight $\mu+\tau$.  
Since $h_j h$ ($1 \le j \le k$) are linearly independent, 
 one concludes $[\Pi_{\lambda+\nu}|_{G_{\mathbb{C}}'}:\pi_{\mu + \tau}] \ge k$ $(=m(\Pi_{\lambda}|_{G_{\mathbb{C}}'}))$.  
Thus the lemma is proved.  
\end{proof}

\begin{proof}
[Proof of Theorem \ref{thm:bddfd}]
Since neither of the assumption nor the conclusion changes
 if we replace $G$
 by the identity component
 of the derived subgroup $[G,G]$, 
 we may and do assume $G$ is a connected semisimple Lie group.  
Moreover, 
 we can assume $G'$ is connected.  
We take compatible positive systems
 $\Delta^+({\mathfrak{g}}_{\mathbb{C}},\widetilde {\mathfrak{j}}_{\mathbb{C}})$
 and $\Sigma^+({\mathfrak{g}}_{\mathbb{C}},{\mathfrak{j}}_{\mathbb{C}})$
 as in Section \ref{subsec:Satake}, 
 and define a parabolic subalgebra ${\mathfrak{q}}$
 associated to $\Sigma^+({\mathfrak{g}}_{\mathbb{C}},{\mathfrak{j}}_{\mathbb{C}})$
 as in Definition \ref{def:Borel}. 
Then ${\mathfrak{q}}$ is a Borel subalgebra
 for the symmetric space $G/H$.  
We denote by ${\mathfrak{l}}$ 
 the centralizer of ${\mathfrak{j}}$ 
in ${\mathfrak{g}}$.

We begin with the case 
 where $G$ is contained
 in a simply connected complexification $G_{\mathbb{C}}$.  
Let $Q$ be the parabolic subgroup
 of $G_{\mathbb{C}}$
 with Lie algebra ${\mathfrak{q}}$, 
 which is a Borel subgroup for $G/H$.

Suppose $G_{\mathbb{C}}/Q$ is not $G_{\mathbb{C}}'$-spherical.  
By Lemma \ref{lem:VKN}, 
 one finds $\lambda \in \Lambda_+$
 such that $\Pi_{N\lambda} \in \operatorname{Irr}(G_{\mathbb{C}};{\mathfrak{q}})_{\operatorname{hol}}$
 and 
 $m(\Pi_{N \lambda}|_{G_{\mathbb{C}}'}) \ge N+1$ for all $N \in {\mathbb{N}}$.  We recall 
$\widetilde{\mathfrak{j}}
=
 \widetilde{\mathfrak{j}}^{\sigma} 
\oplus 
 \widetilde{\mathfrak{j}}^{-\sigma}
={\mathfrak{t}} + {\mathfrak{j}}
$
 is a $\sigma$-split Cartan subalgebra 
 and that the positive system $\Delta^+({\mathfrak{g}}_{\mathbb{C}}, \widetilde{\mathfrak{j}}_{\mathbb{C}})$ is compatible with $\Sigma^+({\mathfrak{g}}_{\mathbb{C}}, {\mathfrak{j}}_{\mathbb{C}})$.  
We also recall
 that any element $\nu \in \Lambda_+({\mathfrak{g}}_{\mathbb{C}};{\mathfrak{h}}_{\mathbb{C}})$
 satisfies $\sigma \nu=-\nu$
 by the definition \eqref{eqn:JLTCH}.  
We take $\nu \in \Lambda_+({\mathfrak{g}}_{\mathbb{C}};{\mathfrak{h}}_{\mathbb{C}})$
which is regular enough, 
 namely, 
 $2 \langle \nu, \alpha \rangle \ge \langle \sigma \lambda, \alpha \rangle$
 for all $\alpha \in \Delta^+({\mathfrak{g}}_{\mathbb{C}}, \widetilde {\mathfrak{j}}_{\mathbb{C}})$
 such that $\alpha|_{\mathfrak{j}} \ne 0$.  
On the other hand, 
 for $\alpha \in \Delta^+({\mathfrak{g}}_{\mathbb{C}}, \widetilde {\mathfrak{j}}_{\mathbb{C}})$
 with $\alpha|_{\mathfrak{j}} = 0$, 
 one has $\langle \sigma \lambda, \alpha \rangle = \langle \lambda, \alpha \rangle=0$
because such $\alpha$ is regarded as an element of $\Delta({\mathfrak{l}}_{\mathbb{C}}, \widetilde {\mathfrak{j}}_{\mathbb{C}})$
 whereas $\lambda$ vanishes on $[{\mathfrak{l}}_{\mathbb{C}}, {\mathfrak{l}}_{\mathbb{C}}]$
 as the differential of a character of $Q$.  
Thus $\langle 2 \nu - \sigma \lambda, \alpha \rangle \ge 0$
 for all 
$
   \alpha 
   \in 
\Delta^+({\mathfrak{g}}_{\mathbb{C}}, \widetilde{\mathfrak{j}}_{\mathbb{C}}). 
$   
We define a dominant integral weight by 
\[
      \Lambda:=\lambda+\nu-\sigma(\lambda+\nu)=\lambda+(2 \nu-\sigma \lambda) 
\]
which vanishes on ${\mathfrak{t}}$ $(= \widetilde {\mathfrak{j}}^{\sigma})$.  
Thus one has 
$
   2\Lambda
   \in \Lambda_+({\mathfrak{g}}_{\mathbb{C}};{\mathfrak{h}}_{\mathbb{C}}).  
$
Moreover
 the irreducible representation $\Pi_{2N\Lambda}$ of $G_{\mathbb{C}}$
 satisfies
$
  m(\Pi_{2N\Lambda}|_{G_{\mathbb{C}}'}) \ge m(\Pi_{2N\lambda}|_{G_{\mathbb{C}}'})\ge 2 N+1
$
 for any $N \in {\mathbb{N}}$
 when restricted to the subgroup $G_{\mathbb{C}}'$
 by Lemma \ref{lem:increase}.

When $G_{\mathbb{C}}$ is not simply connected, 
 we replace $\Lambda$ by $k \Lambda$
 where $k$ is a positive integer given in Lemma \ref{lem:CH}.  
Then by the Cartan--Helgason theorem 
(Lemma \ref{lem:CH}), 
 $\Pi_{2N\Lambda} \in  \operatorname{Irr}(G)_H$.  
Thus Theorem \ref{thm:bddfd} is shown.  
\end{proof}

\subsection{Proof of Theorem \ref{thm:tensor}}
\label{subsec:tensorpf}

Suppose we are in the setting of Theorem \ref{thm:tensor}.  
We observe
 that the direct product group $B_{G/H_1} \times B_{G/H_2}$
 is a Borel subgroup 
 for the reductive symmetric space
 $(G \times G)/(H_1 \times H_2)$, 
 see Definition \ref{def:Borel}.

We also observe 
 that the restriction of the outer tensor product representation
 $\Pi_1 \boxtimes \Pi_2$
 of $G \times G$
 to its subgroup
 $\operatorname{diag} G$
 $(\simeq G)$
 is nothing but the tensor product representation
 $\Pi_1 \otimes \Pi_2$.  
Hence Theorem \ref{thm:tensor} follows 
 as a special case of Theorem \ref{thm:bdd}
 for the restriction $(G \times G) \downarrow \operatorname{diag} G$.

\section{Classification of the triples $(G,H,G')$}
\label{sec:classification}
Problem \ref{q:bdd} asks a criterion 
 for the triple 
 $H \subset G \supset G'$
 having the bounded multiplicity property \eqref{eqn:BBH}
 of the restriction, 
 namely, 
\begin{equation}
\label{eqn:BBHsup}
   \underset{\Pi \in \operatorname{Irr}(G)_H}\sup
\,\,
   \underset{\pi \in \operatorname{Irr}(G')}\sup
   [\Pi|_{G'}:\pi]<\infty, 
\end{equation}
where we recall
 $[\Pi|_{G'}:\pi]=\dim_{\mathbb{C}} \invHom{G'}{\Pi|_{G'}}{\pi}$
 from \eqref{eqn:JLTp2}.

In this section 
 we give a classification 
 of the triples $H \subset G \supset G'$
 in the setting 
 where $(G,H)$ and $(G,G')$ are symmetric pairs. 
In order to state the classification, 
 let us consider what will be a natural equivalence relation
 on triples $(G,H,G')$ for this problem.  
First we observe
 that the bounded multiplicity property \eqref{eqn:BBHsup}
 does not change by taking finite coverings
 or the identity components of the groups.  
Second, 
 it does not change also 
 by replacing the two subgroups $H$ and $G'$ simultaneously 
  with $\sigma H$ and $\sigma G'$, 
respectively, 
 by an automorphism $\sigma$ of the group $G$.  
Third, 
 we observe a $G$-equivalence of the regular representations
 on $C^{\infty}(G/H) \simeq C^{\infty}(G/a H a^{-1})$
 for any $a \in G$
 and a $G'$-equivalence of the restriction 
 $\Pi|_{G'} \simeq \Pi|_{b G'b^{-1}}$
 via the inner automorphism $G' \simeq b G'b^{-1}$
 for any $b \in G$.  
Hence we shall adopt the following definition
 of the infinitesimal equivalence
 in our classification 
 of the triples $(G,H,G')$
 satisfying the bounded multiplicity property \eqref{eqn:BBHsup}.  

\begin{definition}
[equivalence of the triple $(G,H,G')$]
\label{def:eqtriple}
We say the triples
 $H \subset G \supset G'$
 and $\widetilde H \subset \widetilde G \supset \widetilde G'$
are {\it{infinitesimally equivalent}}
 if there is an isomorphism
${\mathfrak{g}} \simeq \widetilde{\mathfrak{g}}$
 between the Lie algebras of $G$ and $\widetilde G$, 
 and if via this identification, 
 ${\mathfrak{h}}$ is conjugate to $\widetilde{\mathfrak{h}}$
 by an inner automorphism
 and ${\mathfrak{g}}'$ is conjugate to $\widetilde{\mathfrak{g}}'$
 by another inner automorphism.  
\end{definition}

The above definition says, 
 in particular, 
 that the triples
 $H_1 \subset G \supset G_1'$
 and $H_2 \subset G \supset G_2'$
 are infinitesimally equivalent if 
there exist $a,b \in G$ and $\sigma \in \operatorname{Aut}({\mathfrak{g}})$
 such that 
$
  {\mathfrak{h}}_2 = \sigma \operatorname{Ad}(a){\mathfrak{h}}_1
$
 and 
$
  {\mathfrak{g}}_2'= \sigma \operatorname{Ad}(b){\mathfrak{g}}_1'.  
$
We note that $\sigma$ may be an outer automorphism.

In what follows, 
 we shall assume
 that at least 
 one of the symmetric spaces $G/H$ or $G/G'$ is irreducible.  
This means that we shall treat the following cases:
\par\noindent
{\bf{Case I.}}\enspace
\hphantom{I}\,\,\,$G$ is simple, 
\par\noindent
{\bf{Case II.}}\hphantom{ii}
$G/H$ is a group manifold $({}^{\backprime} G \times {}^{\backprime} G)/\operatorname{diag}{}^{\backprime} G$ for simple ${}^{\backprime} G$, 
\par\noindent
{\bf{Case III.}}\enspace
$G/G'$ is a group manifold $({}^{\backprime} G \times {}^{\backprime} G)/\operatorname{diag}{}^{\backprime} G$ for simple ${}^{\backprime} G$.  
\par
Case II deals with the representation 
 $\Pi = \tau \boxtimes \tau^{\vee}$
 of $G={}^{\backprime}G \times {}^{\backprime}G$
 for some $\tau \in \operatorname{Irr}({}^{\backprime} G)$, 
 whereas the restriction of representations
 $\Pi|_{G'}$ in Case III is nothing 
 but the tensor product 
 of two irreducible representations
 $\Pi_1$ and $\Pi_2$ of ${}^{\backprime}G$
 when $\Pi \in \operatorname{Irr}(G)$ is of the form
 $\Pi_1 \boxtimes \Pi_2$.  
We note that there is an overlap 
between Case II and Case III.

Finally, 
 our criterion in Theorem \ref{thm:bdd}
 tells that it suffices to classify the triples
 of the {\it{complexified Lie algebras}}
 $({\mathfrak{g}}_{\mathbb{C}}, {\mathfrak{h}}_{\mathbb{C}}, {\mathfrak{g}}_{\mathbb{C}}')$.  
With these observations, 
 the classification of the triples $(G,H,G')$
 will be given in Theorems \ref{thm:listreal}
 and \ref{thm:cpxlist} for Case I, 
 in Theorem \ref{thm:gplist} for Case II, 
 and in Theorem \ref{thm:tensorlist} for Case III.

The proof of these theorems will be given in Section \ref{sec:pfclass}.  

\subsection{Classification of $(G,H,G')$ with $G$ simple}
In this subsection 
 we deal with the case
 that $G$ is a simple Lie group.  
The classification 
of the triples $(G, H, G')$
 having the bounded multiplicity property
 is given in Theorem \ref{thm:listreal}
 when 
 ${\mathfrak{g}}$ is not a complex Lie algebra
 (equivalently, the complexification ${\mathfrak{g}}_{\mathbb{C}}$ is simple)
 and Theorem \ref{thm:cpxlist}
 when ${\mathfrak {g}}$ is a complex simple Lie algebra
  up to the infinitesimal equivalence 
 in Definition \ref{def:eqtriple}.

\begin{theorem}
\label{thm:listreal}
Suppose that ${\mathfrak{g}}_{\mathbb{C}}$ is simple
 and that $(G,H)$ and $(G,G')$ are symmetric pairs.  
Then the bounded multiplicity property \eqref{eqn:BBHsup}
 holds for the triple $(G,H,G')$ 
 if and only if the complexified Lie algebras 
 $({\mathfrak{g}}_{\mathbb{C}}, {\mathfrak{h}}_{\mathbb{C}}, {\mathfrak{g}}_{\mathbb{C}}')$ 
 are in Table \ref{tab:0.1}.  
In the table, 
 $p$, $q$ are arbitrary subject to $n=p+q$.

\begin{table}[H]
\begin{minipage}[t]{.45\textwidth}
\begin{tabular}[t]{ccc}
${\mathfrak{g}}_{\mathbb{C}}$
&${\mathfrak{h}}_{\mathbb{C}}$
&${\mathfrak{g}}_{\mathbb{C}}'$
\\
\hline
${\mathfrak{sl}}_n$
&${\mathfrak{gl}}_{n-1}$
&${\mathfrak{sl}}_p \oplus {\mathfrak{sl}}_q \oplus {\mathbb{C}}$
\\
${\mathfrak{sl}}_{2m}$
&${\mathfrak{gl}}_{2m-1}$
&${\mathfrak{sp}}_m$
\\
${\mathfrak{sl}}_{6}$
&${\mathfrak{sp}}_{3}$
&${\mathfrak{sl}}_4 \oplus {\mathfrak{sl}}_2 \oplus {\mathbb{C}}$
\\
${\mathfrak{so}}_n$
&${\mathfrak{so}}_{n-1}$
&${\mathfrak{so}}_p \oplus {\mathfrak{so}}_q$
\\
${\mathfrak{so}}_{2m}$
&${\mathfrak{so}}_{2m-1}$
&${\mathfrak{gl}}_m$
\\
${\mathfrak{so}}_{2m}$
&${\mathfrak{so}}_{2m-2} \oplus {\mathbb{C}}$
&${\mathfrak{gl}}_m$
\\
${\mathfrak{sp}}_n$
&${\mathfrak{sp}}_{n-1} \oplus {\mathfrak{sp}}_1$
&${\mathfrak{sp}}_p \oplus {\mathfrak{sp}}_q$
\\
${\mathfrak{sp}}_n$
&${\mathfrak{sp}}_{n-2} \oplus {\mathfrak{sp}}_2$
&${\mathfrak{sp}}_{n-1} \oplus {\mathfrak{sp}}_1$
\\
${\mathfrak{e}}_6$
&${\mathfrak{f}}_4$
&${\mathfrak{so}}_{10} \oplus {\mathbb{C}}$
\\
${\mathfrak{f}}_4$
&${\mathfrak{so}}_{9}$
&${\mathfrak{so}}_9$
\\
\end{tabular}
  \end{minipage}
  \hfill
  \begin{minipage}[t]{.45\textwidth}
\begin{tabular}[t]{ccc}
${\mathfrak{g}}_{\mathbb{C}}$
&${\mathfrak{h}}_{\mathbb{C}}$
&${\mathfrak{g}}_{\mathbb{C}}'$
\\
\hline
${\mathfrak{sl}}_n$
&${\mathfrak{so}}_{n}$
&${\mathfrak{gl}}_{n-1}$
\\
${\mathfrak{sl}}_{2m}$
&${\mathfrak{sp}}_{m}$
&${\mathfrak{gl}}_{2m-1}$
\\
${\mathfrak{sl}}_n$
&${\mathfrak{sl}}_p \oplus {\mathfrak{sl}}_q \oplus {\mathbb{C}}$
&${\mathfrak{gl}}_{n-1}$
\\
${\mathfrak{so}}_{n}$
&${\mathfrak{so}}_{p} \oplus {\mathfrak{so}}_q$
&${\mathfrak{so}}_{n-1}$
\\
${\mathfrak{so}}_{2m}$
&${\mathfrak{gl}}_{m}$
&${\mathfrak{so}}_{2m-1}$
\\
\end{tabular}
  \end{minipage}
\caption{Triples $({\mathfrak{g}}_{\mathbb{C}}, {\mathfrak{h}}_{\mathbb{C}}, {\mathfrak{g}}_{\mathbb{C}}')$ with ${\mathfrak{g}}_{\mathbb{C}}$ simple in Theorem \ref{thm:bdd}}
\label{tab:0.1}
\hfil
\end{table}

\end{theorem}

\begin{remark}
\label{rem:Table}
Since the classification is given by the infinitesimal equivalence
 given in Definition \ref{def:eqtriple}, 
 we have omitted some cases such as
\begin{alignat*}{3}
&({\mathfrak{g}}_{\mathbb{C}}, {\mathfrak{g}}_{\mathbb{C}}')
&&=({\mathfrak{so}}_8, {\mathfrak{spin}}_7)
&&\sim ({\mathfrak{so}}_8, {\mathfrak{so}}_7), 
\\
&({\mathfrak{g}}_{\mathbb{C}}, {\mathfrak{h}}_{\mathbb{C}}, {\mathfrak{g}}_{\mathbb{C}}')
&&=({\mathfrak{so}}_8, {\mathfrak{gl}}_4, {\mathfrak{so}}_6 \oplus {\mathfrak{so}}_2)
&&\sim ({\mathfrak{so}}_8, {\mathfrak{so}}_6 \oplus {\mathfrak{so}}_2, {\mathfrak{gl}}_4), 
\\
&({\mathfrak{g}}_{\mathbb{C}}, {\mathfrak{h}}_{\mathbb{C}}, {\mathfrak{g}}_{\mathbb{C}}')
&&=({\mathfrak{sl}}_4, {\mathfrak{sp}}_2, {\mathfrak{sl}}_2 \oplus {\mathfrak{sl}}_2 \oplus {\mathbb{C}})
&&\sim ({\mathfrak{so}}_6, {\mathfrak{so}}_5, {\mathfrak{so}}_4 \oplus {\mathfrak{so}}_2), 
\end{alignat*}
because the right-hand sides appear
 as special cases of the more general family
 in Table \ref{tab:0.1}.  
\end{remark}

{}From the classification in Theorem \ref{thm:listreal}, 
 one obtain the following:
\begin{corollary}
\label{cor:rankone}
Suppose $G/H$ is a symmetric space
 of rank one, 
 and ${\mathfrak{g}}_{\mathbb{C}}$ is simple.  
Then for any symmetric pair $(G,G')$ the bounded multiplicity property \eqref{eqn:BBHsup} holds
except for the following two cases:
\[\text{
 $({\mathfrak{g}}_{\mathbb{C}}, {\mathfrak{h}}_{\mathbb{C}}, {\mathfrak{g}}_{\mathbb{C}}')
=
 ({\mathfrak{sl}}_n, {\mathfrak{gl}}_{n-1}, {\mathfrak{so}}_{n})$, 
$({\mathfrak{sp}}_n, {\mathfrak{sp}}_{n-1} \oplus {\mathfrak{sp}}_1,
 {\mathfrak{gl}}_{n})$, 
or 
 $({\mathfrak{f}}_4, {\mathfrak{so}}_9, {\mathfrak{sp}}_3 \oplus {\mathfrak{sl}}_2)$.  }
\]

\end{corollary}

We shall examine in Example \ref{ex:SU3}
 the failure of \eqref{eqn:BBHsup}
 in the simplest exceptional case
 $({\mathfrak{sl}}_3, {\mathfrak{gl}}_2, {\mathfrak{so}}_3)$
 by an explicit computation
 of multiplicities.

Theorem \ref{thm:listreal} was announced in \cite[Cor.~7.8]{K21}.  
The right-hand side of Table \ref{tab:0.1} collects
 the case \eqref{eqn:BBlist}
 where a stronger bounded multiplicity property \eqref{eqn:BB} holds 
 for the restriction $G \downarrow G'$ independently of $H$, 
 see \cite[Thm.~D]{xktoshima}.  
The left-hand side includes:
\begin{example}
\label{ex:Opq}
Suppose $p_1 + p_2=p$ and $q_1 + q_2 =q$.  
We set 
$G/H:=O(p,q)/O(p,q-1)$, 
$
  G:'=O(p_1,q_1) \times O(p_2, q_2)
$.  
Then the triple $(G,H,G')$ satisfies the bounded multiplicity property
 \eqref{eqn:BBH}
 as the triple 
$({\mathfrak{g}}_{\mathbb{C}}, 
{\mathfrak{h}}_{\mathbb{C}}, 
{\mathfrak{g}}_{\mathbb{C}}')$
of the complexified Lie algebras
 occurs 
 in the fourth row of the left-hand column in Table \ref{tab:0.1}.  
The branching laws of the restriction $\Pi|_{G'}$ 
 of a discrete series representation $\Pi$ for $G/H$
 have been studied
 in \cite{K93, K21, MO15, OS19}, 
see Example \ref{ex:Opqanalysis}
 for some more details.  
\end{example}

Next we consider the classification when $G$ has a complex structure.  

\begin{theorem}
\label{thm:cpxlist}
Suppose that $G$ is a complex simple Lie group, 
 and that both $(G,H)$ and $(G,G')$ are symmetric pairs.  
Then the bounded multiplicity property \eqref{eqn:BBHsup} holds
 if and only if one of the following conditions hold:

\begin{enumerate}
\item[{\bf{a.}}]
Both ${\mathfrak{h}}$ and ${\mathfrak{g}}'$ are complex Lie subalgebras
 and the triple $({\mathfrak{g}}, {\mathfrak{h}}, {\mathfrak{g}}')$
 are in Table \ref{tab:0.1}.  

\item[{\bf{b.}}]
$({\mathfrak{g}}, {\mathfrak{h}})=({\mathfrak{s o}}_n({\mathbb{C}}), {\mathfrak{s o}}_{n-1}({\mathbb{C}}))$
 and ${\mathfrak{g}}'$ is any real form of ${\mathfrak{g}}$. 

\item[{\bf{c.}}]
$({\mathfrak{g}}, {\mathfrak{g}}')$ is 
$({\mathfrak{s l}}_n({\mathbb{C}}), {\mathfrak{g l}}_{n-1}({\mathbb{C}}))$
 $(n \ge 2)$
 or $({\mathfrak{s o}}_n({\mathbb{C}}), {\mathfrak{s o}}_{n-1}({\mathbb{C}}))$
 $(n \ge 5)$, 
 and ${\mathfrak{h}}$ is any real form of ${\mathfrak{g}}$.  

\item[{\bf{d.}}]
${\mathfrak{g}}={\mathfrak{s l}}_2({\mathbb{C}})$, 
 and both ${\mathfrak{h}}$ and ${\mathfrak{g}}'$ are any real forms 
 of ${\mathfrak{g}}$.  
\end{enumerate}
\end{theorem}

\begin{remark}
As in Remark \ref{rem:Table}, 
 we omit some triples $(G,H,G')$
 for which the infinitesimally equivalent
 ones are in the list.  
\end{remark}

\subsection{Classification of $(G,H,G')$: group manifold case}
In this subsection 
 we give a classification of the triples $(G,H,G')$
 having the bounded multiplicity property \eqref{eqn:BBH}
 in the setting where $G/H$ is a group manifold,  
 namely, 
 $G$ is the direct product group
 ${}^{\backprime} G \times {}^{\backprime} G$
 for some simple Lie group ${}^{\backprime} G$
 and $H$ is of the form $\operatorname{diag}{}^{\backprime} G$.  
We refer to this
 as the \lq\lq{group manifold case}\rq\rq.

In what follows, 
 we use the letter $G$
 to denote a simple Lie group
 instead of ${}^{\backprime} G$.  
Our task now is to classify the symmetric pair $(G \times G, G')$
 such that the following bounded multiplicity property holds:
\begin{equation}
\label{eqn:gpbdd}
\underset{\Pi \in \operatorname{Irr}(G \times G)_{\operatorname{diag}G}}\sup
\,\,
\underset{\pi \in \operatorname{Irr}(G')}\sup
[\Pi|_{G'}:\pi] <\infty, 
\end{equation}
where we recall 
$[\Pi|_{G'}:\pi] =\dim_{\mathbb{C}} \operatorname{Hom}_{G'}(\Pi|_{G'}, \pi)$.  
Here we note that 
$
   \Pi \in \operatorname{Irr}(G \times G)_{\operatorname{diag}G}
$
 if and only if $\Pi$ is of the form $\tau \boxtimes \tau^{\vee}$
 for some $\tau \in \operatorname{Irr}(G)$
 where $\tau^{\vee}$ is the contragredient representation of $\tau$
 in the category ${\mathcal{M}}(G)$.

Up to the infinitesimal equivalence as in Definition \ref{def:eqtriple}, 
 there are two possibilities
 for the subgroup $G'$ in $G \times G$:
\newline\indent
{\bf{Case II-1.}}\enspace
$G'=G_1 \times G_2$
 where $(G, G_j)$
 $(j=1,2)$
 are symmetric pairs.  
\newline\indent
{\bf{Case II-2.}}\enspace
$G'= \operatorname{diag}_{\sigma}(G)$
 where $\sigma$ is an involutive automorphism
 ${}^{\backprime}G$.

We shall see
 that Case II-2 occurs in the classification below
 only when either ${\mathfrak{g}}$ or ${\mathfrak{g}}_{\mathbb{C}}$
 is isomorphic to ${\mathfrak{sl}}_2({\mathbb{C}})$.  

\begin{theorem}
[group manifold case]
\label{thm:gplist}
Let $G$ be a simple Lie group.  

\begin{enumerate}
\item[{\rm{(1)}}]
Suppose $G'=G_1 \times G_2$ 
 where $(G, G_j)$ $(j=1,2)$ are symmetric pairs.  
Then the bounded multiplicity property \eqref{eqn:gpbdd} holds
 if and only if 
 $({\mathfrak{g}}, {\mathfrak{g}}_1, {\mathfrak{g}}_2)$ or
$({\mathfrak{g}}_{\mathbb{C}}, {{\mathfrak{g}}_1}_{\mathbb{C}}, {{\mathfrak{g}}_2}_{\mathbb{C}})$
 are isomorphic to 
$
({\mathfrak{s l}}_n, {\mathfrak{g l}}_{n-1}, {\mathfrak{g l}}_{n-1})
\text{ or }
({\mathfrak{s o}}_{n}, {\mathfrak{s o}}_{n-1}, {\mathfrak{s o}}_{n-1}).  
$

\item[{\rm{(2)}}]
Suppose $G':= \operatorname{diag}_{\sigma}(G)$
 where $\sigma$ is an involutive automorphism of $G$.  
Then \eqref{eqn:gpbdd} holds 
 if and only if
 ${\mathfrak{g}}={\mathfrak{s l}}_2({\mathbb{C}})$
 or ${\mathfrak{g}}_{\mathbb{C}}={\mathfrak{s l}}_2({\mathbb{C}})$
 ($\sigma$ is arbitrary).  
\end{enumerate}
\end{theorem}

Theorem \ref{thm:gplist} (1) includes the following assertion:
{\sl{when $G$ is a complex Lie group
 and $G'=G_1 \times G_2$ with at least one of $G_j$ being a real form of $G$, 
 then the bounded multiplicity property \eqref{eqn:gpbdd} holds 
 if and only if ${\mathfrak{g}}={\mathfrak{sl}}_2({\mathbb{C}})$.  }}

\subsection{Classification of $(G,H,G')$: tensor product case}

In this subsection 
 we give a classification of the triples $(G,H,G')$
 having the bounded multiplicity property \eqref{eqn:BBH}
 in Case III, 
 that is, 
 in the setting
 where $(G,G')=({}^{\backprime} G \times {}^{\backprime} G, \operatorname{diag} {}^{\backprime} G)$
 for some simple Lie group ${}^{\backprime} G$.  
In this case, 
 the restriction of an irreducible representation of $G$ to the subgroup $G'$
 is nothing but the tensor product
 of two irreducible representations of ${}^{\backprime} G$.  
We refer to this 
 as the \lq\lq{tensor product case}\rq\rq.

In what follows, 
 we use the letter $G$ instead of ${}^{\backprime} G$.  
Our task now
 is to classify the symmetric pairs
 $(G \times G, H)$ 
 having the following bounded multiplicity property:
$
   \underset{(\Pi_1, \Pi_2)}\sup\,\,\,
   m(\Pi_1 \otimes \Pi_2) <\infty
$, 
 see \eqref{eqn:JLTp5}, 
 namely, 
\begin{equation}
\label{eqn:gtensor}
\underset{(\Pi_1, \Pi_2)}\sup\,\,\,
\underset{\Pi \in \operatorname{Irr}(G)} \sup
\dim_{\mathbb{C}} \operatorname{Hom}_{G}
(\Pi_1 \otimes \Pi_2, \Pi) <\infty
\end{equation}
where the first supremum is taken
 over all the pairs $(\Pi_1, \Pi_2)$
 with $\Pi_1, \Pi_2 \in \operatorname{Irr}(G)$
 subject to $\Pi_1 \boxtimes \Pi_2 \in \operatorname{Irr}(G \times G)_H$.

Up to the infinitesimal equivalence 
 given as in Definition \ref{def:eqtriple}, 
 there are two possibilities of the subgroup $H$ in $G \times G$:
\newline\indent
{\bf{Case III-1.}}\enspace
$H=H_1 \times H_2$ where $(G, H_j)$
 $(j=1,2)$
 are symmetric pairs, 
\newline\indent
{\bf{Case III-2.}}\enspace
$H$ is of the form $\operatorname{diag}_{\sigma}(G)$
 where $\sigma$ is an involution of $G$.

We note that there is an overlap
 between Case II-2 and Case III-2.  
We shall see in Theorem \ref{thm:tensorlist} below that 
 Case III-2 occurs in the classification
 only when either ${\mathfrak{g}}$
 or ${\mathfrak{g}}_{\mathbb{C}}$
 is isomorphic to ${\mathfrak{sl}}_2({\mathbb{C}})$.  

\begin{theorem}
[Tensor product case]
\label{thm:tensorlist}
Suppose that $G$ is a simple Lie group, 
 and $(G \times G, H)$ is a symmetric pair.  
Then the bounded multiplicity property \eqref{eqn:gtensor} holds
 if and only if one of the following conditions hold:
\begin{enumerate}
\item[{\rm{(1)}}]
Either ${\mathfrak{g}}$ or ${\mathfrak{g}}_{\mathbb{C}}$ is
 ${\mathfrak{s l}}_2({\mathbb{C}})$.  
No condition on ${\mathfrak{h}}$.  

\item[{\rm{(2)}}]
Suppose ${\mathfrak{g}}_{\mathbb{C}}$ is simple
 and ${\mathfrak{g}}_{\mathbb{C}} \ne {\mathfrak{sl}}_2({\mathbb{C}})$.  
Then the bounded multiplicity property \eqref{eqn:gtensor} holds 
if and only if $H$ is of the form $H=H_1 \times H_2$
 and the triple $({\mathfrak{g}}, {\mathfrak{h}}_1, {\mathfrak{h}}_2)$
 is in the following table.  

\begin{table}[H]
\begin{center}
\begin{tabular}[t]{ccc}
${\mathfrak{g}}_{\mathbb{C}}$
&${{\mathfrak{h}}_1}_{\mathbb{C}}$
&${{\mathfrak{h}}_2}_{\mathbb{C}}$
\\
\hline
${\mathfrak{s o}}_n$
&${\mathfrak{s o}}_{n-1}$
&${\mathfrak{s o}}_{n-1}$
\\
${\mathfrak{s o}}_{8}$
&${\mathfrak{s o}}_{7}$
&${\mathfrak{spin}}_7$
\\
${\mathfrak{s o}}_{8}$
&${\mathfrak{s o}}_{7}$
&${\mathfrak{g l}}_4$
\\
${\mathfrak{s l}}_{4}$
&${\mathfrak{g l}}_{3}$
&${\mathfrak{s p}}_2$
\\
\end{tabular}
\end{center}
\caption{Tensor product with bounded multiplicities}
\label{tab:tensor}
\hfil
\end{table}

\item[{\rm{(3)}}]
Suppose ${\mathfrak{g}}$ is a complex simple Lie algebra 
 and ${\mathfrak{g}} \ne {\mathfrak{sl}}_2({\mathbb{C}})$.  
Then the bounded multiplicity property \eqref{eqn:gtensor} holds
 if and only if ${\mathfrak{h}}$ is a complex subalgebra
 and $H$ is of the form $H=H_1 \times H_2$
 and the triple 
 $({\mathfrak{g}}, {\mathfrak{h}}_1, {\mathfrak{h}}_2)$
 is in Table \ref{tab:tensor}.  
\end{enumerate}
\end{theorem}

\begin{remark}
(1)\enspace
The infinitesimal equivalence in Definition \ref{def:eqtriple}
 includes the switch
 of factors $H_1$ and $H_2$
 because it is induced
 by an outer automorphism
 of the direct product group.  
\par\noindent
(2)\enspace
Inside the Lie algebra ${\mathfrak{s o}}_{8}$, 
 there are three subalgebras
 that are isomorphic to ${\mathfrak{s o}}_{7}$
 up to inner automorphisms, 
 to which we may refer 
 as ${\mathfrak{so}}_7$, 
 ${\mathfrak{spin}}_{7}^+$, 
 and ${\mathfrak{spin}}_{7}^-$, 
 and these are conjugate to each other
 by an outer automorphism
 of ${\mathfrak{so}}_8$.  
By the equivalence in Definition \ref{def:eqtriple}, 
 there are two equivalence classes
 of the triples 
$
   ({\mathfrak{g}}_{\mathbb{C}}, {{\mathfrak{h}}_1}_{\mathbb{C}}, 
{{\mathfrak{h}}_2}_{\mathbb{C}})
$
 where ${\mathfrak{g}}_{\mathbb{C}}={\mathfrak{s o}}_{8}$
 and ${{\mathfrak{h}}_j}_{\mathbb{C}}$
 are isomorphic to ${\mathfrak{s o}}_{7}$
 ($j=1,2$), 
 according to whether ${{\mathfrak{h}}_1}_{\mathbb{C}}$
 is conjugate to ${{\mathfrak{h}}_2}_{\mathbb{C}}$
 by an inner automorphism or not.  
In Table \ref{tab:tensor}, 
 we write them 
 as $({\mathfrak{s o}}_{8}, {\mathfrak{s o}}_{7}, {{\mathfrak{s o}}_7})$
 and $({{\mathfrak{s o}}_8}, {\mathfrak{s o}}_7, {\mathfrak{spin}}_{7})$, 
 respectively.  
\end{remark}

\section{Proof of classification results}
\label{sec:pfclass}

By applying the geometric criteria
 in Theorems \ref{thm:bdd} and \ref{thm:tensor}, 
 we complete the proof 
 of Theorems \ref{thm:listreal}, \ref{thm:cpxlist}, 
 and \ref{thm:gplist}
 about the classification 
 of the triples $(G,H,G')$ 
having the bounded multiplicity property \eqref{eqn:BBH}
 of the restriction $\Pi|_{G'}$
 for all $\Pi \in \operatorname{Irr}(G)_H$
 and also the proof of Theorem \ref{thm:tensorlist}
 for the tensor product case.

\subsection{Preliminary lemmas}

In this subsection 
 we collect some lemmas 
 that we shall use in the proof.

Given two $G$-spaces $X_j$ $(j=1,2)$
 and an automorphism $\sigma$ of $G$, 
 we let $g \in G$ act on the direct product space $X_1 \times X_2$ by
 $(x,y) \mapsto (g x, \sigma(g) y)$, 
 and call it the {\it{$\sigma$-twisted diagonal action}} of $G$.  

\begin{lemma}
\label{lem:diagvis}
Let $G_{\mathbb{C}}$ be a complex simple Lie group, 
 and $G_U$ a maximal compact subgroup of $G_{\mathbb{C}}$.  
Suppose that $Q_1$, $Q_2$ are parabolic subgroups
 of $G_{\mathbb{C}}$, 
 and we set $X=G_{\mathbb{C}}/Q_1 \times G_{\mathbb{C}}/Q_2$.  
Let $\sigma$ be an automorphism of $G_{\mathbb{C}}$.  
Then the following four conditions are equivalent:
\begin{enumerate}
\item[{\rm{(i)}}]
$X$ is $G_{\mathbb{C}}$-spherical via the diagonal action;
\item[{\rm{(i)$'$}}]
$X$ is strongly $G_U$-visible via the diagonal action;
\item[{\rm{(ii)}}]
$X$ is $G_{\mathbb{C}}$-spherical via the $\sigma$-twisted diagonal action;  
\item[{\rm{(ii)$'$}}]
$X$ is strongly $G_U$-visible
 via the $\sigma$-twisted diagonal action.  
\end{enumerate}
\end{lemma}

\begin{proof}
See \cite{xtanaka} for the equivalences
 (i) $\iff$ (i)$'$
 and  (ii) $\iff$ (ii)$'$.  
The equivalences (i) $\iff$ (ii) and (i)$'$ $\iff$ (ii)$'$
 are obvious
 when $\sigma$ is an inner automorphism, 
 and follow from the classification
 of strongly visible actions
 \cite{xtanaka12} for (i)$'$ $\iff$ (ii)$'$, 
 or alternatively that of spherical varieties  \cite[Thm.~5.2]{xhnoo}
 for (i) $\iff$ (ii)
 when $\sigma$ is an outer automorphism.  
\end{proof}

Since our criteria in Theorems \ref{thm:bdd} and \ref{thm:tensor}
 are formulated
 by the complexified Lie group $G_{\mathbb{C}}$, 
 it is convenient to fix our convention
 when $G$ itself has a complex structure.  
Suppose $G$ is a complex Lie group.  
We write $J$ for the complex structure on ${\mathfrak{g}}$, 
 and decompose ${\mathfrak{g}}_{\mathbb{C}}={\mathfrak{g}} \otimes_{\mathbb{R}}{\mathbb{C}}$
 into the direct sum of the eigenspaces
 ${\mathfrak{g}}^{\operatorname{hol}}$ and ${\mathfrak{g}}^{\operatorname{anti}}$
 of $J$
 with eigenvalues
 $\sqrt{-1}$ and $-\sqrt{-1}$, 
respectively.  
Then one has a direct sum decomposition:
\[
 {\mathfrak{g}} \oplus {\mathfrak{g}} \overset \sim \to
 {\mathfrak{g}}^{\operatorname{hol}} \oplus {\mathfrak{g}}^{\operatorname{anti}} ={\mathfrak{g}}_{\mathbb{C}}, 
\quad
  (X,Y) \mapsto \frac 1 2 (X-\sqrt{-1} J X, Y+\sqrt{-1}JY).  
\]
Accordingly, 
 the complexification $G_{\mathbb{C}}$
 of the complex Lie group $G$
 is given by the totally real embedding 
\begin{equation}
\label{eqn:cpxcpx} 
\operatorname{diag} \colon G \hookrightarrow G \times G=:G_{\mathbb{C}}, 
\end{equation}
where the second factor equipped 
 with the reverse complex structure.

For a connected complex simple Lie group $G$, 
 there are two types for symmetric pairs $(G,H)$
 defined by an involution $\sigma$ of $G$:
\begin{enumerate}
\item[(1)]($\sigma$ is holomorphic)\quad\quad\,\,
$H$ is a complex subgroup of $G$, 
\item[(2)]($\sigma$ is anti-holomorphic)\enspace
$H$ is a real form of $G$.  
\end{enumerate}

For simplicity, 
 suppose that the subgroup $H$ is connected.  
Then via the identification 
 $G_{\mathbb{C}} \simeq G \times G$ in \eqref{eqn:cpxcpx}, 
 one has 
\begin{alignat*}{3}
H_{\mathbb{C}} &\simeq && H \times H
&&\text{for (1)}, 
\\
H_{\mathbb{C}} &\simeq && \operatorname{diag}_{\sigma}(G)
:=\{(g, \sigma g):g \in G\}
\quad
&&\text{for (2)}.  
\end{alignat*}
Correspondingly, 
 the Borel subgroup $B_{G/H}$
 for the symmetric space $G/H$
 (Definition \ref{def:Borel})
 is given as follows.

\begin{lemma}
\label{lem:QGcpxH}
Suppose $G$ is a complex simple Lie group, 
 and $G/H$ a symmetric space defined
 by an involutive automorphism $\sigma$ of $G$.  
Let $B_{G/H}$ be a Borel subgroup for $G/H$ 
 as a (real) symmetric space, 
 which is regarded as a subgroup of $G \times G$
 via the identification
 $G_{\mathbb{C}} \simeq G \times G$ in \eqref{eqn:cpxcpx}.  
\begin{enumerate}
\item[{\rm{(1)}}]
If $\sigma$ is holomorphic, 
 then $G/H$ is a complex symmetric space.  
We write $B_{G/H}^c$ for the Borel subgroup
 for the {\it{complex}} symmetric space 
 $G/H$, 
 see below.  
Then $B_{G/H}$ is isomorphic to $B_{G/H}^c \times B_{G/H}^c$.  
\item[{\rm{(2)}}]
If $\sigma$ is anti-holomorphic, 
 then $H$ is a real form of $G$
 and $B_{G/H}$ is isomorphic to $B \times B$
 where $B$ is a Borel subgroup of the complex Lie group $G$.  
\end{enumerate}
\end{lemma}

Here, 
 a Borel subgroup $B_{G/H}^c$
 for the {\it{complex}} symmetric space $G/H$
 is defined as a parabolic subgroup of $G$, 
 rather than that of the complexification $G_{\mathbb{C}}$.  
That is, 
 when $G$ itself is a complex reductive Lie group 
 and $\sigma$ is a holomorphic involution, 
 we define $B_{G/H}^c$ to be the parabolic subgroup of $G$ 
 associated to $\Sigma^+({\mathfrak{g}}, {\mathfrak{j}})$
 instead
 of that of $G_{\mathbb{C}}$
 associated to $\Sigma^+({\mathfrak{g}}_{\mathbb{C}}, {\mathfrak{j}}_{\mathbb{C}})$
 in Definition \ref{def:Borel}.

The case ${\mathfrak{g}}$ or ${\mathfrak{g}}_{\mathbb{C}}={\mathfrak{sl}}_2({\mathbb{C}})$ is distinguished in the classification from other cases, 
 for which we formulate in the following two lemmas.  
\begin{lemma}
\label{lem:3equiv}
Let $G$ be a complex simple Lie group 
 with Lie algebra ${\mathfrak{g}}$, 
 and $B$ a Borel subgroup of $G$. 
Then the following three conditions are equivalent:
\begin{enumerate}
\item[{\rm{(i)}}]
$(G \times G)/(B \times B)$
 is $G$-spherical
 via the diagonal action.  
\item[{\rm{(ii)}}]
$(G \times G \times G)/\operatorname{diag}(G)$
 is spherical as a $(G \times G \times G)$-space.  
\item[{\rm{(iii)}}]
${\mathfrak{g}}={\mathfrak{sl}}_2({\mathbb{C}})$.  
\end{enumerate}
\end{lemma}

\begin{proof}
The first equivalence (i) $\iff$ (ii) is 
 clear from the bijection:
\[
   \operatorname{diag} B \backslash (G \times G)/(B \times B)
\simeq
 (B \times B \times B) \backslash (G \times G \times G)/\operatorname{diag}G.  
\]
See \cite[Ex.~2.8.6]{Ksuron} or \cite[Prop.~4.3]{xKMt}
 {\it{e.g.}}, for the equivalence (ii) $\iff$ (iii).  
\end{proof}

\begin{lemma}
\label{lem:sl2}
Let $G$ be a semisimple Lie group, 
 $G_{\mathbb{C}}$ a complexification of $G$, 
 and $B$ a Borel subgroup of $G_{\mathbb{C}}$.  
Assume that ${\mathfrak{g}}_{\mathbb{C}}$ is a direct sum
 of copies of ${\mathfrak{sl}}_2({\mathbb{C}})$.  
Then for any symmetric pair $(G,G')$, 
 $G_{\mathbb{C}}/B$ is $G_{\mathbb{C}}'$-spherical.  
\end{lemma}

\begin{proof}
It suffices to show 
 when the symmetric pair $(G,G')$ is irreducible, 
 namely, 
 either ${\mathfrak{g}}$ is simple
 or $({\mathfrak{g}}, {\mathfrak{g}}')=({}^{\backprime}{\mathfrak{g}}\oplus{}^{\backprime} {\mathfrak{g}}, \operatorname{diag}{}^{\backprime}{\mathfrak{g}})$
 for simple ${}^{\backprime}{\mathfrak{g}}$.  
The assertion is straightforward
 in the first case
 where $G_{\mathbb{C}}/B \simeq {\mathbb{P}}^1{\mathbb{C}}$, 
 and follows from the implication (iii) $\Rightarrow$ (i) and (ii)
 of Lemma \ref{lem:3equiv}
 in the second case.  
\end{proof}

\subsection{Proof of Theorem \ref{thm:listreal} (${\mathfrak{g}}_{\mathbb{C}}$ simple)}

In this subsection, 
 we give a proof of Theorem \ref{thm:listreal}
 which deals with the case
 that $G$ is a simple Lie group
 and $G$ is not a complex Lie group, 
 namely, 
 the complexified Lie algebra ${\mathfrak{g}}_{\mathbb{C}}$ is simple.

Let $G/H$ and $G/G'$ be symmetric spaces, 
 and $B_{G/H}$ $(\subset G_{\mathbb{C}})$ a Borel subgroup for $G/H$.  
The proof of Theorem \ref{thm:listreal} is 
 based on the criterion in Theorem \ref{thm:bdd}
 for the bounded multiplicity property \eqref{eqn:BBH}, 
 which we observe 
 is determined
 only by the complexified Lie algebras
 ${\mathfrak{g}}_{\mathbb{C}}$, 
 ${\mathfrak{h}}_{\mathbb{C}}$, 
 and ${\mathfrak{g}}_{\mathbb{C}}'$.  
Then our strategy to classify the triple
 $({\mathfrak{g}}_{\mathbb{C}},{\mathfrak{h}}_{\mathbb{C}},{\mathfrak{g}}_{\mathbb{C}}')$
 is to fix a complex symmetric pair
 $({\mathfrak{g}}_{\mathbb{C}},{\mathfrak{g}}_{\mathbb{C}}')$
 and  to classify ${\mathfrak{h}}_{\mathbb{C}}$ 
 such that $G_{\mathbb{C}}/B_{G/H}$
 is $G_{\mathbb{C}}'$-spherical, 
 which is divided into two steps.  
\par\noindent
{\bf{Step 1.}}\enspace
Classify parabolic subgroups $P$
 of a complex simple Lie group $G_{\mathbb{C}}$
 such that $G_{\mathbb{C}}/P$ is $G_{\mathbb{C}}'$-spherical
 or equivalently, 
 is $G_U'$-strongly visible.  
\par\noindent
{\bf{Step 2.}}\enspace
Classify the complex symmetric pairs $({\mathfrak{g}}_{\mathbb{C}}, {\mathfrak{h}}_{\mathbb{C}})$ 
 such that the Borel subgroup $B_{G/H}$ 
 appears in the list obtained in Step 1.  
\par
Step 1 is done in \cite[Thm.~5.2]{xhnoo}. 
See also \cite{xrims40, xtanaka12}
 for some classification results of strongly visible actions.

\begin{proof}[Proof of Theorem \ref{thm:listreal}]
By the above argument, 
 it suffices is to carry out a computation in Step 2
 for each complex symmetric pairs
 $({\mathfrak{g}}_{\mathbb{C}}, {\mathfrak{g}}_{\mathbb{C}}')$
 with ${\mathfrak{g}}_{\mathbb{C}}$ simple.  
We illustrate this computation in the following setting:
\begin{equation}
\label{eqn:slsl}
\text{$({\mathfrak{g}}_{\mathbb{C}},{\mathfrak{g}}_{\mathbb{C}}')=({\mathfrak{sl}}_n, {\mathfrak{sl}}_p \oplus {\mathfrak{sl}}_{q} \oplus {\mathbb{C}})$
 with $n=p+q$.  }
\end{equation}

Take a subset $\Theta$ of the set $\{\alpha_1, \cdots, \alpha_{n-1}\}$
 of simple roots
 for $\Delta^+({\mathfrak{g}}_{\mathbb{C}}, \widetilde{\mathfrak{j}}_{\mathbb{C}})$
 labelled as in Bourbaki \cite{Bou7a9}.  
We write $\Theta^{c}$
 for the complement of $\Theta$, 
 and $P^{\Theta}$ for the parabolic subgroup of $G_{\mathbb{C}}$
 corresponding to $\Theta$.  
We recall our convention
 that $P^{\Theta}$ is a Borel subgroup of $G_{\mathbb{C}}$
 if ${\Theta}$ is empty.  
Then $G_{\mathbb{C}}/{P^{\Theta}}$ is $G_{\mathbb{C}}'$-spherical
 (\cite[Thm.~5.2]{xhnoo}), 
 or equivalently,
 $G_U'$-strongly visible 
 (\cite[Thm.~A]{K07b}, 
 see also \cite[Thm.~16]{xrims40})
 if and only if $\Theta$ satisfies one of the following conditions:

\par\noindent\phantom{n}
{\bf{Case 1.}}\enspace
$\# \Theta^c \le 1$, 
\vskip 0.5pc
\par\noindent\hphantom{n}
{\bf{Case 2.}}\enspace
$\Theta^c =\{\alpha_1, \alpha_i\}$, 
$\{\alpha_i, \alpha_{i+1}\}$, $\{\alpha_i, \alpha_{n-2}\}$
 for some $i$, 
\vskip 0.5pc
\par\noindent\hphantom{n}
{\bf{Case 3.}}\enspace
$\operatorname{min}(p,q)=2$
 and $\# \Theta^c =2$, 
\vskip 0.5pc
\par\noindent\hphantom{n}
{\bf{Case 4.}}\enspace
$\operatorname{min}(p,q)=1$
 and $\Theta$ is arbitrary.  
\par
As a second step, 
 we now examine
 if the corresponding parabolic subalgebra ${\mathfrak{p}}^{\Theta}$ is isomorphic to the Borel subalgebra
 of some symmetric pair $({\mathfrak{g}}, {\mathfrak{h}})$.  
Suppose $({\mathfrak{g}}, {\mathfrak{h}})$ is a symmetric pair.  
Let ${\mathfrak{g}}_{\mathbb{R}}$ be the real form of ${\mathfrak{g}}_{\mathbb{C}}$
 corresponding to the pair
 $({\mathfrak{g}}_{\mathbb{C}},{\mathfrak{h}}_{\mathbb{C}})$
 as in \eqref{eqn:gRghC}, 
 and $\Theta$ the set of the black circles
 in the Satake diagram  of ${\mathfrak{g}}_{\mathbb{R}}$, 
 see \cite[Ch.~X, Table VI]{He78}, 
 for instance.  
By Lemma \ref{lem:Satake}, 
 the complex parabolic subalgebra ${\mathfrak{p}}^{\Theta}$
 is a Borel subalgebra ${\mathfrak{b}}_{G/H}$
 for the symmetric pair $({\mathfrak{g}}, {\mathfrak{h}})$.  
Among real forms ${\mathfrak{g}}_{\mathbb{R}}={\mathfrak{sl}}_n({\mathbb{R}})$,
 ${\mathfrak{su}}(p,n-p)$,
 and ${\mathfrak{su}}^{\ast}(n)$ ($n$: even)
 of ${\mathfrak{g}}_{\mathbb{C}}={\mathfrak{s l}}_n({\mathbb{C}})$, 
 the number of the white circles
 ($= \# \Theta^c$)
 in the Satake diagram 
 is equal to 1 or 2
 if and only if the real form ${\mathfrak{g}}_{\mathbb{R}}$ is isomorphic to 
 one of the following:
\begin{equation}
\label{eqn:white2}
   {\mathfrak{sl}}(n,{\mathbb{R}})\,\,(n=2,3);\,
   {\mathfrak{su}}^{\ast}(2m)\,\,(m=2,3);
   \text{ or }
   {\mathfrak{su}}(1,n-1), 
\end{equation}
 and correspondingly, 
 the set of the white circles 
 is given by
\[
   \{\alpha_1\}, \{\alpha_1, \alpha_2\};\,\, 
   \{\alpha_2\}, \{\alpha_2,\alpha_4\};\,\,
   \{\alpha_1, \alpha_{n-1}\}, 
\]
respectively.  
Therefore $G_{\mathbb{C}}/P^{\Theta}$ is $G_{\mathbb{C}}'$-spherical
 if and only if one of the following conditions holds:
\par\noindent\hphantom{n}
{\bf{Case 1.}}\enspace
${\mathfrak{g}}_{\mathbb{R}} \simeq {\mathfrak{sl}}(2,{\mathbb{R}})$
 or ${\mathfrak{su}}^{\ast}(4)$, 
\vskip 0.5pc
\par\noindent\hphantom{n}
{\bf{Case 2.}}\enspace
${\mathfrak{g}}_{\mathbb{R}} \simeq {\mathfrak{sl}}(3,{\mathbb{R}})$
 or ${\mathfrak{su}}(1,n-1)$, 
\vskip 0.5pc
\par\noindent\hphantom{n}
{\bf{Case 3.}}\enspace
$\operatorname{min}(p,q)=2$
 and ${\mathfrak{g}}_{\mathbb{R}}$ is one of \eqref{eqn:white2}, 
\vskip 0.5pc
\par\noindent\hphantom{n}
{\bf{Case 4.}}\enspace
$\operatorname{min}(p,q)=1$
 and ${\mathfrak{g}}$ is arbitrary.  
\par
\vskip 0.5pc
This exhausts the list in Table \ref{tab:0.1} 
 with $({\mathfrak{g}}_{\mathbb{C}},{\mathfrak{g}}_{\mathbb{C}}')=({\mathfrak{sl}}_n, {\mathfrak{sl}}_p \oplus {\mathfrak{sl}}_{q} \oplus {\mathbb{C}})$
 except for the triple
 $({\mathfrak{g}}_{\mathbb{C}},{\mathfrak{h}}_{\mathbb{C}}, {\mathfrak{g}}_{\mathbb{C}}')=({\mathfrak{sl}}_4, {\mathfrak{sp}}_2, {\mathfrak{sl}}_2 \oplus {\mathfrak{sl}}_2 \oplus {\mathbb{C}})$, 
 which we have omitted from the table 
 because it is isomorphic to a special case of the triple 
$({\mathfrak{so}}_{a+b}, {\mathfrak{so}}_{a+b-1}, {\mathfrak{so}}_a \oplus {\mathfrak{so}}_b)$
 with $(a,b)=(2,4)$.

For other symmetric pairs 
 $({\mathfrak{g}}_{\mathbb{C}}, {\mathfrak{g}}_{\mathbb{C}}')$, 
 the proof of the classification ${\mathfrak{h}}_{\mathbb{C}}$
 is similar
 (and often simpler).  
\end{proof}
\subsection{Proof of Theorem \ref{thm:cpxlist} 
(${\mathfrak{g}}$ complex simple)}
\label{subsec:pfcpxlist}

In this section, 
 we give a proof of Theorem \ref{thm:cpxlist}
 which deals with the case
 that $G$ is a complex simple Lie group.  
Since the Lie algebra ${\mathfrak{g}}$ 
 has a complex structure, 
 there are four possibilities
 for the symmetric pairs $({\mathfrak{g}}, {\mathfrak{h}})$
 and $({\mathfrak{g}}, {\mathfrak{g}}')$:
\newline\indent
{\bf{Case I${}^c$-a.}}\enspace
Both ${\mathfrak{h}}$ and ${\mathfrak{g}}'$ are complex subalgebras.  
\newline\indent
{\bf{Case I${}^c$-b.}}\enspace
${\mathfrak{h}}$ is a complex subalgebra
 and ${\mathfrak{g}}'$ is a real form of ${\mathfrak{g}}$.  
\newline\indent
{\bf{Case I${}^c$-c.}}\enspace
${\mathfrak{h}}$ is a real form of ${\mathfrak{g}}$, 
 and ${\mathfrak{g}}'$ is a complex subalgebra.  
\newline\indent
{\bf{Case I${}^c$-d.}}\enspace
Both ${\mathfrak{h}}$ and ${\mathfrak{g}}'$ are real forms 
 of ${\mathfrak{g}}$.

\begin{proof}[Proof of Theorem \ref{thm:cpxlist}]
The classification in Case I$^c$-a goes similarly
 to the aforementioned proof of Theorem \ref{thm:listreal}.  
By Lemma \ref{lem:diagvis}, 
 the classification in Case I$^c$-b is equivalent to the special case
 of Theorem \ref{thm:tensorlist} (2)
 with ${\mathfrak{h}}_1 = {\mathfrak{h}}_2$, 
 which will be proved in Section \ref{subsec:pftensor}.  
The classification in Case I$^c$-c is reduced 
 to the classification of the pairs 
 $({\mathfrak{g}}, {\mathfrak{g}}')$
 such that $(G_{\mathbb{C}} \times G_{\mathbb{C}})/(B \times B)$
 is $(G_{\mathbb{C}}' \times G_{\mathbb{C}}')$-spherical, 
 or equivalently, 
 $G_{\mathbb{C}}/B$ is $G_{\mathbb{C}}'$-spherical.  
This is classified as in \eqref{eqn:BBlist}
 by Kr{\"a}mer \cite{xkramer}.  
The pair $({\mathfrak{so}}_8, {\mathfrak{spin}}_7)$
 is not listed in Theorem \ref{thm:cpxlist}
 because the classification is listed
 up to the equivalence in Definition \ref{def:eqtriple}.  
The classification in Case I$^c$-d
 is equivalent to that of ${\mathfrak{g}}$ 
 such that $(G_{\mathbb{C}} \times G_{\mathbb{C}})/(B \times B)$
 is $G_{\mathbb{C}}$-spherical under the diagonal action of $G_{\mathbb{C}}$.  
Then ${\mathfrak{g}}_{\mathbb{C}} \simeq {\mathfrak{sl}}_2({\mathbb{C}})$
 by Lemma \ref{lem:3equiv}.  
\end{proof}

\subsection{Proof of Theorem \ref{thm:gplist} for group manifold case}
\label{subsec:pfgplist}

Let $G$ be a simple Lie group. 
In this section we complete the proof of Theorem \ref{thm:gplist}
 by classifying symmetric pairs
 $(G \times G, G')$ 
 having the bounded multiplicity property \eqref{eqn:BBH}.  
We need to treat the following cases.  
\newline\indent
{\bf{Case II-1.}}\enspace
${\mathfrak{g}}'={\mathfrak{g}}_1' \oplus {\mathfrak{g}}_2'$
 such that $({\mathfrak{g}}, {\mathfrak{g}}_j')$
 ($j=1,2$) are symmetric pairs.  
\newline\indent\indent
{\bf{Case II$^r$-1.}}\enspace
${\mathfrak{g}}_{\mathbb{C}}$ is simple.
\newline\indent\indent
{\bf{Case II$^c$-1.}}\enspace
${\mathfrak{g}}$ is a complex simple Lie algebra.  
\newline\indent
{\bf{Case II-2.}}\enspace
${\mathfrak{g}}'=\operatorname{diag}_{\sigma}({\mathfrak{g}})$
 for some involutive automorphism $\sigma$ of ${\mathfrak{g}}$.  
\newline\indent\indent
{\bf{Case II$^r$-2.}}\enspace
${\mathfrak{g}}_{\mathbb{C}}$ is simple.  
\newline\indent\indent
{\bf{Case II$^c$-2.}}\enspace
${\mathfrak{g}}$ is a complex simple Lie algebra.

We recall that for ${\mathfrak{g}}$ a simple Lie algebra, 
 ${\mathfrak{g}}$ is a complex simple Lie algebra 
 if and only if
 ${\mathfrak{g}}_{\mathbb{C}}$ is not simple.  

\begin{proof}
[Proof of Theorem \ref{thm:gplist}]
We note 
 that the Borel subgroup for the symmetric space
 $(G \times G)/\operatorname{diag}G$
 is given by $B \times B_-$, 
 where $B$ is a Borel subgroup of $G_{\mathbb{C}}$
 and $B_-$ its opposite Borel subgroup.  
Since $B_-$ is conjugate by $B$, 
 we shall simply use $B \times B$
 instead of $B \times B_-$.  
Then the bounded multiplicity property \eqref{eqn:gpbdd} holds
 if and only if $(G_{\mathbb{C}} \times G_{\mathbb{C}})/(B \times B)$
 is $G_{\mathbb{C}}'$-spherical by Theorem \ref{thm:bdd}.  
\newline
(1)\enspace
In Case II-1, 
 it suffices to classify the triple 
 $({\mathfrak{g}},{\mathfrak{g}}_1, {\mathfrak{g}}_2)$
 for which the flag variety
 $(G_{\mathbb{C}} \times G_{\mathbb{C}})/(B \times B)$
 is $({G_1}_{\mathbb{C}} \times {G_2}_{\mathbb{C}})$-spherical.  
By a classical result of Kr{\"a}mer \cite{xkramer}, 
 this happens
 if and only if $({\mathfrak{g}}, {\mathfrak{g}}_j)$
 or $({\mathfrak{g}}_{\mathbb{C}}, {{\mathfrak{g}}_j}_{\mathbb{C}})$
 are in the list \eqref{eqn:BBlist}.  
By the equivalence relation 
 for triples in Definition \ref{def:eqtriple}, 
 the classification in Case II-1 follows.  
\par\noindent
(2)\enspace
In Case II-2, 
 it suffices to classify $(G, \sigma)$
 for which $(G_{\mathbb{C}} \times G_{\mathbb{C}})/(B \times B)$
 is $G_{\mathbb{C}}$-spherical
 via the $\sigma$-twisted diagonal action.  
Then the classification follows from Lemma \ref{lem:3equiv}
 in Case II$^r$-2.  
Similarly for Case II$^c$-2.  
\end{proof}

\subsection{Proof of Theorem \ref{thm:tensorlist} for tensor product case}
\label{subsec:pftensor}

Let $G$ be a simple Lie group.  
In this section we complete the proof of Theorem \ref{thm:tensorlist}
 by classifying the symmetric pairs $(G \times G, H)$
 having the bounded multiplicity property
 \eqref{eqn:gtensor}
 for the tensor product representations
 by using the criterion in Theorem \ref{thm:tensor}.  
According to whether the simple Lie algebra ${\mathfrak{g}}$ has 
 a complex structure or not, 
 we treat separately in the following subcases:
\newline\indent
{\bf{Case III-1.}}\enspace
${\mathfrak{h}}={\mathfrak{h}}_1 \oplus {\mathfrak{h}}_2$
 such that $({\mathfrak{g}}, {\mathfrak{h}}_j)$
 ($j=1,2$) are symmetric pairs.  
\newline\indent\indent
{\bf{Case III$^r$-1.}}\enspace
${\mathfrak{g}}_{\mathbb{C}}$ is simple.  
\newline\indent\indent
{\bf{Case III$^c$-1.}}\enspace
${\mathfrak{g}}$ is a complex simple Lie algebra.  
\newline\indent
{\bf{Case III-2.}}\enspace
${\mathfrak{h}}=\operatorname{diag}_{\sigma}({\mathfrak{g}})$
 for some involutive automorphism $\sigma$ of ${\mathfrak{g}}$.  
\newline\indent\indent
{\bf{Case III$^r$-2.}}\enspace
${\mathfrak{g}}_{\mathbb{C}}$ is simple.  
\newline\indent\indent
{\bf{Case III$^c$-2.}}\enspace
${\mathfrak{g}}$ is a complex simple Lie algebra.

We first prove the following proposition:
\begin{proposition}
\label{prop:210499}
Let $B_{G/H_j}$ $(\subset G_{\mathbb{C}})$ be
 the Borel subgroup for symmetric spaces $G/H_j$ $(j=1,2)$.  
Suppose ${\mathfrak{g}}_{\mathbb{C}}$ is simple.  
Then the following three conditions are equivalent:
\begin{enumerate}
\item[{\rm{(i)}}]
$(G_{\mathbb{C}} \times G_{\mathbb{C}})/(B_{G/H_1} \times B_{G/H_2})$ is
$G_{\mathbb{C}}$-spherical.  
\item[{\rm{(ii)}}]
$(G_{\mathbb{C}} \times G_{\mathbb{C}})/(B_{G/H_1} \times B_{G/H_2})$ is
$G_U$-strongly visible.  
\item[{\rm{(iii)}}]
The triple $({\mathfrak{g}}_{\mathbb{C}}, {{\mathfrak{h}}_1}_{\mathbb{C}}, {{\mathfrak{h}}_2}_{\mathbb{C}})$ is in Table \ref{tab:tensor}.  
\end{enumerate}
\end{proposition}

In Case III-1, 
 we recall from Theorem \ref{thm:tensor}
 that the bounded multiplicity property \eqref{eqn:bddt}
 of the tensor product representation $\Pi_1 \otimes \Pi_2$
 is equivalent to one of (therefore, any of) (i) and (ii) 
 in Proposition \ref{prop:210499}.  

\begin{proof}
The complete list of the pairs of parabolic subgroups $(Q_1,Q_2)$
 satisfying (ii) (or equivalently (i))
 is given in \cite{xtanaka12}, 
see also \cite{K07b} for type $A$, 
 and \cite{Li94} for maximal parabolic subgroups.  
In order to prove (ii) $\iff$ (iii)
 (or (i) $\iff$ (iii)), 
 it suffices to determine
 when $Q_j$ ($j=1,2$) are isomorphic to Borel subgroups
 for some symmetric pairs $(G, H_j)$.  
We illustrate the proof for (ii) $\iff$ (iii)
 with two cases: ${\mathfrak{g}}_{\mathbb{C}}$
 is of type $A$ and of type $D$, 
 in particular,
 of $D_4$.  
Other cases are similar.

Suppose ${\mathfrak{g}}_{\mathbb{C}}$ is of type $A$.  
We first observe from \cite{K07b} or \cite{xtanaka12}
 that if (ii) holds
 then 
$
\# \Theta_1^c + \# \Theta_2^c \le 3
$
 when ${\mathfrak{g}}_{\mathbb{C}}$
 is of type $A$.

In view of the Satake diagram
 for real forms of ${\mathfrak{g}}_{\mathbb{C}}$ of type $A$, 
$\# \Theta_1^c \le 2$
 if and only if one of the following holds:
\begin{alignat*}{2}
\# \Theta_1^c&=1
\quad
&&({\mathfrak{g}}_{\mathbb{C}}, {{\mathfrak{h}}_1}_{\mathbb{C}})
=({\mathfrak{sl}}_4({\mathbb{C}}), {\mathfrak{sp}}_2({\mathbb{C}})), 
\\
\# \Theta_1^c&=2
\quad
&&({\mathfrak{g}}_{\mathbb{C}}, {{\mathfrak{h}}_1}_{\mathbb{C}})
=({\mathfrak{sl}}_n({\mathbb{C}}), {\mathfrak{gl}}_{n-1}({\mathbb{C}})).  
\end{alignat*}
Hence
$\# \Theta_1^c + \# \Theta_2^c \le 3$
 only if 
$({\mathfrak{g}}_{\mathbb{C}}, {{\mathfrak{h}}_1}_{\mathbb{C}}, {{\mathfrak{h}}_2}_{\mathbb{C}})$
 is either 
$({\mathfrak{sl}}_4({\mathbb{C}}), {\mathfrak{sp}}_2({\mathbb{C}}), {\mathfrak{sp}}_2({\mathbb{C}}))$
 or $({\mathfrak{sl}}_4({\mathbb{C}}), {\mathfrak{sp}}_2({\mathbb{C}}), 
{\mathfrak{gl}}_3({\mathbb{C}}))$
 up to switch of factors.  
Conversely, 
 the condition (ii) holds in this case by \cite{K07b}.

Suppose ${\mathfrak{g}}$ is of type $D$.  
Then (ii) holds
 only if 
$
\# \Theta_1^c + \# \Theta_2^c \le 3
$.  
Then a similar argument to the type $A$ case tells
 that the equality is not attained 
 if $n \ge 5$
 when $\Theta_j^c$ arise from symmetric pairs
 $({\mathfrak{g}}, {\mathfrak{h}}_j)$
 and that (ii) holds 
 if and only if 
$
({\mathfrak{g}}_{\mathbb{C}}, {{\mathfrak{h}}_1}_{\mathbb{C}}, {{\mathfrak{h}}_2}_{\mathbb{C}})
=({\mathfrak{so}}_{2n}({\mathbb{C}}), {\mathfrak{so}}_{2n-1}({\mathbb{C}}), {\mathfrak{so}}_{2n-1}({\mathbb{C}}))$
 for the type $D$ case
 if $n \ge 5$.

The remaining case is when ${\mathfrak{g}}_{\mathbb{C}}$ is of type $D_4$
 and 
$
\# \Theta_1^c + \# \Theta_2^c = 3
$.  
By \cite{xtanaka12}, 
 the list of such pairs $(\Theta_1^c, \Theta_2^c)$ 
 satisfying (ii) is of three types
\begin{alignat*}{3}
& (\{\alpha_i\}, \{\alpha_j, \alpha_k\})
\quad
&&\{i,j,k\}=\{1,3,4\}
\quad
&&\text{(3 cases),}
\\
& (\{\alpha_i\}, \{\alpha_i, \alpha_j\})
\quad
&&\{i,j\} \subset \{1,3,4\}
&&\text{(6 cases), }
\\
& (\{\alpha_i\}, \{\alpha_2, \alpha_j\})
\quad
&&\{i,j\} \subset\{1,3,4\}
&&\text{(6 cases), }
\end{alignat*}
up to switch of factors.  
The first two types do not arise from symmetric pairs, 
 whereas the third types arises from 
\begin{equation}
\label{eqn:JLTp49}
  ({\mathfrak{g}}_{\mathbb{C}}, {{\mathfrak{h}}_1}_{\mathbb{C}}, {{\mathfrak{h}}_2}_{\mathbb{C}})
=({\mathfrak{so}}_{8}({\mathbb{C}}), {\mathfrak{so}}_{7}({\mathbb{C}}), {\mathfrak{gl}}_{4}({\mathbb{C}}))
\end{equation}
up to outer automorphisms.  
(We recall
 that the group of outer automorphisms of $D_4$ is of order 6.)

The cases when ${\mathfrak{g}}_{\mathbb{C}}$ is not of type $A$ or $D$, 
the proof is similar.  
\end{proof}

\begin{proof}
[Proof of Theorem \ref{thm:tensorlist}]
Let $Q$ $(\subset G_{\mathbb{C}} \times G_{\mathbb{C}})$ be a Borel subgroup 
 for the symmetric space $(G \times G)/H$.  
By Theorem \ref{thm:bdd}, 
 it suffices to determine
 when the flag variety $(G_{\mathbb{C}} \times G_{\mathbb{C}})/Q$
 is $G_{\mathbb{C}}'$-spherical.
\par\noindent
(1)\enspace
 For  ${\mathfrak{g}}$ or ${\mathfrak{g}}_{\mathbb{C}}={\mathfrak{sl}}_2({\mathbb{C}})$, 
 the assertion follows from Lemma \ref{lem:sl2}
 for any ${\mathfrak{h}}$ $(\subset {\mathfrak{g}} \oplus {\mathfrak{g}})$.  
\par\noindent
(2)\enspace
Suppose that ${\mathfrak{g}}_{\mathbb{C}}$ is simple.  
In Case III$^r$-1, 
 the classification is given by Proposition \ref{prop:210499}.  
\par\noindent
In Case III$^r$-2, 
 $Q$ is the direct product of Borel subgroup of $G_{\mathbb{C}}$, 
 hence the classification follows from Lemma \ref{lem:3equiv}.  

\par\noindent
(3)\enspace
Suppose that ${\mathfrak{g}}$ is a complex simple Lie algebra.  
In this case, 
 the complexification $G_{\mathbb{C}}$ of $G$
 is given as $G \times G$ by \eqref{eqn:cpxcpx}.  
\par\noindent
In Case III$^c$-1, 
 the proof is the same with Case III$^r$-1
 if both ${\mathfrak{h}}_1$ and ${\mathfrak{h}}_2$ are complex subalgebras 
 of ${\mathfrak{g}}$.

If at least one of ${\mathfrak{h}}_1$ or ${\mathfrak{h}}_2$
 is a real form of ${\mathfrak{g}}$ in Case III$^c$-1, 
 then one has $Q=Q_1 \times Q_2$
 where $Q_1$ or $Q_2$ must be a Borel subgroup of $G_{\mathbb{C}}$
 by Lemma \ref{lem:QGcpxH}.  
Therefore, 
 if $(G_{\mathbb{C}} \times G_{\mathbb{C}})/(Q_1 \times Q_2)$
 is $G_{\mathbb{C}}$-spherical 
 by the diagonal action, 
 then ${\mathfrak{g}}_{\mathbb{C}}$ must be 
 ${\mathfrak{sl}}_2({\mathbb{C}})$ from the implication (i) $\Rightarrow$ (iii)
 in Lemma \ref{lem:3equiv}.

Similarly in Case III$^c$-2, 
the Borel subgroup $Q$ for the symmetric space
 $(G \times G)/\operatorname{diag}_{\sigma}(G)$
 is the direct product of Borel subgroups of $G_{\mathbb{C}}$, 
 hence the bounded multiplicity criterion in Theorem \ref{thm:tensor} amounts
 to the sphericity of $(G_{\mathbb{C}} \times G_{\mathbb{C}})/(B \times B)$
 by the diagonal $G_{\mathbb{C}}$-action, 
 which forces ${\mathfrak{g}}$ to be ${\mathfrak{sl}}_2({\mathbb{C}})$ again 
 by (i) $\Rightarrow$ (iii) in Lemma \ref{lem:3equiv}.  
\end{proof}

\section{Examples and some perspectives}

By the branching problems
 we mean the broad problem of understanding
 how irreducible representations
 of a group behave when restricted to a subgroup.  
As viewed in \cite{xKVogan2015}, 
 we may divide the branching problems
 into the following three stages:
\par\noindent
{\bf{Stage A.}}\enspace
Abstract features of the restriction;
\par\noindent
{\bf{Stage B.}}\enspace
Branching law;
\par\noindent
{\bf{Stage C.}}\enspace
Construction of symmetry breaking/holographic operators.  

The role of Stage A is to develop 
 an abstract theory on the restriction of representations
 as generally
 as possible.  
In turn, 
 we could expect a detailed study of the restriction
 in Stages B (decomposition of representations)
 and C (decomposition of vectors)
 in the specific settings 
 that are {\it{a priori}} guaranteed
 to be \lq\lq{nice}\rq\rq\
 in Stage A.

Solving Problem \ref{q:Bdd} or Problem \ref{q:bdd}
 on the bounded multiplicity property
 may be considered as in Stage A.  
In this section, 
 we discuss some promising examples
 in Stages B and C in the new settings
 that fit well into the framework of the present article.  
We also mention some examples
 which are \lq\lq{outside}\rq\rq\
 this framework, 
 to clarify the limitation
 as well.

When the pair $(G,G')$ satisfies the bounded multiplicity property \eqref{eqn:BB}, 
or
 when the complexified pair 
 $({\mathfrak{g}}_{\mathbb{C}}, {\mathfrak{g}}_{\mathbb{C}}')$
 is essentially 
 $({\mathfrak{sl}}_n, {\mathfrak{gl}}_{n-1})$
 or 
 $({\mathfrak{s o}}_n, {\mathfrak{s o}}_{n-1})$
 up to outer automorphisms, 
 see \eqref{eqn:BBlist}, 
 there have been active and rich study
 of the branching problems
 in Stages B and C
 in recent years, 
 such as
 the Gan--Gross--Prasad conjecture (Stage B)
 and 
 the construction of symmetry breaking operators
 (Stage C), 
 see {\it{e.g.,}} an exposition \cite{K19a}
 and references therein.  
Let us now focus on the {\bf{new settings}}
 when the pair $(G,G')$ does not satisfy \eqref{eqn:BB}, 
 but the triple $H \subset G \supset G'$
 satisfies 
the bounded multiplicity property \eqref{eqn:BBH}, 
 the triple $(\Pi,G,G')$ 
 with $\Pi \in \operatorname{Irr}(G)$ 
 satisfies $m(\Pi|_{G'}) < \infty$
 (see \eqref{eqn:msup}), 
 or the triple $(G, \Pi_1, \Pi_2)$
 with $\Pi_1, \Pi_2 \in \operatorname{Irr}(G)$ satisfies
 $m(\Pi_1 \otimes \Pi_2)<\infty$ (see \eqref{eqn:JLTp5}).  
We shall see in Examples \ref{ex:Opqanalysis}--\ref{ex:tri}
 below some previous successful results on the analysis
 of the branching laws $\Pi|_{G'}$
 in Stages B and C in these settings.  
One may observe
 that the existing results treated only a small part
 of this new framework 
 in comparison 
 with the complete list in Section \ref{sec:classification}.  
This observation indicates possible new avenues 
 of the rich study of the branching problems
 in Stages B and C.

\begin{example}
[Restriction of discrete series representations for $G/H$]
\label{ex:Opqanalysis}
~~~\newline
Suppose
 that $(G, H, G')=(O(p,q), O(p-1,q), O(p_1, q_1) \times O(p_2, q_2))$
 where $p_1+p_2=p$ and $q_1+q_2=q$.  
For simplicity, 
 suppose $p+q \ge 5$.  
In this case
 the pair $(G,G')$ does not satisfy
 the bounded multiplicity property \eqref{eqn:BB}
 if $p_1+q_1>1$ and $p_2+q_2>1$, 
 and not the finite multiplicity property \eqref{eqn:PP}
 if $p>1$ and $q>1$ 
 in addition, 
 see \cite{xKMt}, 
 however, 
 the triple $(G,H,G')$ always satisfies the finer bounded multiplicity property
 \eqref{eqn:BBH}
 as was seen in Example \ref{ex:Opq}.  
The branching problem 
 for the restriction $\Pi|_{G'}$
 of $\Pi \in \operatorname{Irr}(G)_H$ relates
 harmonic analysis 
 involving three groups $G$, $H$, and $G'$.  
Discrete series representations $\Pi$ were classified
 by Faraut \cite{F79}
 and Strichartz \cite{S83}, 
 which can be expressed also 
 in algebraic terms of Zuckerman derived functor modules
 $A_{\mathfrak{q}}(\lambda)$ \cite{KO02}.  
The branching laws $\Pi|_{G'}$ to the subgroup $G'$ were determined
 in the discretely decomposable case $(p_2=0)$
 in \cite{K93}, 
 and also in the case containing continuous spectrum
 under the assumption 
 that $(q_1,q_2)=(1,0)$
 by Frahm and Y.~Oshima \cite{MO15}.  
For general $(p_1, p_2, q_1, q_2)$, 
 full discrete spectrum
 of the restriction $\Pi|_{G'}$ occurs in a multiplicity-free fashion
 and is constructed and classified in \cite{K21}.  
See also a recent work 
 of {\O}rsted and Speh \cite{OS19}
 for another approach to capture a generic part
 of discrete spectrum
 in the branching law of $\Pi|_{G'}$.  
The triple $(G,H,G')$ in this example is a real form 
 of the complexified triple appearing in the fourth row of the left column in Table \ref{tab:0.1}.  
\end{example}

\begin{example}
[Unitary branching laws for mirabolic]
\label{ex:kop}
Let $G=GL_n({\mathbb{R}})$
 and $P$ a mirabolic, 
 {\it{i.e.,}}
 a maximal parabolic subgroup 
 with Levi factor $GL_1({\mathbb{R}}) \times GL_{n-1}({\mathbb{R}})$.  
Then for any symmetric pair $(G,G')$, 
 namely, 
 $G'=O(p,n-p)$, 
$GL_p({\mathbb{R}}) \times GL_{n-p}({\mathbb{R}})$,  
 $Sp_m({\mathbb{R}})$
 or $GL_m({\mathbb{C}})$
 when $n=2m$, 
 the generalized flag variety
 $G_{\mathbb{C}}/P_{\mathbb{C}}$ is $G_{\mathbb{C}}'$-spherical, 
 hence the triple $(G,P,G')$ is an example
 that fulfills the geometric assumption in Theorem \ref{thm:Qsph}
 with $Q=P_{\mathbb{C}}$.  
In this case the branching laws 
 of the unitary representation $\Pi|_{G'}$
 is explicitly found in \cite{KOP11}
for all the symmetric pairs $(G,G')$
 when $\Pi$ is a unitarily induced representation
 of a unitary character of $P$.  
The multiplicity in the (unitary) branching laws 
 of the restriction $\Pi|_{G'}$ is guaranteed
 to be uniformly bounded
 by Theorem \ref{thm:Qsph}, 
 even though the symmetric pairs $(G,G')$ do not satisfy
 the general finite multiplicity condition \eqref{eqn:PP}
 for most of the cases, 
 see \cite{xKMt}.  
\end{example}

\begin{example}
[Symmetry breaking operators]
\label{ex:NO}
For the symmetric pair $(G,G')=(Sp_n({\mathbb{R}}), GL_n({\mathbb{R}}))$, 
 the general finite multiplicity property \eqref{eqn:PP} fails
 \cite{xKMt}.  
However, 
 if we take $P$ to be the Siegel parabolic subgroup of $G$, 
 and $P'$ to be a maximal parabolic subgroup of $G'$, 
 then one has $\#(P_{\mathbb{C}}'\backslash G_{\mathbb{C}}/P_{\mathbb{C}})<\infty$, 
 and therefore the geometric assumption in Theorem \ref{thm:restQQ}
 is satisfied with $Q=P_{\mathbb{C}}$
 and $Q'=P_{\mathbb{C}}'$, 
 and thus the bounded multiplicity property
 \eqref{eqn:dpsbdd} holds
 for the space of symmetry breaking operators.  
Nishiyama--\O rsted \cite{NO18}
 has constructed explicitly (integral) symmetry breaking operators
 between the corresponding degenerate principal series representations of $G$ and $G'$
  generalizing \cite{KS15}.  
\end{example}

\begin{example}
[Invariant trilinear form]
\label{ex:tri}
For a noncompact simple Lie group $G$, 
 the space of invariant trilinear forms
 $\invHom G {\Pi_1 \otimes \Pi_2 \otimes \Pi_3}{\mathbb{C}}$
 is finite-dimensional
 for {\it{all}} $\Pi_1$, $\Pi_2$, $\Pi_3 \in \operatorname{Irr}(G)$, 
 or equivalently, 
 the pair $(G \times G, \operatorname{diag}G)$
 satisfies the finite multiplicity condition \eqref{eqn:PP}, 
 if and only if ${\mathfrak{g}}$ is isomorphic
 to ${\mathfrak{s o}}(n,1)$
 (\cite{Ksuron}, see also \cite[Cor.~4.2]{xkProg2014}).  
Beyond this case, 
 one may consider the setting 
 where $\Pi_1$, $\Pi_2$, $\Pi_3$ are \lq\lq{small representations}\rq\rq\
 such as degenerate principal series 
representations.  
For instance, 
 if $G=Sp_n({\mathbb{R}})$
 and $P$ is a Siegel parabolic subgroup, 
 then 
 $G_{\mathbb{C}}/P_{\mathbb{C}}\times G_{\mathbb{C}}/P_{\mathbb{C}}$
 is $G_{\mathbb{C}}$-spherical 
 via the diagonal action
 \cite{Li94}, 
 or equivalently, 
 it is $G_U$-strongly visible
 via the diagonal action \cite{xtanaka12}, 
 and therefore one has the bounded property
 \eqref{eqn:tribdd}
 of the space of invariant trilinear forms
 by Theorem \ref{thm:trilinear}
 and the one \eqref{eqn:JLT4102}
 of the tensor product
 by Corollary \ref{cor:fmtensor}.  
Construction of trilinear forms 
 and explicit evaluations
 of spherical vectors
 by the generalized Bernstein--Reznikov integrals
 in these cases
 have been studied
 in Clerc {\it{et.~al.~}} \cite{CKOP11} 
 and Clare \cite{C15}, 
 for instance.  
\end{example}

Finally, 
 we mention a couple
 of examples
 for which the bounded multiplicity property fails
 by explicit computations.  

\begin{example}
[Compact symmetric pairs of rank one]
\label{ex:SU3}
Let us consider the triple $(G,H,G')=(SU(3), U(2), SO(3))$.  
Note that the triple of the complexified Lie algebras
 $({\mathfrak{g}}_{\mathbb{C}}, 
{\mathfrak{h}}_{\mathbb{C}}, 
{\mathfrak{g}}_{\mathbb{C}}')
=
(
{\mathfrak{sl}}_3, 
{\mathfrak{gl}}_2, 
{\mathfrak{so}}_3)
$
 is not in Table \ref{tab:0.1}, 
and thus the bounded multiplicity property
 \eqref{eqn:BBH} of the branching 
 should fail.  
In fact, 
let $\Pi_n$ be the irreducible representations of $G=SU(3)$
 with highest weight $(n,0,-n)$
 in the standard coordinates, 
 and $\pi_n$ the $(2n+1)$-dimensional irreducible representation 
 of $G'=SO(3)$.  
Then $\Pi_n \in \operatorname{Irr}(G)_H$
 and $[\Pi_n|_{G'}:\pi_n]=[\frac n 2]+1$, 
 hence 
$
  \underset{n \in {\mathbb{N}}}\sup\,\, m(\Pi_n|_{G'}) = \infty
$
 showing the failure
 of the bounded multiplicity property \eqref{eqn:BBH}
 for the triple $(G, H, G')$.  
We note that this triple is one of the three exceptional cases
 when $\operatorname{rank}G/H=1$
 indicated in Corollary \ref{cor:rankone}.  
\end{example}

The following example is a reformulation
 of \cite[Sect.~6.3]{mf-korea}. 

\begin{example}
[Restriction of Harish-Chandra's discrete series representations]
\label{ex:sp2C}
Let $(G,H,G')=(SO(5,{\mathbb{C}}),SO(3,2), SO(3,2))$.  
If $\Pi$ is a discrete series representation $\Pi$ for the symmetric space
 $G/H$, 
 then there exists a (Harish-Chandra) discrete series representation $\pi$
 of $G'$
 such that $\pi$ occurs in the restriction $\Pi|_{G'}$
 as discrete spectrum
 of infinite multiplicity, 
 and in particular, 
$[{\Pi}|_{G'}:\pi]=\infty$.  
In fact, 
 this triple $(G,H,G')$ does not appear
 in the classification given in Theorem \ref{thm:cpxlist}.  
\end{example}

\vskip 1pc
\par\noindent
{\bf{$\langle$Acknowledgements$\rangle$}}\enspace
This work was partially supported
 by Grant-in-Aid for Scientific Research (A) (18H03669), 
Japan Society for the Promotion of Science.

The author is greatly indebted to T.\ Kubo and Y.\ Tanaka
 for reading carefully the first draft
 of this article
 and for making useful comments.

\vskip 1pc
\par\noindent
Toshiyuki Kobayashi, 
Graduate School of Mathematical Sciences
and Kavli IPMU (WPI), 
The University of Tokyo.

\vskip 1pc
Added in Proof.  
Recently, 
 Aisenbud and Gourevitch have posted a preprint 
 \lq\lq{Finite multiplicities beyond spherical pairs}\rq\rq\
 in the arXiv (arXiv:2109.00204)
 concerning the individual finite multiplicity property 
 among other results, 
 though not the uniform estimate
 as in this article.  
Similar statements to the implications (iii) $\Rightarrow$ (i) in Example \ref{ex:Qsph2}
 and (v) $\Rightarrow$ (i) in Corollary \ref{cor:fmtensor} in this article
 may be found there 
 in \lq\lq{Corollaries $G$ and $H$}\rq\rq\
 with the same geometric assumption 
 but without the uniform estimate of the multiplicities.  

\begin{thebibliography}{00000}
\bibitem[BK21]{BK21}
Y.\ Benoist, T.\ Kobayashi, 
Tempered homogeneous spaces III, 
J.~Lie Theory {\bf{31}}, 
(2021), 
833--869.  

\bibitem[Bo61]{Bo61}
A.~Borel, 
Sous-groupes commutatifs et torsion des groupes de Lie compacts
connexes.
T{\^o}hoku Math.~J., 
{\bf{13}}, 
(1961), 216--240.  

\bibitem[Br02]{Bou7a9}
N.~Bourbaki, 
{\em Lie {G}roups and {L}ie Algebras. {C}hapters 4--6}.
 Springer, 2002.

\bibitem[B86]{B86}
M.~Brion, 
Quelques propri\'et\'es des espaces
 homogen\`es sph\'eriques, 
 Manuscripta Math.~{\bf{55}}, 
 (1986), 191--198.  

\bibitem[C15]{C15}
P.~Clare, Invariant trilinear forms 
 for spherical degenerate principal series 
 of complex symplectic groups, 
 Internat.~J.~Math., 
{\bf{26}}, (2015), 1550107, 16pp.  

\bibitem[CK{\O}P11]{CKOP11}
J.-L.~Clerc, 
T.~Kobayashi,
B.~{\O}rsted, and M.~Pevzner, 
\textit{Generalized
Bernstein--Reznikov integrals},  
Math. Ann.,
{\bf{349}} 
(2011), 
\href{http://dx.doi.org/10.1007/s00208-010-0516-4}
{pp.~395--431}.

\bibitem[F79]{F79}
J.~Faraut, Distributions sph{\'e}riques sur les espaces hyperboliques, 
J.~Math.~Pures Appl. {\bf{58}} (1979), 369--444.

\bibitem[GS64]{GS64}
I.~M.~Gel'fand and G.~E.~Shilov, 
Generalized functions Vol.~I:
 Properties and operations, 
Translated by Eugene Soletan, 
Academic Press, New York-London, 1964.  



\bibitem[HNOO13]{xhnoo}
X.~He, 
K.~Nishiyama, 
H.~Ochiai, 
Y.~Oshima, 
On orbits in double flag varieties 
 for symmetric pairs, 
 Transform.~Groups, 
 {\bf{18}}, 
(2013), 1091--1136.  

\bibitem[He78]{He78}
S.~Helgason, 
Differential Geometry, 
 Lie Groups, 
 and Symmetric Space, 
Academic Press, 1978.  

\bibitem[He94]{He94}
S.~Helgason, 
Geometric Analysis on Symmetric Spaces, 
Amer.~Math.~Soc., 
1994.  

\bibitem[HT93]{xhowetan}
R.~Howe, E.~Tan, Homogeneous functions on light cones, 
Bull.~Amer.~Math.~Soc., {\bf{28}} (1993), 1--74.  

\bibitem[Ka83]{Ka83}
M.~Kashiwara, 
 Systems of Microdifferential Equations,
 Progr.~Math., 
{\bf{34}}, 
Birkh{\"a}user, 
Boston, 
 1983, xv+159pp.  

\bibitem[KK81]{KaKw81}
M.~Kashiwara, T.~Kawai, 
On holonomic systems 
 of microdifferential equations. III.
Systems with regular singularities, 
Publ.~Res.~Inst.~Math.~Sci.
{\bf{17}}, (1981), 813--979.  

\bibitem[K93]{K93}
T. Kobayashi, 
\textit{The restriction of $A_{\mathfrak{q}}(\lambda)$
            to reductive subgroups},
           {Proc. Japan Acad.},
           {\textbf{69}}
           (1993),
\href{http://projecteuclid.org/euclid.pja/1195511349}
{262--267}.  

\bibitem[K94]{K94Invent}
T.~Kobayashi, 
{\textit{Discrete decomposability of the restriction of
             $A_{\frak q}(\lambda)$
            with respect to reductive subgroups and its applications}}, 
Invent. Math.,
{\bf{117}} 
(1994), 
\href{http://dx.doi.org/10.1007/BF01232239}
{181--205}.  

\bibitem[K95]{Ksuron}
T.~Kobayashi,
\textit{Introduction to harmonic analysis
 on real spherical homogeneous spaces},
 In F.~Sato, editor, 
Proceedings of the 3rd Summer School on Number Theory
\lq\lq{Homogeneous Spaces and Automorphic Forms}\rq\rq\
in Nagano, 1995, pp.~22--41.  

%
\bibitem[K98a]{xkAnn98}
T.~Kobayashi, 
{\textit{Discrete decomposability of the restriction of
             $A_{\frak q}(\lambda)$
            with respect to reductive subgroups {\rm{II}}---micro-local analysis and asymptotic $K$-support}}, 
Ann. of Math., 
{\bf {147}} 
(1998), 
\href{http://dx.doi.org/10.2307/120963}
{pp.~709--729}.  

\bibitem[K98b]{K98b}
T.~Kobayashi, 
{\textit{Discrete decomposability of the restriction of
             $A_{\frak q}(\lambda)$
            with respect to reductive subgroups {\rm{III}}---restriction of Harish-Chandra modules
 and associated varieties}}, 
Invent. Math., {\bf{131}} (1998), 
\href{http://dx.doi.org/10.1007/s002220050203}
{229--256}.  


\bibitem[K05]{xrims40}
T. Kobayashi, 
\textit{Multiplicity-free representations and visible actions
on complex manifolds}, 
Publ. Res. Inst. Math. Sci. 
{\textbf{41}} (2005),
\href{http://dx.doi.org/10.2977/prims/1145475221}
{pp.~497--549},
special issue commemorating the fortieth anniversary of the founding of RIMS.
%
\bibitem[K07a]{K07}
T.~Kobayashi, 
Visible actions on symmetric spaces, 
 Transform.~Groups, {\bf{12}}, (2007), 671--694.  
%
\bibitem[K07b]{K07b}
T. Kobayashi,
\textit{A generalized Cartan decomposition
 for the double coset space 
$(U(n_1)\times U(n_2)\times U(n_3))\backslash U(n)/(U(p)\times U(q))$}, 
J. Math. Soc. Japan \textbf{59} (2007), 
\href{http://dx.doi.org/10.2969/jmsj/05930669}
{pp.~669--691}.  

\bibitem[K08]{mf-korea}
T. Kobayashi, 
\textit{Multiplicity-free theorems of the restrictions of unitary 
highest weight modules with respect to reductive symmetric pairs}, 
Progr. Math., 
\textbf{255}, 
\href{http://dx.doi.org/10.1007/978-0-8176-4646-2_3}
{45--109}, 
\href{http://dx.doi.org/10.1007/b139076}{45--109}. 
Birkh\"auser, 2008. 
%
\bibitem[K14]{xkProg2014}
T. Kobayashi,
\textit{Shintani functions, real spherical manifolds, and symmetry breaking operators}, 
Dev.\ Math., 
{\bf{37}}, 
(2014), 
\href{http://dx.doi.org/10.1007/978-3-319-09934-7_5}
{127--159}.  
Springer.
%
\bibitem[K15]{xKVogan2015}
T.~Kobayashi, 
\textit{A program for branching problems in the representation 
theory of real reductive groups}, 
In: Representations of Reductive 
Groups---In Honor of the 60th Birthday of David A.\ Vogan,
Jr., (eds.\ M.\ Nevins and P.\ E.\ Trapa),
Progr. Math., 
{\bf{312}}, 
Birkh{\"a}user, 
2015,
\href{http://dx.doi.org/10.1007/978-3-319-23443-4_10}
{277--322}.
%
\bibitem[K19a]{K19a}
T.~Kobayashi, 
Recent advances in branching laws of representations, 
Sugaku {\bf{71}} (2019), 
388--416, 
Math.~Soc.~Japan;
English translation to appear in Sugaku Exposition, 
 Amer.~Math.~Soc.  
%
\bibitem[K19b]{K19b}
T.~Kobayashi, 
Admissible restrictions of irreducible representations
 of reductive Lie groups:
Symplectic geometry and discrete decomposability, 
 Pure and Applied Mathematics Quarterly, 
 in special issue in memory of Bertram Kostant
 (in press).  
Available also at arXiv:1907.12964.  
%
\bibitem[K21]{K21}
T.~Kobayashi, 
Branching laws of unitary representations
 associated to minimal elliptic orbits 
 for indefinite orthogonal group $O(p,q)$, 
Adv.~Math., {\bf{388}}, (2021), 
Paper No. 107862, 38pp.  
%
\bibitem[KM14]{xKMt}
T. Kobayashi, T. Matsuki, 
\textit{Classification of finite-multiplicity symmetric pairs},
Transform. Groups, 
 {\bf{19}} 
(2014), 
\href{http://dx.doi.org/10.1007/s00031-014-9265-x }
{457--493}, 
Special issue in honor of Dynkin
 for his 90th birthday. 
%
\bibitem[K{\O}02]{KO02}
T.~Kobayashi, B.~{\O}rsted, 
Analysis on the minimal representation of $O(p,q)$, II, 
Branching laws, 
Adv.~Math., 
{\bf{180}}
 (2003), 513--550.  
%
\bibitem[K{\O}P11]{KOP11}
T.~Kobayashi, B.~{\O}rsted, M.~Pevzner, 
Geometric analysis on small unitary representations of $GL(N,{\mathbb{R}})$, 
 J.~Funct.~Anal., {\bf{260}}, 
(2011), 1682--1720.  

\bibitem[KO13]{xktoshima}
T.~Kobayashi, T.~Oshima, 
\textit{Finite multiplicity theorems for induction and restriction}, 
Adv. Math., \textbf{248} (2013), 
\href{http://dx.doi.org/10.1016/j.jfa.2010.12.008}
{921--944}. 
%
\bibitem[KO12]{KO12}
T.~Kobayashi, Y.~Oshima, 
\textit{Classification of discretely decomposable $A_{\mathfrak{q}}(\lambda)$
 with respect to reductive symmetric pairs}, 
 Adv.~Math., 
{\bf{231}}, 
(2012), 
2013--2047.  

\bibitem[KO15]{KO15}
T.~Kobayashi, Y.~Oshima, 
\textit{Classification of symmetric pairs
 with discretely decomposable restrictions
of $({\mathfrak{g}}, K)$-modules}, 
 J.~Reine Angew.~Math.
{\bf{703}}, 
(2015), 
201--223.  

\bibitem[KS15]{KS15}
T. Kobayashi, B. Speh,
Symmetry Breaking
 for Representations of Rank One
 Orthogonal Groups, 
(2015), 
 Mem.\ Amer.\ Math.\ Soc. 
 {\bf{238}} 
\href{http://dx.doi.org/10.1090/memo/1126}
{no.1126}, 
 118 pages.  
%
\bibitem[KS18]{xksbonvec}
T.~Kobayashi, B.~Speh, 
Symmetry Breaking for Representations of Rank One Orthogonal Groups, 
Part II, 
\href{https://arxiv.org/abs/1801.00158}
Lecture Notes in Math., {\bf{2234}} 
Springer, 2018.  
xv$+$342 pages.  
%
\bibitem[Kr76]{xkramer}
M.~Kr{\"a}mer, 
\textit{Multiplicity free subgroups
 of compact connected Lie groups},
 Arch. Math. (Basel)
 {\bf{27}}, (1976), 28--36.  
%

\bibitem[Li94]{Li94}
P.~Littelmann, 
On spherical double cones, 
J.~Algebra {\bf{166}}, (1994), 142--157.  

\bibitem[MWZ99]{MWZ99}
P.~Magyar, J.~Weyman, A.~Zelevinsky, 
Multiple flag varieties
 of finite type, 
Adv.~Math., 
{\bf{141}}, (1999), 97--118.  

\bibitem[Mt15]{Mt15}
T.~Matsuki, 
 Orthogonal multiple flag varieties
 of finite type I:
Odd degree case, 
J.~Algebra, {\bf{425}},
(2015), 450--523.  

\bibitem[MO15]{MO15}
J.~M{\"o}llers, Y.~Oshima, 
Restriction of most degenerate representations of $O(1,N)$
 with respect to symmetric pairs. 
J.~Math.~Sci.~Univ.~Tokyo {\bf{22}}, (2015), 279--338.

\bibitem[N{\O}18]{NO18}
K.~Nishiyama, B.~{\O}rsted, 
Real double flag varieties 
 for the symplectic group, 
 J.~Funct.~Anal., {\bf{274}}, (2018), 573--604.  

\bibitem[{\O}S19]{OS19}
B.~{\O}rsted, B.~Speh,
Branching laws for discrete series
 of some affine symmetric spaces, 
 preprint, 
 arXiv:1907.07544.  

\bibitem[O88]{xoshima88b}
T.~Oshima, 
A realization of semisimple symmetric spaces
 and construction of boundary value maps.  
Advanced Studies in Pure Mathematics, 
{\bf{14}}, 
(1988), 
603--650. 

\bibitem[OS84]{OS84}
T.~Oshima, J. ~Sekiguchi, 
 The restricted root systems 
 of a semisimple symmetric pair, 
 Advanced studies 
 in Pure Math., 
{\bf{4}}, 
 (1984), 
433--497.  

\bibitem[Sa60]{Sa60}
I.~Satake, 
On representations and compactifications
of symmetric Riemannian spaces, 
Ann.~Math., {\bf{71}}, (1960), 77--110.  

\bibitem[S59]{S59}
M.~Sato, 
Theory of hyperfunctions I, II, 
J.~Fac~Sci.~Univ.~Tokyo, 
Sect.~I, {\bf{8}}, 
(1959), 133--193;
ibid, (1960), 387--431.  


\bibitem[St83]{S83}
R.~S.~Strichartz, 
Analysis of the Laplacian on the complete Riemannian manifold.
J.~Funct.~Anal., {\bf{52}}, (1983), 48--79.  

\bibitem[SZ12]{xsunzhu}
B.~Sun, C.-B.~Zhu,
Multiplicity one theorems:
the Archimedean case,
Ann.~of Math., 
{\bf{175}},
(2012), 
pp.~23--44.  

\bibitem[Tn12]{xtanaka12}
Y.~Tanaka, 
 Classification of visible actions on flag varieties, 
 Proc.~Japan Acad.~Ser.~A Math.~Sci., 
{\bf{88}}, 
(2012), 
91--96.  

\bibitem[Tn21]{xtanaka}
Y.~Tanaka, 
Visible actions of compact Lie groups on complex spherical varieties, 
to appear in J.~Diff.~Geom.  

\bibitem[Tu19]{xtauchi}
T.~Tauchi, 
A generalization of the Kobayashi--Oshima
 uniformly bounded multiplicity theorem, 
to appear in Internat.~J.~Math.  
Available also at arXiv:2108.02139
 (Ph.D.~thesis, Chap.~4, Univ.~Tokyo, 2019).  

\bibitem[V86]{V86}
{\'E}.~B.~Vinberg, 
 Complexity of action of reductive groups,
 Func.~Anal.~Appl. {\bf{20}}, 
 (1986), 1--11.  

\bibitem[VK78]{VK78}
{\'E}.~B.~Vinberg and B.~N.~ Kimel'fel'd, 
Homogeneous domains on flag manifolds and spherical subgroups
 of semisimple Lie groups, 
Functional Anal.~Appl., 
{\bf{12}}, 
(1978), 
168--174.  


\bibitem[Wal92]{WaI}
N. R. Wallach,
Real reductive groups. I, II, 
Pure Appl.\ Math. 
{\bf{132}} 
 Academic Press, Inc., Boston, MA, 1988;
{\bf{132}}-II, ibid, 1992. 

\bibitem[War72]{War72}
G.~Warner, 
Harmonic Analysis on Semi-Simple Lie Groups I, 
Grundlehren Math.~Wiss., 
vol. {\bf{188}}, 
 Springer-Verlag, 
New York, 
1972.  

\bibitem[Wo69]{W69}
J.~A.~Wolf, 
 The action of a real semisimple group on a complex flag manifold. I.  
Orbit structure and holomorphic arc components.  
Bull.~Amer.~Math.~Soc.~{\bf{75}}, 
(1969), 1121--1237.  

\end{thebibliography}
\end{document}